\documentclass[10pt]{amsart}
\usepackage{amscd, amsfonts, amsthm, amsgen, amsmath, amssymb ,verbatim, enumerate, textcomp}
\usepackage{hyperref} 
\usepackage[cmtip, all]{xy}
\usepackage{graphicx}
\usepackage[margin=1in]{geometry}

\usepackage{amscd}
\usepackage{xypic}
\usepackage{amsmath, amssymb}

\newtheorem{theorem}{Theorem}[section]

\newtheorem{lemma}[theorem]{Lemma}
\newtheorem{corollary}[theorem]{Corollary}
\newtheorem{proposition}[theorem]{Proposition}
\theoremstyle{definition}
\newtheorem{definition}[theorem]{Definition}
\theoremstyle{notation}

\newtheorem{construction}[theorem]{Construction}
\newtheorem{conventions}[theorem]{Conventions}

\theoremstyle{remark}
\newtheorem{remark}[theorem]{Remark}
\newtheorem{remarks}[theorem]{Remarks}

\numberwithin{equation}{section}
\usepackage{amscd}
\usepackage{xypic}
\usepackage{amsmath, amssymb}

\numberwithin{equation}{subsection}
\newcommand{\be}%
  {\protect\setcounter{equation}{\value{subsubsection}}}
  \newcommand{\ee}%
   {\protect\setcounter{subsubsection}{\value{equation}}}

  {\protect\setcounter{subsubsection}{\value{equation}}}

\def \rmA{\rm A}

\def \BG{\rm BG}
\def \EG1{\rm EG}
\def \BN(T){\rm BN_{G}(T)}
\def \rmB{\rm B}
\def \BH{\rm BH}
\def \BK{\rm BK}
\def \BGL{\rm BGL}
\def \Br{\rm Br}
\def \BT{\rm BT}

\def \rmC{\rm C}

\def \Cl{\mathbb C}

\def \colimm{\underset {m \rightarrow \infty}  {\hbox {lim}}}

\def \colimalpha{\underset {\alpha}  {\hbox {colim}}}

\def \colimK.{\underset {\underset K^.  \rightarrow}  {\hbox {lim}}}

\def \colimU.{\underset {\underset U_.  \rightarrow}  {\hbox {lim}}}

\def \compl{\, \, {\widehat {}}}
\def \rmc{\rm c}

\def \rmD{\rm D}

\def \DS1X{\rm {DS^1X}}
\def \rmD{\rm D}

\def \E{{\rm E}}
\def \cE{\mathcal E}

\def \EG1{E{(G \times {\mathbb C}^*)}{\underset {G\times {\mathbb C}^*} 
\times}}

\def \EZ(s)1{E{(Z(s) \times {\mathbb C}^*)}{\underset {(Z(s)\times {\mathbb
C}^*)}  \times}}

\newcommand{\eps}{ \, {\boldsymbol\varepsilon} \,}

\def \EM(u){EM(u){\underset {M(u)}  \times}}
\def \EM(us){EM(u,s){\underset {M(u, s)}  \times}}

\def \EG{\rm EG}
\def \rmE{\rm E}
\def \EH{\rm EH}
\def \EK{\rm EK}
\def \ET{\rm ET}
\def \EGL{\rm EGL}

\def \rmf{\rm f}
\def \rmF{\rm F}

\def \group{\rm G}

\def \rmG{\rm G}
\def \GL{\rm {GL}}
\def \rmGL{\rm {GL}}

\def \rmg{\rm g}

\def\holimD{\mathop{\textrm{holim}}\limits_{\Delta }}

\def \hocolimD{\underset \Delta  {\hbox {hocolim}}}

\def \holimn {\underset {\infty \leftarrow n}  {\hbox {holim}}}

\def \H{\mathbb H}
\def \rmH{\rm H}
\def \HRR{\rm {HRR}}

\def \Hom{\underline {Hom}}
\def \Hom{{\mathcal H}om}

\def \rmh{\rm h}
\def \h{\it h}

\def \invlim1{\underset {\infty \leftarrow q}  {\hbox {lim}}^1}

\def \rmI{\rm I}

\def \rmj{\rm j}

\def \rmk{\rm k}
\def \rmK{\rm K}
\def \k{\it k}

\def \L3{\Lambda \times \Lambda \times \Lambda}
\def \L2{\Lambda \times \Lambda}

\def \lim{\underset \leftarrow  {\hbox {lim}}}

\def \longright2arrow{{\overset \longrightarrow  {\overset {} 
\longrightarrow}}}

\def \L{L\times \Cl ^*}

\def \rmL{\rm L}

\def \rmM{\rm M}

\def \Map{\underline {Map}}
\def \Map{{\mathcal M}ap}

\def \cN{\mathcal N}

\def \N(T){\rm {N_{G}(T)}}
\def \rmN{\rm N}
\def \NT{\rm N_{G}(T)}
\def \Nis{\rm {Nis}}

\def \O{{\mathcal O}}

\def \rmp{\rm p}

\def \rmP{\rm P}

\def \Spt{\rm {Spt}}

\def \cP{\mathcal P}

\def \cQ{\mathcal Q}

\def \ra{\rightarrow}
\def \Ra{\Rightarrow}
\def \RG^{R(G)^{\hat {}}\ }

\def \res{respectively}
\def \rmS{\rm S}

\def \RHom{{{\mathcal R}{\mathcal H}om}}

\def \rmR{\rm R}

\def \Sm{\rm {Sm}}

\def \Speck{{\rm {Spec}}\, {\it k}}

\def\Spt{\rm {\bf Spt}}

\def \Sph{\rm {Sph}}

\def \SH{{\mathcal S}{\mathcal H}}
\def\Spt{\rm {\bf Spt}}
\def\Spc{\rm {\bf Spc}}

\def \mbS{\mathbb S}
\def \rmS{\rm S}

\def \topGcoh*{^{top, *} _{G}}
\def \topGho*{ _{top,*} ^{G}}

\def \T{{\mathbf T}}

\def \tr{\it {tr}}
\def \rmT{\rm T}
\def \Th{\rm Th}

\def \rmU{\rm U}
\def \U{\mathcal U}

\def \rmu{\rm u}

\def \rmV{\rm V}

\def \W{\mathbf W}
\def \rmW{\rm W}
\def \rmw{\rm w}

\def \rmX{\rm X}

\def \cX{\mathcal X}
\def \X{\mathcal X}

\def \x{\it x}
\def \rmx{\rm x}

\def \cY{\mathcal Y}
\def \Y{\mathcal Y}
\def \rmY{\rm Y}
\def \rmy{\rm y}

\def \Z(s){Z(s) \times {\mathbb C}^*}
\def \Z{\mathcal Z}
\def \bZ{\mathbb Z}

\def \rmZ{\rm Z}
\def \z{\it z}

\begin{document}

\title{Additivity and Double Coset formulae for the Motivic and \'Etale Becker-Gottlieb transfer}

\author{Roy Joshua}
\address{Department of Mathematics, Ohio State University, Columbus, Ohio,
43210, USA.}

\email{joshua.1@math.osu.edu}
\author{Pablo Pelaez}
\address{Instituto de Matem\'aticas, Ciudad Universitaria, UNAM, DF 04510, M\'exico.}
\email{pablo.pelaez@im.unam.mx}
\thanks{\indent Both authors would like to thank the Isaac Newton Institute for Mathematical Sciences, Cambridge, for support and hospitality during the programme {\it K-Theory, Algebraic Cycles and Motivic Homotopy Theory} where part of the work on this paper was undertaken. 
This work was supported in part by EPSRC grant no EP/R014604/1. The first author would also like to thank the Simons Foundation for support.}
\begin{abstract} In this paper, which is a continuation of earlier work by the first author and
	Gunnar Carlsson, one of the first results we establish is the additivity of the motivic Becker-Gottlieb transfer, as well as their \'etale realizations.
	This extends the additivity results the authors already established for the corresponding traces. 
	We then apply this to derive several important consequences: for example, 
	in addition to obtaining the analogues of various double coset formulae known in the classical setting of algebraic topology, we also obtain applications
	to Brauer groups of homogeneous spaces associated to reductive groups over separably closed fields. We also consider 
        the relationship between the transfer on schemes provided with a compatible
	action by a $1$-parameter subgroup and the transfer associated to the fixed point scheme of the $1$-parameter subgroup.
\end{abstract}
\maketitle

\centerline{\bf Table of contents}
\vskip .2cm 
1. Introduction
\vskip .2cm
2. The $\rmG$-equivariant pre-transfer
\vskip .2cm
3. Additivity of the pre-transfer
\vskip .2cm
4. Mayer-Vietoris and Additivity of the Transfer
\vskip .2cm
5. The rigidity property 
\vskip .2cm
6. More on Nisnevich neighborhoods
\vskip .2cm
7. Applications of the Additivity and Multiplicativity of the Transfer
\vskip .2cm
\markboth{ Roy Joshua and Pablo Pelaez}{Additivity and  Double Coset Formulae for the Motivic and \'Etale Becker-Gottlieb transfer}
\input xypic

\vfill \eject
\section{\bf Introduction}
\label{intro}
This paper is a continuation of earlier work where the first author and Gunnar Carlsson
set up motivic and \'etale variants of the classical Becker-Gottlieb transfer. If one recalls, the power and utility of the classical Becker-Gottlieb transfer 
stems from the fact it provided a convenient mechanism to obtain splittings to certain  maps
in the stable homotopy category, making use of the Euler-characteristic of the fiber.
The most notable example of this is the calculation of the Euler-characteristic of $\rmG/\NT$, where $\rmG$ is a compact Lie group
and $\NT$ is the normalizer of a maximal torus in $\rmG$. A 
particularly nice way to prove this is to first observe that the above Euler characteristic is closely related to the
trace of the suspension spectrum of $\rmG/\NT$, and then show that the trace is additive. 
\vskip .1cm
In \cite{JP23}, the authors exploited this idea to prove a corresponding result in the motivic context 
for any split linear algebraic group $\rmG$, and $\NT$ the normalizer of a split maximal torus in $\rmG$. A key point here is to prove
that the trace (see Definition ~\ref{pretransfer.3} below) is {\it additive}. Moreover, as the
trace is defined in terms of the pre-transfer, the additivity of the trace may be deduced from the additivity of the pre-transfer. 
Thus already several of the key splitting results make use of the additivity of the pre-transfer.

\vskip .1cm
{\it The main goal of this paper is to prove an additivity theorem for the corresponding Becker-Gottlieb transfer and to consider several applications of
 such a theorem}. We establish a number of applications, such as various double coset formulae, and others in the motivic and \'etale framework.
 In fact all of section 7 is devoted to a discussion of various applications, beginning with Theorem ~\ref{torus.act} through Theorem ~\ref{double.coset.1} and Corollary ~\ref{double.coset.5}.
In addition to obtaining motivic variants of several of the classical applications, we also obtain
 applications to the Brauer groups of certain homogeneous spaces, algebraic stacks and GIT-quotients: see Corollary ~\ref{double.coset.4} and \cite{DIJ23}.
\vskip .2cm
Here is a quick overview of the paper. Section 2 is devoted to a quick review of the equivariant pre-transfer and section 3 establishes its
additivity property, with much
of the details worked out already in \cite[2.1]{JP23}. The Mayer-Vietoris and additivity properties of the transfer are established in section 4. 
The stronger results on additivity for the transfer hold only on generalized cohomology theories defined with respect to
spectra that have {\it the rigidity property}: section 5 is devoted to a discussion on this property. Section 6 discusses Nisnevich neighborhoods
and section 7 is devoted to various applications of the additivity property of the transfer. Moreover, we have written this paper in such a manner
 that it is more or less self-contained and independent of the details of the construction of the transfer discussed elsewhere.
\vskip .2cm
Throughout the paper, we will adopt the terminology rom \cite{CJ23-T1}. We will work with the category, $\Sm_{\k}$, of smooth schemes of finite type over a fixed base field $\k$. 
 Let $\rmG$ denote a linear algebraic group defined 
over $\k$. $\Spc(\k)$ ($\Spc_*(\k)$) will denote the category of simplicial presheaves 
(pointed simplicial presheaves, \res) on $\Sm_{\k}$. $\Spc^{\rmG}(\k)$ ($\Spc_*^{\rmG}(\k)$) will denote 
the corresponding category of $\rmG$-equivariant simplicial presheaves.
$\Spt^{\rmG}(\k_{mot})$, ${\widetilde {\Spt}}^{\rmG}(\k_{mot})$ and ${\widetilde {\Spt}}(\k_{mot})$ ($\Spt^{\rmG}(\k_{et})$, ${\widetilde {\Spt}}^{\rmG}(\k_{et})$ and ${\widetilde {\Spt}}(\k_{et})$)
 will denote the corresponding category of spectra defined on the Nisnevich site (on the \'etale site, \res) as discussed in \cite[section 4]{CJ23-T1}. One may recall that
 $\Spt^{\rmG}(\k_{mot})$ ($\Spt^{\rmG}(\k_{et})$) denotes the category of $\rmG$-equivariant motivic spectra (\'etale spectra) and 
 that ${\widetilde {\Spt}}^{\rmG}(\k_{mot})$, ${\widetilde {\Spt}}(\k_{mot})$ are categories of motivic spectra intermediate between 
 $\Spt^{\rmG}(\k_{mot})$ and the usual category of motivic spectra $\Spt(\k_{mot})$.  Similarly ${\widetilde {\Spt}}^{\rmG}(\k_{\rm et})$, ${\widetilde {\Spt}}(\k_{\rm et})$ are categories of \'etale spectra intermediate between 
 $\Spt^{\rmG}(\k_{\rm et})$ and the usual category of \'etale spectra $\Spt(\k_{et})$. 
 We will primarily work with the category ${\widetilde \Spt}^{\rmG}(\k_{mot})$.
\section{\bf The $\rmG$-equivariant pre-transfer}
The following is a summary of the discussion in \cite[2.1]{JP23}, which itself is a variant of the discussion in \cite[Chapter III]{LMS}.
\begin{definition} 
\label{co-mod.st}
(Co-module structures)
Let $\rmC$ denote an unpointed simplicial presheaf in $\Spc(\k)$, and let $\rmC_+$ denote the associated pointed simplicial presheaf. Then the diagonal map $\Delta : \rmC_+ \ra \rmC_+ \wedge \rmC_+$ together with
the augmentation $\epsilon: \rmC_+ \ra \rmS^0$ defines the structure of an associative co-algebra of simplicial presheaves on $\rmC_+$.
A pointed simplicial presheaf $\rmP$ in $\Spc_*(\k)$ will be called a right $\rmC_+$-co-module, if it comes equipped with maps $\Delta: \rmP \ra \rmP \wedge \rmC_+$
so that  the diagrams:
\begin{equation}
 \label{diagonal.1}
  \xymatrix{{ \rmP} \ar@<1ex>[r]^{\Delta} \ar@<1ex>[d]^{\Delta} & {\rmP \wedge \rmC_+} \ar@<1ex>[d]^{id \wedge \Delta}\\  
             {\rmP \wedge \rmC_+} \ar@<1ex>[r]^{\Delta \wedge id} & {\rmP \wedge \rmC_+ \wedge \rmC_+ }}
  \quad \mbox{ and } \xymatrix{ {\rmP} \ar@<1ex>[d]^{\Delta} \ar@<1ex>[dr]^{id} \\           
                                {\rmP \wedge \rmC_+} \ar@<1ex>[r]^{id \wedge \epsilon} & {\rmP \wedge \rmS^0}}
\end{equation}
commute. {\it The most common choice of $\rmP$ is with $\rmP = \rmC_+$} and with the obvious diagonal map $\Delta: \rmC_+ \ra \rmC_+ \wedge \rmC_+$ as 
providing the co-module structure. It needs to be pointed out that the reason we are constructing the pre-transfer in this generality (see the definition below) is so that we are able to obtain
strong additivity results as in Theorem ~\ref{additivity.tr.0}. 
\end{definition}
\vskip .2cm
We define the {\it $\rmG$-equivariant pre-transfer} as follows: see \cite[Definition 2.2]{JP23} for the 
definition of the (non-equivariant) pre-transfer.
\begin{definition}
 	  \label{pretransfer.3}
 	Assume that the pointed simplicial presheaf $\rmP$ belongs to $\Spc_*^{\rmG}(\k)$ for a given linear algebraic group $\rmG$ so 
        that (i) $\mbS^{\group} \wedge \rmP$ is dualizable in ${\widetilde {\Spt}}^{\rmG}(\k_{\rm mot})$ and (ii) is provided
 	with a $\rmG$-equivariant pointed map $\rmf: \rmP \ra \rmP$. Assume further that $\rmC$ is an unpointed simplicial presheaf in $\Spc^{\rmG}(\k)$ so that $\rmP$ is a
 	right $\rmC_+$-co-module. 
 	Then the {\it $\rmG$-equivariant pre-transfer with respect to $\rmC$} is defined to be a map
  \be	\begin{equation}
 	tr(f)'^{\rmG}: \mbS^{\rmG} \ra \mbS^{\rmG} \wedge \rmC_+,
   \end{equation} \ee
 	which is the
 	composition of the following maps. Let $e: \rmD(\mbS^{\rmG }\wedge \rmP )
 	\wedge \mbS^{\rmG} \wedge \rmP \ra \mbS^{\rmG}$ denote the evaluation map. Observe that, this map being natural, is automatically $\rmG$-equivariant.
 	We take the dual of this map to obtain:
 	\be \begin{equation}
 	  \label{coeval}
 	c=D(e): \mbS^{\rmG} \simeq \rmD(\mbS^{\rmG}) \ra \rmD(\rmD(\mbS^{\rmG} \wedge \rmP ) \wedge (\mbS^{\rmG} \wedge \rmP) ) {\overset {\simeq} \leftarrow} \rmD(\mbS^{\rmG} \wedge \rmP)\wedge 
 	(\mbS^{\rmG} \wedge \rmP ) {\overset {\tau} \ra}  (\mbS^{\rmG} \wedge \rmP ) \wedge \rmD(\mbS^{\rmG} \wedge \rmP).
 	\end{equation} \ee
 	
 	Here $\tau$ denotes the obvious flip map interchanging the two factors and $c$ {\it denotes the co-evaluation}. The reason that taking
 	the double dual yields the same object up to weak-equivalence
 	is because we are in fact taking the dual in ${\widetilde {\Spt}}^{\rmG}(\k_{mot})$, which with a suitable stable model structure was shown 
 	to be Quillen equivalent to $\Spt(\k_{mot})$ with its (usual) stable model structure: see \cite[Proposition 6.2]{CJ23-T1}.
 	  Observe that all the maps that {\it go in the left-direction are weak-equivalences}. All the maps involved in the definition of the co-evaluation map are  {natural maps and therefore
 	automatically $\rmG$-equivariant}.
 	\vskip .2cm
 	To complete the definition of the pre-transfer $tr(f)'^{\rmG}: \mbS^{\rmG} \ra \mbS^{\rmG} \wedge \rmC_+$, one simply composes
 	the co-evaluation map with the following composite map:
 	\be \begin{multline}
 	\begin{split}
 	(\mbS^{\rmG} \wedge \rmP )  \wedge \rmD(\mbS^{\rmG} \wedge \rmP) {\overset {\tau} \ra} \rmD(\mbS^{\rmG} \wedge \rmP)\wedge (\mbS^{\rmG} \wedge \rmP ){\overset {id \wedge f} \ra } \rmD(\mbS^{\rmG} \wedge \rmP)\wedge (\mbS^{\rmG} \wedge \rmP ) \\
 	 {\overset {id \wedge \Delta} \ra }\rmD(\mbS^{\rmG} \wedge \rmP)\wedge (\mbS^{\rmG} \wedge \rmP ) \wedge (\mbS^{\rmG} \wedge \rmC_+ )
 	{\overset {e \wedge id} \ra } \mbS^{\rmG} \wedge (\mbS^{\rmG} \wedge \rmC_+ ) \simeq \mbS^{\rmG} \wedge \rmC_+.
 	\end{split}
 	\end{multline} \ee
 	\vskip .1cm
 	The corresponding {\it $\rmG$-equivariant trace}, $\tau^{\rmG}(f)$ is defined as the composition of the above pre-transfer $tr(f)'^{\rmG}$ with the projection sending 
 	$\rmC_+$ to $\rmS^0_+$. 
 	\vskip .1cm
 	When  $\rmf=id_{\rmP}$, the pre-transfer (trace) will be denoted $tr_{\rmP}'^{{\rmG}}$ ($\tau^{\rmG}_{\rmP}$, \res), and when
 	$\rmP =\rmC_+ $ and $\rmf=id_{\rmP}$, the pre-transfer (trace) will be denoted $tr_{\rmC_+}'^{{\rmG}}$ ($\tau^{\rmG}_{\rmC_+}$, \res).
 \end{definition}
 \begin{remark}
  \label{key.ident.trace} Observe that now the trace map identifies with the following composite maps:
  \[ \tau_{\rmC_+}^{\rmG}: \mbS^{\rmG} {\overset c \ra} \mbS^{\rmG} \wedge \rmC_+ \wedge \rmD(\mbS^{\rmG} \wedge \rmC_+) {\overset {\tau} \ra} \rmD(\mbS^{\rmG} \wedge \rmC_+) \wedge \mbS^{\rmG} \wedge \rmC_+ {\overset e \ra} \mbS^{\rmG} \mbox{ and }\]
 \end{remark}
 \begin{definition}
 \label{generalized.traces}
 If $\cE^{\rmG}$ denotes {\it any commutative ring spectrum} in $\Spt^{\rmG}(\k_{\rm mot})$, one may replace the sphere spectrum
 $\mbS^{\rmG}$ everywhere by $\cE^{\rmG}$ and define the pre-transfer and trace similarly, provided the
 unpointed simplicial presheaf $\rmC$ is such that $\cE^{\group} \wedge \rmC_+$ is dualizable in ${\widetilde {\Spt}}^{\rmG}(\k_{mot}, {\cE^{\rmG}})$ and is provided
 	with a $\rmG$-equivariant pointed map $\rmf: \rmC_+ \ra \rmC_+$.  The corresponding $\rmG$-equivariant pre-transfer will be
 denoted $\tr(f)'_{\cE^{\rmG}}$ in general and when $f=id_{C_+}$, by $\tr'_{\cE^{\rmG}}$ or $\tr'_{\rmC_+, \cE^{\rmG}}$. 
 \end{definition}
 

\vskip .2cm
 \section{\bf The additivity of the pre-transfer} \index{pre-transfer: additivity} \index{trace: additivity}
 Let 
 \be \begin{equation}
      \label{additivity.0}
 \rmU_+ {\overset {\rmj_+} \ra} \rmX_+ {\overset {\rmk_+} \ra} \rmX/\rmU = Cone(\rmj) \ra \rmS^1 \wedge \rmU_+
 \end{equation} \ee
 denote a cofiber sequence where both $\rmU$ and $\rmX$ are unpointed simplicial presheaves in $\Spc(\k)$, with $j_+$ a cofibration. Now a key point to observe is that all of $\rmU_+$, $\rmX_+$ and $\rmX/\rmU$ have the
 structure of right $\rmX_+$-co-modules. The right $\rmX_+$-co-module structure on $\rmX_+$ is given by the diagonal map $\Delta: \rmX_+ \ra \rmX_+ \wedge \rmX_+$,
 while the right $\rmX_+$-co-module structure on $\rmU_+$ is given by the map $\Delta: \rmU _+ {\overset {\Delta} \ra} \rmU_+ \wedge \rmU_+ {\overset {id \wedge j_+} \ra} \rmU_+ \wedge \rmX_+$,
 where $j: \rmU \ra \rmX$ is the given map. The right $\rmX_+$-co-module structure on $\rmX/\rmU$ is obtained in view of the commutative square
 \vskip .1cm
 \be \begin{equation}
      \xymatrix{{\rmU} \ar@<1ex>[r]^{(id \times j)\circ \Delta} \ar@<1ex>[d]^j & {\rmU \times X} \ar@<1ex>[d]^{j \times id}\\
                {\rmX} \ar@<1ex>[r]^{\Delta} & {\rmX \times \rmX}}
 \end{equation} \ee
 which provides the map 
 \be \begin{equation}
   \label{comod.XU}
 \rmX/\rmU \ra (\rmX \times \rmX)/(\rmU \times \rmX) \cong (\rmX/\rmU) \wedge \rmX_+.
\end{equation} \ee
If one assumes that all the simplicial presheaves in ~\eqref{additivity.0} are provided with the action by a linear algebraic group $\rmG$
 so that all the maps in ~\eqref{additivity.0} are also $\rmG$-equivariant, one may see also that the all the above co-module structures
 are compatible with the actions by $\rmG$.
 
 \vskip .1cm 
We begin with the following results, which are variants of \cite[Theorem 7.10, Chapter III and Theorem 2.9, Chapter IV]{LMS} adapted to our contexts.
 \begin{theorem} 
  \label{additivity.tr.0}
  Let $\rmU_+ {\overset {\rmj_+} \ra} \rmX {\overset {\rmk_+} \ra} \rmX/\rmU = Cone(\rmj) \ra \rmS^1 \wedge \rmU_+$ denote a cofiber sequence 
  as in ~\eqref{additivity.0} where all the above simplicial presheaves  are provided with actions by a linear algebraic group $\rmG$
  so that all the above maps are $\rmG$-equivariant. Let $f: \rmU_+ \ra \rm U_+$, $g: \rmX_+ \ra \rmX_+$ denote two $\rmG$-equivariant maps so that the diagram
  \vskip .1cm
  \xymatrix{{\rmU_+} \ar@<1ex>[r]^{j_+} \ar@<1ex>[d]_f & {\rmX_+} \ar@<1ex>[d]_g \\
            {\rmU_+} \ar@<1ex>[r]^{j_+} & {\rmX_+} } 
 \vskip .1cm \noindent
 commutes. Let  $h: \rmX/\rmU \ra \rmX/\rmU$ denote the corresponding induced map. Then, with the right $\rmX$-co-module structures discussed above,
 one obtains the following commutative diagram:
 \be \begin{equation}
      \label{additivity.tr.diagram.0}
\xymatrix{{\rmU_+} \ar@<1ex>[r]^{j_+} \ar@<1ex>[d]_{\Delta} & {\rmX_+} \ar@<1ex>[r]^{k_+} \ar@<1ex>[d]_{\Delta} & {\rmX/\rmU} \ar@<1ex>[r]^l \ar@<1ex>[d]_{\Delta} & {\rmS^1\wedge \rmU_+} \ar@<1ex>[d]_{\rmS^1 \Delta}\\
            {\rmU_+\wedge \rmX_+} \ar@<1ex>[r]^{j_+ \wedge id} & {\rmX_+ \wedge \rmX_+} \ar@<1ex>[r]^{k_+ \wedge id} & {(\rmX/\rmU) \wedge \rmX_+} \ar@<1ex>[r]^{l \wedge id}  & {\rmS^1\wedge \rmU_+ \wedge \rmX_+}}
 \end{equation} \ee
 \vskip .1cm \noindent
  Assume further that the $\mbS^{\rmG}$-suspension spectra of all the above simplicial presheaves are dualizable in 
  ${\widetilde {\Spt}}^{\rmG}(\k_{\rm mot})$. Then 
  \be \begin{equation}
  \label{add.pretr.trace}
  \tr'^{\rmG}(g) = \tr'^{\rmG}(f) + \tr'^{\rmG}(h)\quad  \mbox{ and } \tau^{\rmG}(g) = \tau^{\rmG}(f) + \tau^{\rmG}(h).
 \end{equation} \ee 
  Moreover if $\tr(f)$, $\tr(g)$ and $\tr(h)$ denote the induced transfer maps
 in any one of the three basic contexts (a), (b) or (c) of \cite[section 2]{CJ23-T2}, then 
 \be \begin{equation}
 \label{add.transfer}
 \tr(g) = j\circ \tr(f) + \tr(h)
 \end{equation} \ee
 where $\tr(f)$ denotes the transfer associated to the pre-transfer denoted $\tr'^{\rmG}(f)_{\rmU}: \mbS^{\rmG} \ra \mbS^{\rmG} \wedge \rmU_+$: see Remark ~\ref{clar.add.tr} below.
 \vskip .1cm
 Let $\cE^{\rmG}$ denote a commutative ring spectrum in ${\widetilde {\Spt}}^{\rmG}(\k_{\rm mot})$ or ${\widetilde {\Spt}}^{\rmG}(\k_{et})$. 
 In the latter case, we will further assume that $\cE^{\rmG}$ is $\ell$-complete for some prime $\ell \ne char(k)$. Then the 
 	corresponding results also hold if the smash products of the above simplicial presheaves
 	with the ring spectrum $\cE^{\rmG}$ are dualizable in ${\widetilde {\Spt}}^{\rmG}(\k_{\rm mot}, \cE^{\rmG} )$ and
 	${\widetilde {\Spt}}^{\rmG}(\k_{et}, \cE^{\rmG})$.
 \end{theorem}
\begin{remark}
  \label{clar.add.tr}
  Here it is important to observe that $\tr'^{\rmG}(f): \mbS^{\rmG} \ra \mbS^{\rmG} \wedge \rmX_+$ is defined as in  Definition ~\ref{pretransfer.3}, making use of the right 
  $\rmX_+$ co-module structure on $\rmU_+$. If one uses $\tr'^{\rmG}(f)_{\rmU}$ to denote the corresponding pre-transfer $\mbS^{\rmG} \ra \mbS^{\rmG} \wedge \rmU_+$,
  then it is clear that $\tr'^{\rmG}(f) = j \circ \tr'^{\rmG}(f)_{\rmU}$. Similarly it is important to observe that $\tr'^{\rmG}(h): \mbS^{\rmG} \ra \mbS^{\rmG} \wedge\rmX_+$
  is defined by making use of the right $\rmX_+$ co-module structure on $\rmX/\rmU$. Moreover, $\tr(f)$ will henceforth denote the transfer associated to the pre-transfer denoted $\tr'^{\rmG}(f)_{\rmU}$.
  \end{remark}
\begin{proof}Proposition ~\ref{pairing.hom.cats} below shows that the additivity of the transfer as in ~\eqref{add.transfer} follows from the additivity of the pre-transfer. Therefore,
 in what follows, we will only discuss the additivity of the pre-transfer. One may observe that we have proved the corresponding statements in the 
 non-equivariant case in \cite[Theorem 2.5]{JP23}. Moreover, the proof of \cite[Theorem 2.5]{JP23} is by showing that the proof of \cite[Theorem 7.10, Chapter III]{LMS} carries over to that framework.
 Therefore, we proceed to verify that the proof of \cite[Theorem 7.10, Chapter III]{LMS} carries over to our framework.  This will then complete the proof of
  Theorem ~\ref{additivity.tr.0}. This amounts to verifying that the big commutative diagram given on \cite[p. 166]{LMS} carries over to our framework. One may observe that this big diagram is broken up into various sub-diagrams, labeled (I) through (VII) and that it suffices to verify that each of these sub-diagrams commutes up to homotopy. 
   Moreover, one may observe that the maps that make up each of these sub-diagrams are natural maps and therefore are $\rmG$-equivariant, so that they hold
   in the present context. Moreover, one may feed each of these sub-diagrams into the Borel construction. This will prove that additivity holds for the $\rmG$-equivariant pre-transfer and the induced
   transfer on Borel-style
   generalized cohomology theories.
 
 \end{proof}
 \begin{theorem}
 	\label{additivity.tr.1}
 	Let $\rmF= \rmF_1\sqcup_{\rmF_3}\rmF_2$ denote a  pushout of unpointed simplicial
 	presheaves in $\Spc(\k)$, with 
 	the corresponding maps $\rmF_3 \ra \rmF_2$, $\rmF_3 \ra \rmF_1$ and $\rmF_j \ra \rmF$, for $j=1,2, 3$, assumed to be
 	cofibrations.
  Assume further the following:
 	\begin{enumerate}[\rm(i)]
 	\item all the above simplicial presheaves are provided with compatible actions by the linear algebraic group $\rmG$ making all the maps above $\rmG$-equivariant {\it and}
 	\item the $\mbS^{\rmG}$-suspension spectra of all the above simplicial presheaves are dualizable in ${\widetilde {\Spt}}^{\rmG}(\k_{mot})$.
 	\end{enumerate}
 	Let $i_j$ denote the  $\rmF_j \ra \rmF$, $j=1,2, 3$ as well 
 	as the corresponding map induced by $i_j$ on the Borel constructions.
 	Then 
 	\begin{enumerate}
 	\item 	$tr_{\rmF_+}'^{\rmG} =  i_1 \circ tr_{\rmF_{1+}}'^{\rmG} +  i_2 \circ tr_{\rmF_{2+}}'^{\rmG} -  i_3 \circ tr_{\rmF_{3+}}'^{\rmG} \mbox{ and } \tau_{\rmF_+}^{\rmG} =  \tau_{\rmF_{1+}}^{\rmG} +  \tau_{\rmF_{2+}}^{\rmG} -  \tau_{\rmF_{3+}}^{\rmG}$, \\
 	where \, $tr_{\rmF_+}'^{\rmG}$ and $tr_{\rmF_{j+}}'^{\rmG}$, $j=1, 2, 3$ ($\tau_{\rmF_+}^{\rmG}$, $\tau_{\rmF_{j+}}$, $j=1,2, 3$) denote the $\rmG$-equivariant pre-transfer maps ($\rmG$-equivariant trace maps, respectively) with equality holding in 
 	${\widetilde {\SH}}^{\rmG}(k_{mot})$, which denotes the corresponding homotopy category.
 	Moreover, $tr_{\rmF} =  i_1 \circ tr_{\rmF_{1}} +  i_2 \circ tr_{\rmF_{2}} -  i_3 \circ tr_{\rmF_{3}}$ which denote the corresponding transfer maps
 	in any one of the three basic contexts as in \cite[section 2]{CJ23-T2}, with equality holding in $\SH(k_{mot})$.
 	\item
 	In particular, taking $\rmF_2=*$, and $\rmF= Cone(\rmF_3 \ra \rmF_1)$, we obtain:\\
	    $tr_{\rmF}'^{\rmG} = i_1 \circ tr_{\rmF_{1+}}'^{\rmG} - i_3 \circ tr_{\rmF_{3+}}'^{\rmG}, \quad \, \tau_{\rmF}^{\rmG} =  \tau_{\rmF_{1+}}^{\rmG} - \tau_{\rmF_{3+}}^{\rmG}$ and 
	    $tr_{\rmF} = i_1 \circ tr_{\rmF_{1}} - i_3 \circ tr_{\rmF_{3}}$ \\ the last of which denote the corresponding transfer maps
 	in any one of the three basic contexts as in \cite[section 2]{CJ23-T2} with equality holding in ${\SH}(k_{mot})$.
        \end{enumerate}
 	Let $\cE^{\rmG}$ denote a commutative ring spectrum in ${\widetilde {\Spt}}^{\rmG}(\k_{\rm mot})$ or ${\widetilde {\Spt}}^{\rmG}(\k_{et})$. In the latter case, we will further assume that $\cE^{\rmG}$ is $\ell$-complete for some prime $\ell \ne char(k)$. Then the 
 	corresponding results also hold if the smash products of the above simplicial presheaves
 	with the ring spectrum $\cE^{\rmG}$ are dualizable in ${\widetilde {\SH}}^{\rmG}(k_{mot}, \cE^{\rmG})$ and
 	${\widetilde {\SH}}^{\rmG}(k_{mot}, \cE^{\rmG})$, which denotes the corresponding homotopy category.
 \end{theorem}	
 \begin{proof}
  One may observe that the hypotheses of Theorem \ref{additivity.tr.0} are satisfied with $\rmU$, $\rmX$ and $\rmX/\rmU$ there
  equal to the $\mbS^{\rmG}$-suspension spectra of $(\rmF_1 \sqcup \rmF_2)$, $\rmF$ and $\rmS^1 \wedge \rmF_{3,+}$. These arguments, therefore reduce the proof of Theorem ~\ref{additivity.tr.1} to 
   that of Theorem ~\ref{additivity.tr.0}.
 \end{proof}

\vskip .2cm
 \begin{proposition}(Multiplicative property of the pre-transfer and trace) \index{pre-transfer: multiplicativity} \index{trace: multiplicativity}
	\label{mult.prop}
 Assume $\rmF_i$, $i=1,2$ are simplicial presheaves provided with actions by the group $\rmG$.
 Let $\rmf_i: \rmF_i \ra \rmF_i$, $i=1,2$ denote a $\rmG$-equivariant map. Let $\rmF = \rmF_{1+} \wedge \rmF_{2+}$
 and let $\rmf= \rmf_{1+} \wedge \rmf_{2+}$.
 Then 
 \[ tr'_{\rmF}(\rmf) = tr'_{\rmF_{1+}}(\rmf_{1+}) \wedge tr'_{\rmF_{2+}}(\rmf_{2+}). \]
A corresponding result holds if $\rmF_2$ is a pointed simplicial presheaf with $\rmF = \rmF_{1+} \wedge \rmF_2$.
\end{proposition}
\begin{proof} A key point to observe is that the evaluation $e_{\rmF}: D(\rmF) \wedge \rmF \ra \mbS^{\rmG}$ is given by starting with 
\[e_{\rmF_{1+}} \wedge e_{\rmF_{2+}}: D(\rmF_{1+}) \wedge \rmF_{1+} \wedge D(\rmF_{2+}) \wedge \rmF_{2+} \ra \mbS^{\rmG} \wedge \mbS^{\rmG} \simeq \mbS^{\rmG},\]
and by precomposing it
with the map $D(\rmF) \wedge \rmF = D(\rmF_{1+} \wedge \rmF_{2+}) \wedge \rmF_{1+} \wedge \rmF_{2+} {\overset {\tau} \ra} D(\rmF_{1+}) \wedge {\rmF}_{1+} \wedge D(\rmF_{2+})\wedge \rmF_{2+} $, where $\tau$ is the
obvious map that interchanges the factors. Similarly the co-evaluation map
\[c: \mbS^{\rmG} \simeq \mbS^{\rmG} \wedge \mbS^{\rmG} {\overset {c_{\rmF_{1+}} \wedge c_{\rmF_{2+}}} \ra} \rmF_{1+}\wedge D(\rmF_{1+}) \wedge \rmF_{2+} \wedge D(\rmF_{2+})\]
provides the co-evaluation map for $\rmF$. The multiplicative 
property of the pre-transfer follows readily from the above two observations as well as from the definition of the pre-transfer as in Definition ~\ref{pretransfer.3}.
 These prove the statements when $\rmF_+ = \rmF_{1+} \wedge \rmF_{2+}$. The corresponding statements when
 $\rmF_2$ is already a pointed simplicial presheaf may be proven along entirely similar lines.
\end{proof}
\vskip .1cm
Let $\rmB$ denote a smooth quasi-projective scheme over the given base scheme $\Speck$. 
Let $\rmE \ra \rmB$ denote a  $\rmG$-torsor (in the given topology, which could be either the Zariski, Nisnevich, or \'etale) 
for the action of a linear algebraic group $\rmG$.
\vskip .2cm
Let $\cP_1$, $\cP_2$ denote $\rmG$-equivariant pointed simplicial presheaves on $\Speck$.
We may assume $\cP_1$ is a functorial cofibrant replacement of the given $\cP_1$, and $\cP_2$ is a functorial fibrant replacement 
of $\cP_2$ in the given model structure on the category of non-equivariant 
simplicial presheaves, so that we let $[\cP_1, \cP_2] = \pi_0(\Map(\cP_1, \cP_2))$
where $\Map$ denotes the simplicial mapping space. One may recall that, according to \cite[Proposition 2.5]{CJ23-T1},
both $\cP_1$ and $\cP_2$ are $\rmG$-equivariant simplicial presheaves.
\begin{proposition}
 \label{pairing.hom.cats} 
 \begin{enumerate}[\rm(i)]
\item Given $\rmG$-equivariant pointed simplicial presheaves $\cP_1, \cP_2$ on $\Speck$, one obtains a natural map
\[ [\cP_1, \cP_2] \ra [ \rmE{\underset {\rmG} \times}\cP_1, \rmE{\underset {\rmG} \times}\cP_2]\]
where the $[\quad, \quad]$ is defined above and the $[\quad, \quad]$ on the right denotes the Hom in the homotopy category of simplicial presheaves over 
$\rmB$ as in \cite[Terminology 2.3]{CJ23-T1}.
Moreover the quotient construction $\rmE{\underset {\rmG} \times}\cP_i$ is carried out  as in \cite[(8.3.6)]{CJ23-T1}, that is,
 when $\rmG$ is special as a linear algebraic group, the quotient is taken on the Zariski (or Nisnevich) site of $\rmE$, while
 when $\rmG$ is not special, it is taken on the \'etale site and then followed by a derived push-forward to the
 Nisnevich site. 
\vskip .1cm
\item Given pointed simplicial presheaves $\cQ_i$, $i=1,2,3,4$, over $\rmB$, we obtain a pairing:
\[ [\cQ_1, \cQ_2] \times [\cQ_3, \cQ_4] \ra [\cQ_1 {\underset {\rmB} \wedge} \cQ_2, \cQ_3 {\underset {\rmB} \wedge} \cQ_4]\]
where the $[\quad, \quad ]$ denotes the Hom in the homotopy category of simplicial presheaves over $\rmB$.
\end{enumerate}
\end{proposition}
\begin{proof} (i) One may recall that, according to \cite[Proposition 2.5]{CJ23-T1},
the functorial cofibrant and fibrant replacements of $\cP_i$, $i=1,2$, in the given model structure on the category of non-equivariant 
simplicial presheaves are $\rmG$-equivariant simplicial presheaves. Therefore, one may apply the Borel construction to them, and
 then the assertion in (i) is clear.
\vskip .1cm
For (ii), we may assume $\cQ_1, \cQ_3$ are cofibrant and $\cQ_2, \cQ_4$ are fibrant in the given model structure, so that $[\cQ_1, \cQ_2] = \pi_0(\Map(\cQ_1, \cQ_2))$
 and $[\cQ_3, \cQ_4] = \pi_0(\Map(\cQ_3, \cQ_4))$, where $\Map$ denotes the corresponding simplicial mapping space. Then the assertion in (ii)
is also clear.
\end{proof}

\section {\bf Mayer-Vietoris and additivity of the transfer}

 We establish the additivity and Mayer-Vietoris property for the transfer, but {\it only} for generalized
  cohomology theories defined with respect to spectra that have {\it the rigidity property} as discussed
  in Definition ~\ref{rigid.prop}. This will also provide a second proof of the additivity and Mayer-Vietoris property
  for the trace, but for generalized cohomology theories defined with respect to spectra that have the rigidity property.
 \begin{definition} (Rigidity) \index{rigidity}
\label{rigid.prop}
Assume the base field $k$ is infinite.
 Let $\rmM$ denote a motivic spectrum, so that its presheaves of homotopy groups are all $\ell$-primary torsion for 
 a fixed prime $\ell$ different from the characteristic of the base field $k$. We will say that $\rmM$ has {\it the rigidity property} if for any smooth scheme
$\rmX$ of finite type over the given base field $k$, and a point $x$ of $\rmX$, with residue field $\k(x)$ and Hensel local ring at $x$ given by
$\O_{\rmX, \x}^h$, the  map ${\rm Spec}\, k(x) \ra {\rm Spec} \, \O_{\rmX, \x}^h$ induces a weak-equivalence: $\Gamma ({\rm Spec} \, \O_{\rmX, \x}^h, \rmM) \simeq \Gamma({\rm Spec} \, \k(\x), \rmM)$. 
\vskip .1cm
Let $\rmM$ denote a spectrum in $\Spt(\k_{et})$ so that its presheaves of homotopy groups are all $\ell$-primary torsion for 
 a fixed prime $\ell$ different from the characteristic of the base field $k$. Assume further that the field $\k$ has finite $\ell$-cohomological dimension. We will say that $\rmM$ has {\it the rigidity property} if for any smooth scheme
$\rmX$ of finite type over the given base field $k$, and a point $x$ of $\rmX$, with residue field $k(x)$ and  strict Hensel local ring at $x$ given by
$\O_{\rmX, \x}^{sh}$, with residue field ${\overline {k(x)}}$, the  following conditions are satisfied: (i) the 
map ${\rm Spec}\, {\overline {\k(x)}} \ra {\rm Spec} \, \O_{\rmX, \x}^{sh}$ induces a weak-equivalence: $\Gamma ({\rm Spec} \, \O_{\rmX, \x}^{sh}, \rmM) \simeq \Gamma({\rm Spec} \, {\overline {k(\x)}}, \rmM)$, and (ii) for any inclusion $\k \ra K$ of separably closed fields, the induced map $\Gamma ({\rm Spec} \, \k, \rmM) \ra \Gamma ({\rm Spec} \, K, \rmM)$
is a weak-equivalence.
\end{definition} 
\begin{remark}
The following {\it rather subtle point} is the main role of rigidity in our work. Given a presheaf $\rmM$ on the {\it big Nisnevich site}
over $\k$, there is no apriori reason for the cohomology with respect to $\rmM$ for the Henselization of a given scheme
$\rmX$ along a closed sub-scheme $\rmZ$ to be isomorphic to the cohomology of $\rmZ$ with respect to $\rmM$. This issue
does not arise if $\rmM$ is a sheaf on the {\it small} Nisnevich site of $\rmX$. The assumption that $\rmM$ has the
rigidity property, then does ensure that the above cohomologies are isomorphic.
\end{remark}
\vskip .1cm
        The first step in the corresponding {\it transfer} is to apply a suitable form of the Borel construction $\rmE{\underset {\rmG} \times} \quad$ to the pre-transfer, where
        $\rmE \ra \rmB$ is a $\rmG$-torsor, for the linear algebraic group $\rmG$.
The  transfer $\tr(\rmf)$ defined on generalized equivariant motivic cohomology theories is obtained by starting with
 the $\rmG$-equivariant pre-transfer, feeding it to a suitable form of the Borel construction, and then performing various modifications to it
 as discussed in detail in \cite[8.4. Construction of the transfer]{CJ23-T1}. 

We begin with the following results. 
\begin{construction}
 \label{more.rigidity}
{\rm One may make use of the following localization technique to force rigidity, especially when doing the Borel construction with respect to linear algebraic groups that are {\it not special}.
Let $\{\rmZ \ra \rmX^h_{\rmZ}\}$ denote the family of Henselizations of smooth schemes $\rmX$ along a closed smooth subscheme $\rmZ$.
Now one may enlarge the generating trivial cofibrations on the stable motivic homotopy category $\Spt(\k_{\rm mot})$ by including the $\T$-suspension spectra of the above family of
maps among the generating trivial cofibrations. In the resulting model category, one can see that the fibrant objects are exactly the fibrant
spectra in $\Spt(\k_{\rm mot})$ that have the rigidity property. We will denote the corresponding model category of motivic spectra by $\Spt(\k_{mot,r})$.
Let $\epsilon^*: \Spt(\k_{mot, r}) \ra \Spt(\k_{et})$ denote the pull-back to the \'etale site.
In order that $\epsilon^*$ be a left-Quillen functor, it is clear that we need to enlarge the generating trivial cofibrations on $\Spt(\k_{et})$
by adding maps of the form: $\{\epsilon^*({\Sigma^{\infty}_{\T}}\rmZ) \ra \epsilon^*({\Sigma^{\infty}_{\T}}\rmX^h_{\rmZ})\}$. We will denote the resulting model category of \'etale spectra by $\Spt(\k_{et,r})$.
In a similar manner, on incorporating the action of a linear algebraic group $\rmG$, we obtain the left-Quillen functor on the corresponding model categories:
\be \begin{equation}
\label{invert.hensel}
\epsilon^*:{\widetilde {\Spt}}^{\group}(\k_{mot, r}) \ra {\widetilde {\Spt}}^{\group}(\k_{et, r})
\end{equation} \ee
with its right adjoint given by} $\epsilon_*$.
\end{construction}

\begin{proposition} \index{rigidity: factorization of the diagonal map}
 \label{diagonal.map}
 Assume  that $\rmZ$ is a smooth closed $\rmG$-subscheme of the smooth $\rmG$-scheme $\rmX$, where
 $\rmG$ is a linear algebraic group, and that $\rmX_{\rmZ}^{\rmh}$ denotes the Henselization of $\rmX$ along $\rmZ$.
 \begin{enumerate}[\rm(i)]
\item Then one obtains the commutative square:
 \[\xymatrix{{\rmX_{\rmZ}^h/ (\rmX_{\rmZ}^h-\rmZ)} \ar@<1ex>[d]^{\Delta^h} \ar@<1ex>[r] & {\rmX/\rmU} \ar@<1ex>[d]^{\Delta}\\
              {(\rmX_{\rmZ}^h/ (\rmX_{\rmZ}^h-\rmZ))\wedge \rmX_{\rmZ,+}^h} \ar@<1ex>[r] & {\rmX/\rmU \wedge \rmX_+}}
 \]
 where the top horizontal map induces a motivic stable equivalence on the associated suspension spectra and $\Delta^h$, $\Delta$ denote the
 corresponding diagonal maps. Taking the smash product with the sphere spectrum $\mbS^{\rmG}$, one obtains a similar
 commutative square, where each term above is replaced by the smash product with $\mbS^{\rmG}$.
 \vskip .1cm
\item Let $\rmp: \rmE \ra \rmB$ denote a $\rmG$-torsor for a linear algebraic group $\rmG$, and let $\epsilon: \Speck_{et} \ra \Speck_{Nis}$ denote the map of sites
  from the big \'etale site of $\Speck$ to the corresponding big Nisnevich site.
 Then the commutative diagram in (i) induces the commutative diagram:
 \[\xymatrix{{\rmE{\underset {\rmG} \times}(\mbS^{\rmG}\wedge(\rmX_{\rmZ}^h/ (\rmX_{\rmZ}^h-\rmZ)))} \ar@<1ex>[d]^{\rmE{\underset {\rmG} \times}(\mbS^{\rmG} \wedge \Delta^h)} \ar@<1ex>[r] & {\rmE{\underset {\rmG} \times}(\mbS^{\rmG}\wedge (\rmX /\rmU))} \ar@<1ex>[d]^{\rmE{\underset {\rmG} \times}(\mbS^{\rmG} \wedge \Delta)}\\
              {\rmE{\underset {\rmG} \times}(\mbS^{\rmG}\wedge(\rmX_{\rmZ}^h/ (\rmX_{\rmZ}^h-\rmZ)\wedge \rmX_{\rmZ,+}^h))} \ar@<1ex>[r] & {\rmE{\underset {\rmG} \times}(\mbS^{\rmG}\wedge (\rmX /\rmU \wedge \rmX_+))}}
 \]
 so that the map in the top row is again a weak-equivalence. Here the quotient construction $\rmE{\underset {\rmG} \times} \quad$ is carried out as 
 follows: 
 when $\rmG$ is special as a linear algebraic group, the quotient is taken on the big Zariski (or the big Nisnevich) site, while
 when $\rmG$ is not special, it is taken after applying ${\rm L}\epsilon^*$  to the model category ${\widetilde {\Spt}}^{\group}(\k_{et, r})$ ( on the big \'etale site)
 and then followed by the derived push-forward $\rmR\epsilon_*$ to the model category ${\widetilde {\Spt}}^{\group}(\k_{mot, r})$ (on the big Nisnevich site).
 \vskip .1cm
 \item Assume the situation in {\rm (i)}. Let $\rmM$ denote a fibrant motivic spectrum that has the rigidity property as 
 in Definition ~\ref{rigid.prop}.  For any spectrum $\X \in {\widetilde {\Spt}}^{\rmG}(\k_{\rm mot})$, we let $\X^{\rmV}$ denote the simplicial 
 presheaf $\X(\rmT_{\rmV})$ and 
 $\rmA= (\mbS^{\rmG} \wedge \rmX/\rmU) \wedge \rmD(\mbS^{\rmG} \wedge \rmX/\rmU)$, 
$\tau\rmA = \rmD(\mbS^{\rmG} \wedge \rmX/\rmU) \wedge (\mbS^{\rmG} \wedge \rmX/\rmU)$. Then, one obtains the weak-equivalence
 \be \begin{equation}
 \label{Hensel.weak.equiv}
 \RHom(\{{\rmE{\underset {\rmG} \times}(\tau \rmA \wedge  \rmX_{\rmZ,+}^h}))^{\rmV} \wedge _{\rmB}{\rm S}(\eta^{\rmV}\oplus 1))/s|\rmV\}, \rmM) \simeq \RHom( \{{\rmE{\underset {\rmG} \times}(\tau \rmA \wedge \rmZ_+}))^{\rmV} \wedge_{\rmB}S(\eta^{\rmV}\oplus 1))/s|\rmV\}, \rmM).
 \end{equation} \ee
 where $\RHom$ denotes the derived internal hom in $\Spt(\k_{mot})$, and $s$ denotes the obvious section.
 \vskip .1cm
 \item Let $\rmM$ denote a fibrant motivic spectrum that has the rigidity property as 
 in Definition ~\ref{rigid.prop}. Then corresponding results as in (iii) hold for the spectrum ${\rm L}\epsilon^*(\rmM)$ when the quotients
 appearing above are replaced by the quotients in the \'etale topology of the corresponding sheaves pulled back to the \'etale site.
 \end{enumerate}
\end{proposition}
\begin{proof} The commutative square in (i) follows readily from the cartesian square:
\be \begin{equation}
     \label{Hensel.1}
     \xymatrix{{\rmZ} \ar@<1ex>[r]^{id} \ar@<1ex>[d] & {\rmZ} \ar@<1ex>[d]\\
               {\rmX^h_{\rmZ}} \ar@<1ex>[r] & {\rmX}}
\end{equation} \ee
Next one may recall from \cite[p. 117, Lemma 2.27]{MV} and \cite[p. 115, Theorem 2.23]{MV} the weak-equivalences:
\be \begin{equation}
     \label{def.norm.cone}
     \rmX/\rmU \simeq Th(\cN) \simeq \rmX^h_{\rmZ}/(\rmX^h_{\rmZ} - \rmZ).
\end{equation} \ee
(In fact, the normal bundle $\cN$ pulls back to the normal bundle associated to the closed immersion of $\rmZ$ in $\rmX_{\rmZ}^h$.
Implicit in the purity theorem \cite[p. 115, Theorem 2.23]{MV} is the technique of
deformation to the normal-bundle: see \cite[section 2]{Verd}, or 
\cite[p. 116, Lemma 2.26]{MV}. \index{deformation to the normal bundle})
This proves that the map in the top row of the square in (i), namely that the map $\rmX^h_{\rmZ} /(\rmX^h_{\rmZ} - \rmZ) \ra \rmX/(\rmX- \rmZ)$
is a weak-equivalence. Therefore, these complete the proofs of the statements in (i).
\vskip .1cm
Next we consider the statements in (ii). The functoriality of the Henselization as in Lemma ~\ref{funct.Hensel}
shows that the action of the group $\rmG$ on $\rmX$ and $\rmZ$ induces an action by $\rmG$ on $\rmX^h_{\rmZ}$.
This observation, along with the observations in (i) provides the commutative square in (ii). The fact that the
 top row in the corresponding square is also a weak-equivalence follows from the fact that the 
 top row of the square in (i) is also a weak-equivalence, making use of the appropriate quotient construction
 utilized there, as discussed in (ii).
\vskip .1cm
When the group $\rmG$ is special, the torsor $\rmE \ra \rmB$ trivializes on a Zariski open cover. Therefore, the weak-equivalence in ~\eqref{Hensel.weak.equiv} follows from the above observations, in view of Lemma ~\ref{Cech.hypercover}(i) and
the weak-equivalence provided by Theorem ~\ref{coh.tub.nbd} as the spectrum $\rmM$ is assumed to have the
  rigidity property.  In general, the torsor $\rmE \ra \rmB$ only trivializes on
  an \'etale cover. Then the Borel construction makes use of the quotient construction in $\Spt(\k_{et, r})$ after applying ${\rm L}\epsilon ^*$. The discussion in the first paragraph of 
  the Construction ~\ref{more.rigidity} shows that this preserves weak-equivalences and so does the right derived functor $\rmR\epsilon_*$ in
  ~\eqref{invert.hensel}. These prove the statement in (iii). 

\vskip .1cm 

Next we consider the statement in (iv). This follows readily, since $\epsilon^*(\rmM)$ also has the rigidity property, the torsor $\rmE \ra \rmB$ is locally trivial
 in the \'etale topology and in view of Theorem ~\ref{coh.tub.nbd}(iv). 
 \end{proof}
 \begin{remark}
 One can see that the results of the last Proposition strongly depend on the rigidity property for spectra as well as properties
 of Henselization of smooth schemes along closed smooth subschemes. We provide a detailed discussion on the rigidity property in Propositions ~\ref{rigid.props.1}, ~\ref{slices.rigid}, Corollary ~\ref{slices.rigid} and
 Lemma ~\ref{rigid.et.Nis}, 
 later on in this chapter. Henselization along closed subschemes also leads us to discuss what we call {\it motivic tubular neighborhoods}: see
 Definition ~\ref{tub.nbds} and Theorem ~\ref{coh.tub.nbd}.
  \end{remark}

\vskip .1cm
  \begin{conventions}
   \label{add.transf.convents}
 \begin{enumerate}[\rm(i)]
  \item 
$p \ge 0$ will denote the characteristic of the base field $k$. 
\item Throughout the following theorem, its proof and various applications, such as 
 Theorem ~\ref{add.transf}, Proposition ~\ref{fixed.pts.proj.case}, Theorem ~\ref{double.coset.1} and Corollaries ~\ref{double.coset.2} through  ~\ref{double.coset.5},  we will fix a commutative motivic ring spectrum $\cE$,
   with $\cE^{\rmG}$ its lift to a $\rmG$-equivariant spectrum, all chosen as discussed in \cite[(4.0.24)]{CJ23-T1}.
   (Recall this means that in positive characteristics $p$, $\cE$ denotes any one of the spectra 
   ${\Sigma_{\T}^{\infty}}[p^{-1}]$, ${\Sigma^{\infty}_{\T, (\ell)}}$, ${\Sigma_{\T}^{\infty}}\compl_{\ell}$ with $\cE^{\rmG}$ denoting its 
lift to the equivariant spectrum  $\mbS^{\rmG}[p^{-1}]$, $\mbS^{\rmG}_{(\ell)}$, $\widehat {\mbS^{\rmG}}_{\ell}$, \res.)
  $\rmM$ will always denote module spectra over the given ring spectrum $\cE$.
 \item
  Moreover, whenever we invoke any one of three the basic contexts for the transfer as in  \cite[section 2]{CJ23-T2} in the rest of this paper, we 
  will assume the self-map $\rmf:\rmX\ra \rmX$ there is the identity map, and the scheme $\rmY$ there is the base scheme $Spec \, \k$: that is, the additivity for the 
  transfer and its various applications will be discussed only with $\rmf =id_{\rmX}$ and where $\rmY = Spec \, \k$.
\item
In order to ensure that the Borel construction for non-special groups
 (which has to be carried out by pull-back to the \'etale site) is compatible with the notions of rigidity, it is often convenient (though not essential) to
  adopt the model structures
 as discussed in section ~\ref{more.rigidity}. One may assume this implicitly throughout the following theorem and its various applications.
 \end{enumerate}
 \end{conventions} 
 \vskip .1cm 
  \begin{theorem} (Mayer-Vietoris and Additivity for the transfer) \index{transfer: Mayer-Vietoris} \index{transfer: additivity}
\label{add.transf} 
\begin{enumerate}[\rm(i)]
\item Let $\rmX$ denote a smooth $\rmG$-scheme and let $i_j:\rmU_j \ra \rmX$, $j=1,2$ denote the open immersion of two Zariski open subschemes of $\rmX$ 
which are both  $\rmG$-stable, with $\rmX= \rmU_1 \cup \rmU_2$. 
\vskip .1cm
Then adopting the 
 terminology above, (that is, where $\tr_{\rmP}'^{\rmG}$ denotes the $\rmG$-equivariant pre-transfer associated to the
 $\rmG$-simplicial presheaf $\rmP$ and $\tau_{\rmP}^{\rmG}$ denotes the corresponding $\rmG$-equivariant trace)
 \[ \tr_{\rmX}'^{\rmG} = i_1\circ tr_{\rmU_1}'^{\rmG} + i_2 \circ tr_{\rmU_2}'^{\rmG} - i_3 \circ tr_{\rmU_1 \cap \rmU_2}'^{\rmG} \mbox{ and } \tau_{\rmX_+}^{\rmG} = \tau_{\rmU_{1+}}^{\rmG} + \tau_{\rmU_2}^{\rmG} - \tau_{(\rmU_1 \cap \rmU_2)_+}^{\rmG}\]
in case $char (k) =0$. In case $char (k) =p>0$, we obtain
 \[ \tr_{\rmX, {\cE^{\rmG}}}'^{\rmG} = i_1\circ tr_{\rmU_1, {\cE^{\rmG}}}'^{\rmG} + i_2 \circ tr_{\rmU_2, {\cE^{\rmG}}}'^{\rmG} - i_3 \circ tr_{\rmU_1 \cap \rmU_2, {\cE^{\rmG}}}'^{\rmG} \mbox{ and } \tau_{\rmX_+, {\cE^{\rmG}}}^{\rmG} = \tau_{\rmU_{1+}, {\cE^{\rmG}}}^{\rmG} + \tau_{\rmU_{2+}, {\cE^{\rmG}}}^{\rmG}- \tau_{(\rmU_1 \cap \rmU_2)_+, {\cE^{\rmG}}}^{\rmG}.\]
 \vskip .1cm
 We also obtain in all characteristics,
 \[\tr_{\rmX} = i_1\circ tr_{\rmU_1} + i_2 \circ tr_{\rmU_2} - i_3 \circ tr_{\rmU_1 \cap \rmU_2}\]
which denote the induced transfers as in any one of the three basic contexts discussed in \cite[section 2]{CJ23-T2}, with respect to a motivic spectrum $\rmM$.
 \vskip .1cm \noindent
\item Let $i: \rmZ \ra \rmX$ denote a closed immersion of smooth $\rmG$-schemes 
 with $j: \rmU \ra \rmX$ denoting the corresponding open complement.  In this context we will let
 \begin{equation}
 \label{trXU}
 \tr_{\rmX/\rmU}'^{\rmG}: \mbS^{\rmG} \ra \mbS^{\rmG} \wedge \rmX_+
 \end{equation}
 denote the pre-transfer defined as in Definition ~\ref{pretransfer.3} with $\rmP$ ($\rmC$) there denoting $\rmX/\rmU$ ($\rmX$, \res).
 Then adopting the 
 terminology above and in Remark ~\ref{clar.add.tr}, 
\[ \tr_{\rmX}'^{\rmG} = j\circ \tr_{\rmU}'^{\rmG}+ \tr_{\rmX/\rmU}'^{\rmG},  \mbox{ and } \tau_{\rmX_+}^{\rmG} = \tau_{\rmU_+}^{\rmG} + \tau_{(\rmX/\rmU)}^{\rmG}\]
in case $char (k) =0$. In case $char (k) =p>0$, we obtain
\[ \tr_{\rmX, {\cE^{\rmG}} }'^{\rmG} = j\circ \tr_{\rmU, {\cE^{\rmG}}}'^{\rmG}+ \tr_{\rmX/\rmU, {\cE^{\rmG}}}'^{\rmG},  \mbox{ and } \tau_{\rmX_+, {\cE^{\rmG}}}^{\rmG} = \tau_{\rmU_+, {\cE^{\rmG}}}^{\rmG} + \tau_{(\rmX/\rmU), {\cE^{\rmG}}}^{\rmG},\]
where the pre-transfer $\tr_{\rmX/\rmU, \cE^{\rmG}}'^{\rmG}: \cE^{\rmG} \ra \cE^{\rmG} \wedge \rmX_+$ is defined as in
Definition ~\ref{generalized.traces} with $\rmP$ ($\rmC$) there denoting $\rmX/\rmU$ ($\rmX_+$, \res).
\vskip .1cm
 Again adopting the terminology as in Remark ~\ref{clar.add.tr}, we also obtain in all characteristics,
\[ \tr_{\rmX} = j\circ \tr_{\rmU}+ \tr_{\rmX/\rmU}\]
which denote the  transfers induced by the corresponding pre-transfers as in any one of the three basic contexts discussed \cite[section 2]{CJ23-T2}, with respect to a motivic spectrum $\rmM$.  
\vskip .1cm 
For the remainder of this theorem, we will assume that the base field $k$ is infinite and that it contains a $\sqrt{-1}$. \index{$\sqrt{-1}$}
\vskip .1cm \noindent
\item Let $\rmM$ denote a motivic spectrum that has the rigidity property (as discussed in Definition ~\ref{rigid.prop}). Let $\cN$ denote the normal bundle associated
 to the closed immersion $i$ as in (ii), and let $\Th(\cN)$ denotes its Thom-space. Then we obtain in all characteristics
\[ \tr_{\rmX/\rmU} = \tr_{\Th(\cN)} = i \circ \tr_{\rmZ},\]
 \vskip .1cm \noindent
 where the following notational conventions hold: 
 \begin{itemize}
 \item $\tr_{\rmX/\rmU}$, $ \tr_{\Th(\cN)}$, $ \tr_{\rmZ}$
 denote the transfers induced on generalized motivic cohomology with respect to the motivic spectrum $\rmM$ as in \cite[section 2]{CJ23-T2}, and
 \item 
 where the
  corresponding $\rmG$-equivariant pre-transfers are defined with respect to the ring spectrum $\mbS^{\rmG}$ in characteristic $0$ and
  with respect to one of the ring spectra $\mbS^{\rmG}[p^{-1}]$, $\mbS^{\rmG}_{(\ell)}$, $\widehat {\mbS^{\rmG}}_{\ell}$, in case $char (k) = p>0$, with $\ell$ is a prime 
 different from $p$.
 \end{itemize}
 \vskip .1cm \noindent
 \item  Again let $\rmM$ denote a motivic spectrum that has the rigidity property (as discussed in Definition ~\ref{rigid.prop}). Let $\{\rmS_{\alpha}|\alpha\}$ denote a stratification of the smooth scheme $\rmX$ into finitely many locally closed and smooth $\rmG$-stable subschemes 
 $\rmS_{\alpha}$. For each $\alpha$, let $i_{\alpha}: \rmS_{\alpha} \ra \rmX$ denote the corresponding locally closed immersion. Then  
 one obtains in all characteristics
\[ \tr_{\rmX} = {\rm \Sigma} _{\alpha} i_{\alpha} \circ \tr_{\rmS_{\alpha}},\]
 \vskip .1cm \noindent
where the following notational conventions hold: 
\begin{itemize}
\item $\tr_{\rmX}$, $ \tr_{\rmS_{\alpha}}$
denote the transfers induced on generalized motivic cohomology theories with respect to the motivic spectrum $\rmM$ as in the basic contexts in \cite[section 2]{CJ23-T2}, and
\item where the
  corresponding $\rmG$-equivariant pre-transfers are defined with respect to the ring spectrum $\mbS^{\rmG}$ in characteristic $0$ and
  with respect to one of the ring spectra $\mbS^{\rmG}[p^{-1}]$, $\mbS^{\rmG}_{(\ell)}$, $\widehat {\mbS^{\rmG}}_{\ell}$, in case $char (k) = p>0$, and where $\ell$ is a prime 
 different from $p$.
 \end{itemize}
\vskip .1cm \noindent
\item Let $\cE^{\rmG}$ denote a commutative ring spectrum in ${\Spt}^{\rmG}(\k_{\rm mot})$, whose presheaves of homotopy groups are  all $\ell$-primary torsion for a
 fixed prime $\ell \ne char (k)$, and
let $\epsilon^*(\cE^{\rmG})$ denote the corresponding spectrum in $\Spt^{\rmG}(\k_{et})$. 
  Then the 
 	results corresponding to (i) through (ii) also hold if $\tr_{\rmZ}'^{\rmG}$ ($\tau_{\rmZ_+}^{\rmG}$) is
 	replaced by $\tr_{\rmZ_+ \wedge \epsilon^*(\cE^{\rmG})}'^{\rmG}$ ($\tau_{\rmZ_+ \wedge \epsilon^*(\cE^{\rmG})}^{\rmG}$, \res). Assume next that $\rmM$ is a module spectrum over $\cE= (i^*\circ {\widetilde {\mathbb P}} \circ {\widetilde {\rm U}})^*(\cE^{\rmG})$ 
 	(which is the non-equivariant spectrum obtained from $\cE^{\rmG}$ as in \cite[Definition 4.13]{CJ23-T1} so that $\rmM$ 
 has the rigidity property as in Definition ~\ref{rigid.prop}. Then
 	the results corresponding to (i) through (iv) hold for the transfer $tr_Z$ in generalized \'etale cohomology theories defined with respect to the spectrum $\epsilon ^*(\rmM)$.
 \end{enumerate}
 \end{theorem}
 \vskip .2cm
As is shown in \cite{L} and \cite{LMS}, the additivity and Mayer-Vietoris property of the transfer can be deduced by showing that the corresponding pre-transfer (that is, the transfer
 for the trivial group) is additive, or equivalently, has what is often called the Mayer-Vietoris property.\footnote{In our setting, we also need to
 invoke rigidity as discussed in Definition ~\ref{rigid.prop}.}  We establish such a property, by systematically verifying that the same general strategy carries over to the motivic and \'etale framework. 
\vskip .2cm
\begin{proof}
Once again, we will explicitly discuss only the case where the ring spectrum $\cE^{\rmG}$
 ($\cE$) is the equivariant sphere spectrum $\mbS^{\rmG}$ (the sphere spectrum ${\Sigma_{\T}^{\infty}}$, \res), as proofs in the other cases follow 
 along the same lines.
 \vskip .1cm
 First we will consider (i), namely the Mayer-Vietoris sequence. For this,  one begins with the stable cofiber sequence 
 \[\mbS^{\rmG}\wedge (\rmU_1\cap \rmU_2)_+ \ra \mbS^{\rmG}\wedge (\rmU_1 \sqcup \rmU_2)_+ \ra \mbS^{\rmG}\wedge (\rmX)_+. \] 
Then one applies
Theorem ~\ref{additivity.tr.1}(1) to it. This proves (i).
 \vskip .1cm
 Next we will consider (ii). One recalls the stable homotopy cofiber sequence (see \cite[p. 115, Theorem 2.23]{MV})
  \be \begin{equation}
    \label{loc.1}
   \mbS^{\rmG}\wedge \rmU_+ \ra \mbS^{\rmG} \wedge \rmX_+ \ra \mbS^{\rmG}\wedge (\rmX/\rmU) \simeq \mbS^{\rmG} \wedge \Th(\cN)
  \end{equation} \ee
in the stable motivic homotopy category over the base scheme. The  statement in (ii) follows by applying Theorem ~\ref{additivity.tr.0} to the
 stable homotopy cofiber sequence in ~\eqref{loc.1}. 
\vskip .1cm
Next we consider (iii).  However, as shown below in 
Proposition ~\ref{diagonal.map}, one needs to supplement this with the technique of {\it motivic tubular neighborhoods}. \index{motivic tubular neighborhood}
 \vskip .1cm 
 A key step here is to observe the homotopy commutative diagram: 
 \fontsize{9}{12}
\vskip .1cm 

\be \begin{equation}
   \label{factor.diagram}
   \xymatrix{{(\rmE{\underset {\rmG} \times}(\mbS^{\rmG}\wedge \rmX_+)^{\rmV}\wedge _{\rmB}S(\eta^{\rmV}\oplus 1))/s}  & {(\rmE{\underset {\rmG} \times}(\mbS^{\rmG}\wedge \rmX^h_{Z_+})^{\rmV}\wedge _{\rmB}S(\eta^{\rmV}\oplus 1))/s} \ar@<1ex>[l]^{\tilde i^h \wedge id} \\
              {(\rmE{\underset {\rmG} \times}(\tau \rmA \wedge \mbS^{\rmG}\wedge \rmX_+)^{\rmV}\wedge _{\rmB}S(\eta^{\rmV}\oplus 1))/s} \ar@<1ex>[u]^{(e \wedge id)}  & {(\rmE{\underset {\rmG} \times}(\tau \rmA \wedge \mbS^{\rmG}\wedge \rmX^h_{Z_+})^{\rmV}\wedge _{\rmB}S(\eta^{\rmV}\oplus 1))/s} \ar@<1ex>[l]^{\tilde i^h \wedge id} \ar@<1ex>[u]^{(e \wedge id)} \\
              {(\rmE{\underset {\rmG} \times}(\tau \rmA)^{\rmV}\wedge _{\rmB}S(\eta^{\rmV}\oplus 1))/s} \ar@<1ex>[u]^{id \wedge \Delta \wedge id} \ar@<1ex>[ur]_{id \wedge {\Delta^h} \wedge id}\\
              {(\rmE{\underset {\rmG} \times}\rmA^{\rmV} \wedge _{\rmB}S(\eta^{\rmV}\oplus 1))/s} \ar@<1ex>[u]^{\tau \wedge id }\\
              {(\rmE{\underset {\rmG} \times}\mbS^{\rmG, \rmV} \wedge _{\rmB}S(\eta^{\rmV}\oplus 1))/s} \ar@<1ex>[u]^{c\wedge id}}
\end{equation} \ee
\normalsize

\vskip .1cm 
{\it Here we have adopted the following terminology}:
\begin{enumerate}
 \item  $\rmX_{\rmZ}^h$ denotes the Henselization of the scheme $\rmX$ along $\rmZ$, and $\tilde i^h: \rmX^h_{\rmZ} \ra \rmX$ is the obvious map.
 \item If $\{\rmT_{\rmV}|\rmV\}$ denotes  the $\rmG$-equivariant 
sphere spectrum $\mbS^{\rmG}$, $\eta^{\rmV}$ denotes a vector bundle on $\rmB$ complimentary to the vector bundle $\rmE\times_{\rmG} {\rmV}$: see \cite[8.4. Construction of the transfer, Step 2]{CJ23-T1}. Moreover, ${\rm S}(\eta^{\rmV}\oplus 1)$ denotes the corresponding sphere-bundle.
\item
For any spectrum $\X \in {\widetilde {\Spt}}^{\rmG}(\k_{\rm mot})$, we let $\X^{\rmV}$ denote the simplicial presheaf $\X(\rmT_{\rmV})$ and the 
spectrum $\X$ itself will be denoted $\{\X^{\rmV}|\rmV\}$. Moreover,
we will let $\rmA= (\mbS^{\rmG} \wedge \rmX/\rmU) \wedge \rmD(\mbS^{\rmG} \wedge \rmX/\rmU)$ and let 
$\tau\rmA = \rmD(\mbS^{\rmG} \wedge \rmX/\rmU) \wedge (\mbS^{\rmG} \wedge \rmX/\rmU)$. In addition, we are also assuming that both $\rmB$ and $\rmE$ are {\it smooth affine} schemes, with $\rmE$ a $\rmG$-torsor. 
\item
Let $\X =\{\X^{\rmV}|\rmV\}$ denote a spectrum in $\Spt^{\rmG}(\k_{\rm mot})$. When $\rmG$ is special, 
$(\rmE{\underset {\rmG} \times}(\X)^{\rmV}\wedge _{\rmB}S(\eta^{\rmV}\oplus 1))/s$ denotes just that, while when $\rmG$ is not special,
it denotes the object $\rmR\epsilon_*(\rmE{\underset {\rmG} \times^{et}}\epsilon^*((\X)^{\rmV})\wedge _{\epsilon^*(\rmB)}\epsilon^*(S(\eta^{\rmV}\oplus 1))/s)$.
$s$ denotes the canonical section, and we are taking the quotient using the
quotient construction $\rmE{\underset {\rmG} \times} \quad$ carried out as in \cite[8.4. Construction of the transfer, Step 5]{CJ23-T1}, that is,
 when $\rmG$ is special as a linear algebraic group, the quotient is taken on the Zariski (or Nisnevich) site of $\rmE$, while
 when $\rmG$ is not special, it is taken on the \'etale site and  this is indicated by the {\it superscript `et'} on $\times_{\rmG}$. 
 { In addition, one needs to modify the model structures as discussed in Construction ~\ref{more.rigidity} to properly address rigidity when the group $\rmG$ is not special.}
 \end{enumerate}
 \vskip .1cm
The commutativity of the top square is clear, while the homotopy commutativity of the triangle below
it results from Proposition ~\ref{diagonal.map}(i), once we make use of the identification in ${\widetilde \SH}^{\rmG}(k_{mot})$:
\[\tau\rmA = \rmD(\mbS^{\rmG} \wedge \rmX/\rmU) \wedge (\mbS^{\rmG} \wedge \rmX/\rmU) \cong \rmD(\mbS^{\rmG} \wedge \rmX^h_{\rmZ}/\rmX^h_{\rmZ}- \rmZ) \wedge (\mbS^{\rmG} \wedge \rmX^h_{\rmZ}/\rmX^h_{\rmZ}- \rmZ).\]

{\it Then the composition of maps in the left column represents the transfer $\tr_{\rmX/\rmU}$.}
\footnote{It may be worth recalling from \cite[(8.2.1)]{CJ23-T1}, that the co-evaluation map is a composition of several maps,
with at least some of them going in the wrong direction, but where such wrong-way maps are all weak-equivalences and also equivariant. Since these 
wrong way maps are all weak-equivalences and equivariant, the construction of the transfer from the pre-transfer proceeds as if these wrong-way maps are not present.}
\vskip .1cm
Let $\rmM$ denote a motivic spectrum, and let $\RHom$ denote the derived functor of the internal $\Hom$ in the category ${{\Spt}}(\k_{\rm mot})$. On applying the functor $\RHom(\quad, \rmM)$ to the diagram ~\eqref{factor.diagram}, we obtain the
 homotopy commutative diagram:
 \vskip .1cm 
\fontsize{9}{12}
\be \begin{equation}
   \label{factor.diagram.1}
   \xymatrix{{\RHom(\{(\rmE{\underset {\rmG} \times}(\mbS^{\rmG}\wedge \rmX_+)^{\rmV}\wedge _{\rmB}S(\eta^{\rmV}\oplus 1))/s|\rmV\}, \rmM)} \ar@<1ex>[r]^{(\tilde i^h \wedge id)^{*}} \ar@<1ex>[d]^{(e \wedge id)^*} & {\RHom(\{(\rmE{\underset {\rmG} \times}(\mbS^{\rmG}\wedge \rmX^h_{Z_+})^{\rmV}\wedge _{\rmB}S(\eta^{\rmV}\oplus 1))/s|\rmV\}, \rmM)}\ar@<1ex>[d]^{(e \wedge id)^*} \\
              {\RHom(\{(\rmE{\underset {\rmG} \times}(\tau \rmA \wedge \mbS^{\rmG}\wedge \rmX_+)^{\rmV}\wedge _{\rmB}S(\eta^{\rmV}\oplus 1))/s|\rmV\}, \rmM)} \ar@<1ex>[r]^{(\tilde i^h \wedge id)^{*}} \ar@<1ex>[d]^{id \wedge \Delta^* \wedge id} & {\RHom(\{(\rmE{\underset {\rmG} \times}(\tau \rmA \wedge \mbS^{\rmG}\wedge \rmX^h_{Z_+})^{\rmV}\wedge _{\rmB}S(\eta^{\rmV}\oplus 1))/s|\rmV\}, \rmM)} \ar@<1ex>[dl]^{id \wedge {\Delta^h \wedge id}^*}\\
              { \RHom(\{(\rmE{\underset {\rmG} \times}(\tau \rmA)^{\rmV}\wedge _{\rmB}S(\eta^{\rmV}\oplus 1))/s|\rmV\}, \rmM)} \ar@<1ex>[d]^{\tau^* \wedge id}\\
              {\RHom(\{(\rmE{\underset {\rmG} \times}\rmA^{\rmV} \wedge _{\rmB}S(\eta^{\rmV}\oplus 1))/s|\rmV\}, \rmM)} \ar@<1ex>[d]^{c^*\wedge id}\\
              {\RHom (\{(\rmE{\underset {\rmG} \times}\mbS^{\rmG, \rmV} \wedge _{\rmB}S(\eta^{\rmV}\oplus 1))/s|\rmV \}, \rmM )}}
\end{equation} \ee
\normalsize
Next assume that the spectrum $\rmM$ has the rigidity property as in Definition ~\ref{rigid.prop}. \index{rigidity}
 In view of the assumption that the spectrum $\rmM$ has the rigidity property, we then obtain the weak-equivalences (see Proposition ~\ref{diagonal.map}(iii)):
\fontsize{9}{12}
 \[{\RHom(\{(\rmE{\underset {\rmG} \times}(\tau \rmA \wedge \mbS^{\rmG}\wedge \rmZ_+)^{\rmV} \wedge _{\rmB}S(\eta^{\rmV}\oplus 1))/s|\rmV\}, \rmM)} \simeq {\RHom(\{(\rmE{\underset {\rmG} \times}(\tau \rmA \wedge \mbS^{\rmG}\wedge \rmX_{\rmZ,+}^h)^{\rmV} \wedge _{\rmB}S(\eta^{\rmV}\oplus 1))/s|\rmV\}, \rmM)} \mbox{ and }\]
\[{\RHom(\{(\rmE{\underset {\rmG} \times}( \mbS^{\rmG}\wedge \rmZ_+)^{\rmV} \wedge _{\rmB}S(\eta^{\rmV}\oplus 1))/s|\rmV\}, \rmM)} \simeq {\RHom(\{(\rmE{\underset {\rmG} \times}( \mbS^{\rmG}\wedge \rmX_{\rmZ,+}^h)^{\rmV} \wedge _{\rmB}S(\eta^{\rmV}\oplus 1))/s|\rmV\}, \rmM).} \]
\normalsize

Therefore, it suffices to show that the composite map
\be \begin{equation}
     \label{id.tr.Z}
  (c \wedge id)^*  \circ  (\tau \wedge id)^* \circ (id \wedge \Delta^h \wedge id)^* \circ  (e \wedge id)^* 
\end{equation} \ee
identifies with the transfer $\tr_{\rmZ}^*\circ  i^{h*}$, where $i^h: \rmZ \ra \rmX^h_{\rmZ}$ is the obvious closed immersion. We proceed to prove this identification.
\vskip .1cm
Next, observe (see Definition ~\ref{pretransfer.3}) that the pre-transfer $\tr_{\rmX^h_{\rmZ}}'$ is given by the composite map:
\be \begin{multline}
   \label{trXh}
   \begin{split}
\tr_{\rmX^h_{\rmZ}}':\mbS^{\rmG} {\overset c \ra} \mbS^{\rmG} \wedge (\rmX^h_{\rmZ}/(\rmX^h_{\rmZ}-\rmZ)) \wedge \rmD(\mbS^{\rmG} \wedge \rmX^h_{\rmZ}/(\rmX^h_{\rmZ}-\rmZ)) {\overset {\tau } \ra} \rmD(\mbS^{\rmG} \wedge \rmX^h_{\rmZ}/(\rmX^h_{\rmZ}-\rmZ)) \wedge  (\mbS^{\rmG} \wedge (\rmX^h_{\rmZ}/\rmX^h_{\rmZ}-\rmZ))\\
{\overset {id \wedge \Delta} \ra} \rmD(\mbS^{\rmG} \wedge \rmX^h_{\rmZ}/(\rmX^h_{\rmZ}-\rmZ)) \wedge  (\mbS^{\rmG} \wedge (\rmX^h_{\rmZ}/(\rmX^h_{\rmZ}-\rmZ))) \wedge (\mbS^{\rmG} \wedge \rmX^h_{\rmZ}) {\overset {e \wedge id} \ra} \mbS^{\rmG} \wedge \rmX^h_{\rmZ},
 \end{split} 
 \end{multline} \ee
\vskip .1cm \noindent
while the pre-transfer $\tr_{\rmZ}'$ is given by the composite map:
\be \begin{equation}
   \label{trZ}
\tr_{\rmZ}': \mbS^{\rmG} {\overset c \ra} (\mbS^{\rmG} \wedge \rmZ_+) \wedge \rmD(\mbS^{\rmG} \wedge \rmZ_+) {\overset {\tau } \ra} \rmD(\mbS^{\rmG} \wedge \rmZ_+) \wedge  (\mbS^{\rmG} \wedge \rmZ_+) {\overset {id \wedge \Delta} \ra} \rmD(\mbS^{\rmG} \wedge \rmZ_+) \wedge  (\mbS^{\rmG} \wedge \rmZ_+) \wedge (\mbS^{\rmG} \wedge \rmZ_+) {\overset {e \wedge id} \ra} \mbS^{\rmG} \wedge \rmZ_+.
\end{equation} \ee
\vskip .2cm
Next, let $\rmU \ra \rmB $ denote an open cover of $\rmB$ in the given topology, (that is either the Zariski or the \'etale topology),
so that $\rmp: \rmE \ra \rmB$ restricted to each $\rmU$ is trivial. Let $\rmU_{\bullet} = cosk_0^{\rmB}(\rmU)$ denote the corresponding $\check{\rm C}$ech-hypercover
of $\rmB$. Then 
\be \begin{equation}
  \label{EP}
(\rmE{\underset {\rmG} \times}\rmP)_{|\rmU_{\bullet}} = \rmU_{\bullet} {\underset {\rmB} \times} (\rmE{\underset {\rmG} \times}\rmP) = (\rmU_{\bullet} {\underset {\rmB} \times} \rmE){\underset {\rmG} \times}\rmP = (\rmU_{\bullet}\times {\rmG}) {\underset {\rmG} \times}\rmP \cong \rmU_{\bullet} \times \rmP,
\end{equation} \ee
for any $\rmG$-equivariant simplicial presheaf $\rmP$. Observing that the map $\rmU_{\bullet} = cosk_0^{\rmB}(\rmU) \ra \rmB$ of
simplicial presheaves is a weak-equivalence
(stalk-wise) and that the map $\rmE{\underset {\rmG} \times}\rmP \ra \rmB$, being locally trivial is a local fibration, one sees that the
induced map $(\rmE{\underset {\rmG} \times}\rmP)_{|\rmU_{\bullet}} \ra \rmE{\underset {\rmG} \times}\rmP$ of simplicial presheaves is also a weak-equivalence stalk-wise.
\vskip .2cm
Next we apply the functor $(\rmE{\underset {\rmG} \times} \quad)_{|\rmU_{\bullet}}$
to both $\tr_{\rmX^h_{\rmZ}}'$ and $\tr_{\rmZ}'$. Making use of the isomorphism in ~\eqref{EP}, we see that each spectrum that shows
up in the definition of $(\rmE{\underset {\rmG} \times} \tr_{\rmX^h_{\rmZ}}')_{|\rmU_{\bullet}}$ and $(\rmE{\underset {\rmG} \times} \tr_{\rmZ}')_{|\rmU_{\bullet}}$
is of the form $\rmU_{\bullet} \times \rmP$, for some suitable spectrum $\rmP$.
\vskip .1cm
Therefore, adopting the terminology used in ~\eqref{factor.diagram}, the remaining part of the proof of (iii) is to show that there are 
maps from each of the spectra appearing in the definition of 
\[\RHom(\{(\rmE{\underset {\rmG} \times} ({\it i^{h}}\circ \tr_{\rmZ}')^{\rmV}\wedge _{\rmB}\rmS(\eta^{\rmV}\oplus 1))/s_{|\rmU_{\bullet}}|\rmV\}, \rmM)\]
as in ~\eqref{trZ} to the corresponding spectrum appearing in the definition of 
\[\RHom(\{(\rmE{\underset {\rmG} \times} \tr_{\rmX^h_{\rmZ}}'^{\rmV}\wedge _{\rmB}\rmS(\eta^{\rmV}\oplus 1))/s_{|\rmU_{\bullet}}|\rmV\}, \rmM)\]
as in ~\eqref{trXh}, which are weak-equivalences and identifies 
$\RHom(\{(\rmE{\underset {\rmG} \times} \tr_{\rmX^h_{\rmZ}}'^{\rmV}\wedge _{\rmB}\rmS(\eta^{\rmV}\oplus 1))/s_{|\rmU_{\bullet}}|\rmV\}, \rmM)$ with 
 $\RHom(\{(\rmE{\underset {\rmG} \times} ({\it i^{h}} \circ\tr_{\rmZ}')^{\rmV}\wedge _{\rmB}\rmS(\eta^{\rmV}\oplus 1))/s_{|\rmU_{\bullet}}|\rmV\}, \rmM)$.
 This is discussed in Proposition ~\ref{compare.transf} below. 
In view of the construction of the transfer, starting with the pre-transfer as in \cite[8.4. Construction of the transfer]{CJ23-T1},  this completes
the proof of Theorem ~\ref{add.transf}(iii).

 \vskip .3cm
The statements in Theorem ~\ref{add.transf}(iv) now follow from the  statements (ii) and (iii) using ascending induction on the number of strata. 
However, as this induction needs to be handled carefully, we proceed to provide an outline of the
relevant argument. We will assume that the stratification of $\rmX$ defines the following increasing
 filtrations:
 \vskip .1cm
(a) $\phi=\rmX_{-1} \subseteq \rmX_0 \subseteq \cdots \subseteq \rmX_n =\rmX$ where each $\rmX_i$ is closed
and the strata $\rmX_i - \rmX_{i-1}$, $i=0, \cdots, n$ are smooth (regular).
\vskip .1cm
(b) $\rmU_0 \subseteq \rmU_1 \subseteq \cdots \subseteq \rmU_{n-1} \subseteq \rmU_n =\rmX$, where
each $\rmU_i$ is open in $\rmX$ (and therefore smooth (regular)) with $\rmU_i - \rmU_{i-1} = \rmX_{n-i} - \rmX_{n-i-1}$, for all
$i=0, \cdots n$. We let $j_k: \rmU_k \ra \rmX$ and $j_k': \rmU_{k} \ra \rmU_{k+1}$ denote the corresponding
 open immersions while  $i_k': \rmX_k - \rmX_{k-1} \ra \rmX - \rmX_{k-1}$
 denotes the corresponding closed immersion and $i_k: \rmX_k - \rmX_{k-1} \ra \rmX$ denotes the corresponding locally closed immersion.
\vskip .1cm
We now apply Theorem ~\ref{add.transf}(ii) with $\rmU = \rmU_{n-1}$, and $\rmZ = \rmU_n - \rmU_{n-1} = \rmX_0 - \rmX_{-1} = \rmX_0$, the closed stratum.
Since $\rmX$ is now smooth(regular) and so is $\rmZ$, the hypotheses of Theorem ~\ref{add.transf}(ii) 
are satisfied. This provides us
\be \begin{align}
  \label{step1.iterative.add}
tr_{\rmX}'^{\rmG} &= j_{n-1}\circ tr_{\rmU_{n-1}}'^{\rmG} + tr_{\rmX/\rmU_{n-1}}'^{\rmG}, \\
tr_{\rmX} = j_{n-1}\circ tr_{\rmU_{n-1}} &+ tr_{\rmX/\rmU_{n-1}},  \mbox{ and } \tau_{\rmX}^{\rmG} = \tau_{\rmU_{n-1}}^{\rmG} + \tau_{\rmX/\rmU_{n-1}}^{\rmG}. \notag
\end{align} \ee
\vskip .1cm
Similarly applying Theorem ~\ref{add.transf}(iii) with $\rmU = \rmU_{n-1}$, and $\rmZ = \rmU_n - \rmU_{n-1} = \rmX_0 - \rmX_{-1} = \rmX_0$, we obtain:
\be \begin{equation}
  \label{step2.iterative.add}
 \tr_{\rmX/\rmU_{n-1}} =  i_{\rm 0}\circ \tr_{\rmX_0}.
\end{equation} \ee
\vskip .1cm 
Next we replace $\rmX$ by $\rmU_{n-1}$, $\rmU$ by $\rmU_{n-2}$ and $\rmZ$ by $\rmX_1- \rmX_0$. 
Since $\rmX_1 - \rmX_0$ is smooth (regular), Theorem ~\ref{add.transf}(ii) and (iii) now provide us
\be \begin{align}
  \label{step3.iterative.add}
tr_{\rmU_{n-1}} = j_{n-2}' \circ tr_{\rmU_{n-2}} + i_1' \circ tr_{\rmX_1 - \rmX_0}.
\end{align} \ee
\vskip .1cm
Substituting these in ~\eqref{step1.iterative.add}, we obtain
\[\quad tr_{\rmX} = j_{n-2} \circ tr_{\rmU_{n-2}} + i_1 \circ tr_{\rmX_1- \rmX_0} +i_0 \circ tr_{\rmX_0}. \]
Clearly this may be continued inductively to deduce statement (iv) in Theorem ~\ref{add.transf}
 from Theorem ~\ref{add.transf}(ii) and (iii). 
 \vskip .1cm
 Finally, the proof of Theorem ~\ref{add.transf}(v) follows from the compatibility of the 
 pre-transfer with \'etale realization
 as shown in  \cite[Proposition 4.1 and Corollary 4.2]{CJ23-T2}. Therefore, it is immediate that one obtains the
 corresponding statements for the $\rmG$-equivariant pre-transfer. 
 The corresponding statements for the transfer in (i) and (ii) then follow readily from these statements for the
 $\rmG$-equivariant pre-transfer. In order to obtain the corresponding statements for the transfer in (iii) and (iv), one
 needs to invoke Proposition ~\ref{diagonal.map}(iv). For groups that are special, one may also see these
 more directly, as $\epsilon ^*(tr(\rmf)) = {\it tr}(\epsilon ^*(\rmf))$: see  \cite[Corollary 4.2]{CJ23-T2}. Therefore, for groups
 that are special, we may also obtain the corresponding results for the \'etale version of the transfer by simply applying the pull-back 
 functor $\epsilon^*$ to the \'etale site.
 \vskip .1cm
 This completes the proof of Theorem ~\ref{add.transf}, modulo the results of Proposition ~\ref{compare.transf} and Lemmas ~\ref{Cech.hypercover}, ~\ref{loc.extend}. 
 \qed
\end{proof}

\begin{lemma}
 \label{Cech.hypercover} \index{Cech hypercover}
 \begin{enumerate}[\rm(i)]
 \item Let $u:\rmU \ra \rmB$ denote an open cover in either the Zariski, Nisnevich, or \'etale topologies, and let $u_{\bullet} : \rmU_{\bullet}= cosk_0^{\rmB}(\rmU) \ra \rmB$
 denote the corresponding  {C}ech-hypercover. Given a simplicial presheaf $P$ on $\rmB$, $\{u_{n\#}u_n^*(P)|n\}$ is a simplicial resolution of
 $P$, in the sense, $\hocolimD \{u_{n\#}u_n^*(P)|n\} \simeq P$.  Therefore, for any simplicial presheaves $F$, and $P$, the natural map
 \[\RHom(P, F) \ra \holimD \{\RHom(\{u_{n\#}u_n^*(P)|n\}, F)|n\} = \holimD \{\RHom(u_{n}^*(P), u_n^*(F))|n\}\]
 is a weak-equivalence, where $\RHom$ denotes
  the derived internal Hom in the category $\Spt(\k_{\rm mot})$ (or $\Spt(\k_{et})$).
 
 \vskip .1cm
 \item Assume $\rmG$ is special and let $\rmM$ denote a motivic spectrum that has the rigidity property. Let $\rmU$, $\rmU_{\bullet}$ be as in (i). 
 Let 
 $(\rmE\times_{\rmG}\rmX^h_{\rmZ})_{|\rmU_{\bullet}}= (\rmE\times_{\rmG}\rmX^h_{\rmZ})\times_{\rmB} \rmU_{\bullet}$, $(\rmE\times_{\rmG}{\rmZ})_{|\rmU_{\bullet}}= (\rmE\times_{\rmG}\rmZ)\times_{\rmB} \rmU_{\bullet}$.

 \vskip .1cm
  Then, the transfer maps $\tr_{\rmX^h_{\rmZ+}}: {\rm \Sigma^{\infty}_{\T}}\rmB_+ \ra {\rm \Sigma^{\infty}_{\T}}(\rmE\times _{\rmG}\rmX^h_{\rmZ})_+$ and $\tr_{\rmZ}:{\rm \Sigma^{\infty}_{\T}}\rmB_+ \ra {\Sigma^{\infty}_{\T}}(\rmE\times_{\rmG} \rmZ)_+ $
  provide the homotopy commutative diagram
 \[\xymatrix{ {\RHom({\Sigma^{\infty}_{\T}}((\rmE\times_{\rmG}\rmX^h_{\rmZ})_{|\rmU_{\bullet}+}, \rmM)} \ar@<1ex>[r]^{\tr_{\rmX^h_{\rmZ}}|\rmU_{\bullet}} \ar@<-1ex>[d]_{\simeq} & {\RHom({\Sigma^{\infty}_{\T}}\rmB_{\rmU_{\bullet}+}, \rmM)} \ar@<-1ex>[d]_{id}\\
              {\RHom({\Sigma^{\infty}_{\T}}((\rmE\times_{\rmG}\rmZ)_{|\rmU_{\bullet}+}, \rmM)}  \ar@<1ex>[r]^{\tr_{\rmZ}|\rmU_{\bullet}} & {\RHom({\Sigma^{\infty}_{\T}}\rmB_{\rmU_{\bullet}+}, \rmM)} \\
              {\RHom({\Sigma^{\infty}_{\T}}((\rmE\times_{\rmG}\rmZ)_+, \rmM)} \ar@<1ex>[u]^{\simeq} \ar@<1ex>[r]^{\tr_{\rmZ}} & {\RHom({\Sigma^{\infty}_{\T}}\rmB_{+}, \rmM)} \ar@<1ex>[u]^{\simeq} \\
              {\RHom({\Sigma^{\infty}_{\T}}((\rmE\times_{\rmG}\rmX^h_{\rmZ})_+, \rmM)}  \ar@/^7pc/[uuu]^{\simeq} \ar@<1ex>[r]^{\tr_{\rmX^h_{\rmZ}}} \ar@<1ex>[u] & {\RHom({\Sigma^{\infty}_{\T}}\rmB_{+}, \rmM)} \ar@<1ex>[u]^{id} }
              \]
 \vskip .1cm \noindent
 with the vertical maps being weak-equivalences.
 \vskip .1cm
 When $\rmG$ is not special, $(\rmE\times_{\rmG}\rmX^h_{\rmZ})_{|\rmU_{\bullet}}$ will denote $\rmR\epsilon_*((\rmE\times_{\rmG}^{et}\epsilon^*(\rmX^h_{\rmZ}))\times_{\rmB}^{et} \rmU_{\bullet})$ and 
 $(\rmE\times_{\rmG}{\rmZ})_{|\rmU_{\bullet}}$ will denote $\rmR\epsilon_*((\rmE\times_{\rmG}^{et}\epsilon^*(\rmZ))\times_{\rmB}^{et} \rmU_{\bullet})$,  
 with $\rmU_{\bullet}$ denoting a Cech hypercover in the \'etale topology. 
 Similarly, $\rmB$ ( $\rmB_{|\rmU_{\bullet}}$) denotes  $\rmR\epsilon_*(\rmB)$ ($\rmR\epsilon_*(\rmB_{|\rmU_{\bullet}})$, \res). Then we obtain a corresponding diagram with
 the suspension by $\mbS_{\k}$ replaced by suspension by $R\epsilon_*(\epsilon^*(\mbS_{\k}))$.
 Moreover the pull-back $\epsilon^*$ and the derived push-forward $\rmR\epsilon_*$
 make use of the model structures discussed in Construction ~\ref{more.rigidity}.
 
 \end{enumerate}
\end{lemma}
\vskip .1cm \noindent
\begin{proof}  The proof of (i) is straightforward. In case the group $\rmG$ is special, (ii) follows readily from (i) and the construction of the transfer from
 the pre-transfer as in \cite[8.4. Construction of the transfer]{CJ23-T1}. In case the group $\rmG$ is not special, the fact that the map denoted by the curly arrow on the left and
 the middle and top vertical maps on the left are weak-equivalences makes use of the model structures discussed in Construction ~\ref{more.rigidity}. Since the diagram 
 homotopy commutes, it follows that the bottom vertical map on the left is also a weak-equivalence.
\end{proof}

\begin{proposition}
 \label{compare.transf}
 Assume the base field $\k$ is infinite and contains a $\sqrt{-1}$. \index{$\sqrt{-1}$}
Let $\rmM$ denote a fibrant spectrum that has the rigidity property as in Definition ~\ref{rigid.prop}. Let $\rmU \ra \rmB$ denote a covering over which 
the torsor $\rmp: \rmE \ra \rmB$ is trivial. Let $\rmU_{\bullet}$ denote the corresponding $\check{\rm C}$ech-hypercover of $\rmB$.
\vskip .1cm
Then, adopting the terminology used in ~\eqref{factor.diagram}, there are maps from each of the spectra appearing in the definition of 
\[\RHom(\{(\rmE{\underset {\rmG} \times} \tr_{\rmX^h_{\rmZ}}'^{\rmV}\wedge _{\rmB}\rmS(\eta^{\rmV}\oplus 1))/s_{|\rmU_{\bullet}}|\rmV\}, \rmM)\]
as in ~\eqref{trXh} to the corresponding spectrum appearing in the definition of 
\[\RHom(\{(\rmE{\underset {\rmG} \times} ({\it i^{h}} \circ \tr_{\rmZ}')^{\rmV}\wedge _{\rmB}\rmS(\eta^{\rmV}\oplus 1))/s_{|\rmU_{\bullet}}|\rmV\}, \rmM)\]
as in ~\eqref{trZ}, which are weak-equivalences, and identifies 
\[\RHom(\{(\rmE{\underset {\rmG} \times} \tr_{\rmX^h_{\rmZ}}'^{\rmV}\wedge _{\rmB}\rmS(\eta^{\rmV}\oplus 1))/s_{|\rmU_{\bullet}}|\rmV\}, \rmM) \mbox{ with } \RHom(\{(\rmE{\underset {\rmG} \times} ({\it i^{h}} \circ \tr_{\rmZ}')^{\rmV}\wedge _{\rmB}\rmS(\eta^{\rmV}\oplus 1))/s_{|\rmU_{\bullet}}|\rmV\}, \rmM).\]
\end{proposition}
\begin{proof} We will first discuss the proof when $\rmU \ra \rmB$ is a Zariski open cover, and then outline the changes needed if this is an
\'etale cover. For this, we begin with the following observations:
\begin{enumerate}[\rm(i)]
\item The Spanier-Whitehead dual $\rmD$ above is taken in the model category
${\widetilde {\Spt}}^{\group}(\k_{\rm mot})$ which is Quillen equivalent to the model category $\Spt$ of non-equivariant spectra. In fact,
the spectrum $\mbS^{\rmG}$ in the category  ${\widetilde {\Spt}}^{\group}$ identifies with the non-equivariant sphere spectrum ${\Sigma^{\infty}_{\T}}$, but re-indexed
by the Thom-spaces $\{\rmT_{\rmV}|\rmV\}$ of finite dimensional $\rmG$ representations of $\rmG$. (See \cite[Proposition 6.2]{CJ23-T1}.)
\item It suffices to obtain the above identification as maps of {\it sheaves} on the appropriate site of the base scheme $\rmB$ on which the
torsor $\rmp: \rmE \ra \rmB$ is locally trivial. 
\end{enumerate}
Therefore, and in view of Lemma ~\ref{Cech.hypercover} and ~\eqref{EP}, we apply the functor $(\rmE{\underset {\rmG} \times} \quad)_{|\rmU_{\bullet}}$
to both $\tr_{\rmX^h_{\rmZ}}'$ and $\tr_{\rmZ}'$. Making use of the isomorphism in ~\eqref{EP}, we see that each spectrum that shows
up in the definition of $(\rmE{\underset {\rmG} \times} \tr_{\rmX^h_{\rmZ}}')_{|\rmU_{\bullet}}$ and $(\rmE{\underset {\rmG} \times} \tr_{\rmZ}')_{|\rmU_{\bullet}}$
is of the form $\rmU_{\bullet} \times \rmP$, for some suitable spectrum $\rmP$. 
In particular, we may  
replace the equivariant sphere spectrum $\mbS^{\rmG}$ appearing in $(\rmE{\underset {\rmG} \times} \tr_{\rmZ}')_{|\rmU_{\bullet}}$ and
$(\rmE{\underset {\rmG} \times} \tr_{\rmX^h_{\rmZ}}')_{|\rmU_{\bullet}}$ by the non-equivariant sphere spectrum $\mbS_{k}$. 
{\it Moreover, now it suffices to prove
that one can identify $\tr_{\rmX^h_{\rmZ}}'$ and $i^h \circ\tr_{\rmZ}'$, where $i^h: \rmZ \ra \rmX^h_{\rmZ}$ denotes the
obvious closed immersion.}
\vskip .2cm
Let $\cN$ denote the normal bundle associated to the closed immersion $\rmZ \ra \rmX$.  
Next, one may take a Zariski open cover $\rmW \ra\rmZ$ so that the normal bundle $\cN_{|\rmW}$ is trivial. 
 Then one observes the isomorphism
 \be \begin{equation}
 \label{loc.XhZ}
 (\rmX^h_{\rmZ}/(\rmX^h_{\rmZ}-\rmZ))_{|\rmW_{\bullet}} \cong \rmW_{\bullet,+} \wedge \T^{\rmc},
 \end{equation} \ee
 where $\rmc$ denotes the codimension of $\rmZ$ in $\rmX$ and $\rmW_{\bullet} = cosk_0^{\rmZ}(\rmW)$ is the corresponding $\check{\rm C}$ech hypercover. Then the compatibility of the co-evaluation map with products
 (see the proof of Proposition ~\ref{mult.prop})
 shows the following diagram commutes, where the vertical maps are induced by the co-evaluation map
 $c: \mbS_{k} \ra \T^c \wedge \rmD(\T^{\rmc})$:
 \fontsize{9}{12}
 \be \begin{equation}
 \label{big.diagm.0}
  \xymatrix{{\mbS_{\k}} \ar@<1ex>[r]^(.2){c}  &  {{\Sigma^{\infty}_{\T}}  (\rmX^h_{\rmZ}/(\rmX^h_{\rmZ}-\rmZ))_{|\rmW_{\bullet}} \wedge \rmD({\Sigma^{\infty}_{\T}}  \rmX^h_{\rmZ}/(\rmX^h_{\rmZ}-\rmZ)_{|\rmW_{\bullet}})} \ar@<1ex>[r]^{\tau} & {\rmD({\Sigma^{\infty}_{\T}}  \rmX^h_{\rmZ}/(\rmX^h_{\rmZ}-\rmZ)_{|\rmW_{\bullet}}) \wedge  ({\Sigma^{\infty}_{\T}}  (\rmX^h_{\rmZ}/\rmX^h_{\rmZ}-\rmZ)_{|\rmW_{\bullet}})}\\
             {{\mbS_{\k}}} \ar@<1ex>[u]^{id} \ar@<1ex>[r]^c  & {{\Sigma^{\infty}_{\T}}  \rmW_{\bullet,+} \wedge \rmD({\Sigma^{\infty}_{\T}} \rmW_{\bullet,+})} \ar@<1ex>[u] \ar@<1ex>[r]^{\tau} & {\rmD({\Sigma^{\infty}_{\T}}  {\rmW_{\bullet,+}}) \wedge  ({\Sigma^{\infty}_{\T}}  {\rmW_{\bullet,+}})} \ar@<1ex>[u] .}
  \end{equation} \ee
  \normalsize
Varying $\W_n$ in the above hypercover, the above diagram is one of cosimplicial-simplicial objects of spectra. One takes the
homotopy colimit followed by the homotopy inverse limit of the above diagram to obtain the commutative diagram (as in the proof of Lemma ~\ref{loc.extend} below):
\be \begin{equation}
\label{big.diagm.1}
  \xymatrix{{\mbS_{\k}} \ar@<1ex>[r]^(.2){c}  &  {{\Sigma^{\infty}_{\T}}  (\rmX^h_{\rmZ}/(\rmX^h_{\rmZ}-\rmZ)) \wedge \rmD({\Sigma^{\infty}_{\T}}  \rmX^h_{\rmZ}/(\rmX^h_{\rmZ}-\rmZ))} \ar@<1ex>[r]^{\tau} & {\rmD({\Sigma^{\infty}_{\T}}  \rmX^h_{\rmZ}/(\rmX^h_{\rmZ}-\rmZ)) \wedge  ({\Sigma^{\infty}_{\T}}  (\rmX^h_{\rmZ}/\rmX^h_{\rmZ}-\rmZ))}\\
             {\mbS_{\k}} \ar@<1ex>[u]^{id} \ar@<1ex>[r]^c    & {{\Sigma^{\infty}_{\T}}  \rmZ_+ \wedge \rmD({\Sigma^{\infty}_{\T}}  \rmZ_+)} \ar@<1ex>[u] \ar@<1ex>[r]^{\tau} & {\rmD({\Sigma^{\infty}_{\T}}  \rmZ_+) \wedge  ({\Sigma^{\infty}_{\T}}  \rmZ_+)} \ar@<1ex>[u] .}
  \end{equation} \ee
\vskip .2cm
 Next one observes the identifications:
\be \begin{align}
\label{can.ident}
{({\Sigma^{\infty}_{\T}}  (\rmX^h_{\rmZ}/\rmX^h_{\rmZ}-\rmZ)) \wedge \rmD({\Sigma^{\infty}_{\T}}  \rmX^h_{\rmZ}/(\rmX^h_{\rmZ}-\rmZ))  } &\simeq \RHom ({\Sigma^{\infty}_{\T}}  \rmX^h_{\rmZ}/(\rmX^h_{\rmZ}-\rmZ), {\Sigma^{\infty}_{\T}}  \rmX^h_{\rmZ}/(\rmX^h_{\rmZ}-\rmZ)) \mbox{ and}\\
{\Sigma^{\infty}_{\T}} \rmZ_+ \wedge \rmD({\Sigma^{\infty}_{\T}}  \rmZ_+) &\simeq \RHom({\Sigma^{\infty}_{\T}} \rmZ_+, {\Sigma^{\infty}_{\T}}  \rmZ_+,) \notag
\end{align} \ee
where $\RHom$ denotes the derived internal hom in the above category of spectra.  
 Then Lemma ~\ref{loc.extend} below, enables
 us to define a map
\[ \RHom({\Sigma^{\infty}_{\T}}\rmZ_+, {\Sigma^{\infty}_{\T}}\rmZ_+) \ra \RHom({\Sigma^{\infty}_{\T}}\rmX^h_{\rmZ}/(\rmX^h_{\rmZ}-\rmZ), {\Sigma^{\infty}_{\T}}\rmX^h_{\rmZ}/(\rmX^h_{\rmZ}-\rmZ))\]
which is shown there to be a weak-equivalence. This map is also defined by restricting to the hypercover $\rmW_{\bullet}$ and
we skip the verification that it is compatible with the middle vertical map in ~\eqref{big.diagm.1} under the identifications
in ~\eqref{can.ident}. This also proves therefore that the middle vertical map in ~\eqref{big.diagm.1} is a weak-equivalence.
One may prove the right vertical map in ~\eqref{big.diagm.1} is also a weak-equivalence since it is obtained
by applying the permutation of the two factors in the middle column.
\vskip .1cm
In view of the commutative square (where the vertical maps are again induced by the map $i^h$),
\[\xymatrix{{\rmX^h_{\rmZ}/(\rmX^h_{\rmZ}-\rmZ)} \ar@<1ex>[r]^{\Delta} & {\rmX^h_{\rmZ}/(\rmX^h_{\rmZ}-\rmZ) \wedge \rmX^h_{\rmZ,+}}\\
            {\rmZ_+} \ar@<1ex>[u] \ar@<1ex>[r]^{\Delta} & {\rmZ_+ \wedge \rmZ_+} \ar@<1ex>[u] }
\]
we next obtain the commutative diagram:
\be \begin{equation}
   \label{big.diagm.2}
   \xymatrix{{\rmD({\Sigma^{\infty}_{\T}}  \rmX^h_{\rmZ}/(\rmX^h_{\rmZ}-\rmZ) \wedge  {\Sigma^{\infty}_{\T}}  \rmX^h_{\rmZ}/(\rmX^h_{\rmZ}-\rmZ)} \ar@<1ex>[r]^{id \wedge \Delta} & {\rmD({\Sigma^{\infty}_{\T}}  \rmX^h_{\rmZ}/(\rmX^h_{\rmZ}-\rmZ) \wedge  {\Sigma^{\infty}_{\T}}  \rmX^h_{\rmZ}/(\rmX^h_{\rmZ}-\rmZ) \wedge {\Sigma^{\infty}_{\T}}  \rmX^h_{\rmZ,+}}\\
             {\rmD({\Sigma^{\infty}_{\T}}Z_+) \wedge {\Sigma^{\infty}_{\T}}Z_+} \ar@<1ex>[u] \ar@<1ex>[r]^{id \wedge \Delta} &  {\rmD({\Sigma^{\infty}_{\T}}Z_+) \wedge {\Sigma^{\infty}_{\T}}Z \wedge {\Sigma^{\infty}_{\T}}Z_+} \ar@<1ex>[u] .}
    \end{equation} \ee
(In fact, one may start with the commutative diagram:
\fontsize{9}{12}
\be \begin{equation}
   \label{big.diagm.2.1}
   \xymatrix{{\rmD({\Sigma^{\infty}_{\T}}  (\rmX^h_{\rmZ}/(\rmX^h_{\rmZ}-\rmZ))_{|\rmW_{\bullet}}) \wedge  {\Sigma^{\infty}_{\T}}  (\rmX^h_{\rmZ}/(\rmX^h_{\rmZ}-\rmZ))_{|\rmW_{\bullet}}} \ar@<1ex>[r]^{id \wedge \Delta} & {\rmD({\Sigma^{\infty}_{\T}}  (\rmX^h_{\rmZ}/(\rmX^h_{\rmZ}-\rmZ))_{|\rmW_{\bullet}} \wedge  {\Sigma^{\infty}_{\T}}  (\rmX^h_{\rmZ}/(\rmX^h_{\rmZ}-\rmZ))_{|\rmW_{\bullet}} \wedge {\Sigma^{\infty}_{\T}}  (\rmX^h_{\rmZ,+})_{|\rmW_{\bullet}}} \notag\\
             {\rmD({\Sigma^{\infty}_{\T}}{\rmW_{\bullet}}_+) \wedge {\Sigma^{\infty}_{\T}}{\rmW_{\bullet}}_+} \ar@<1ex>[u] \ar@<1ex>[r]^{id \wedge \Delta} &  {\rmD({\Sigma^{\infty}_{\T}}{\rmW_{\bullet}}_+) \wedge {\Sigma^{\infty}_{\T}}{\rmW_{\bullet}} \wedge {\Sigma^{\infty}_{\T}}{\rmW_{\bullet}}_+} \ar@<1ex>[u] .}
    \end{equation} \ee
\normalsize
and take the homotopy colimit followed by the homotopy limit to obtain the commutative diagram in ~\eqref{big.diagm.2}.)
\vskip .2cm    
We next consider the square:
\be \begin{equation}
  \label{big.diagm.3}
  \xymatrix{{\rmD({\Sigma^{\infty}_{\T}}  \rmX^h_{\rmZ}/(\rmX^h_{\rmZ}-\rmZ)) \wedge  {\Sigma^{\infty}_{\T}}  \rmX^h_{\rmZ}/(\rmX^h_{\rmZ}-\rmZ) \wedge {\Sigma^{\infty}_{\T}}  \rmX^h_{\rmZ,+}} \ar@<1ex>[r]^(.7){e \wedge id} & {{\mbS_k }\wedge {\Sigma^{\infty}_{\T}}\rmX^h_{\rmZ,+}}\\
             {\rmD({\Sigma^{\infty}_{\T}}Z_+) \wedge {\Sigma^{\infty}_{\T}}Z_+ \wedge {\Sigma^{\infty}_{\T}}Z_+} \ar@<1ex>[r]^{e \wedge id} \ar@<1ex>[u] & {{\mbS_k} \wedge {\Sigma^{\infty}_{\T}}\rmZ_+} \ar@<1ex>[u]^{id \wedge {\Sigma^{\infty}_{\T}}i^h} .}
\end{equation} \ee
Clearly the commutativity of the above square will be implied by the commutativity of the square
\be \begin{equation}
  \label{big.diagm.3.1}
  \xymatrix{{\rmD({\Sigma^{\infty}_{\T}}  \rmX^h_{\rmZ}/(\rmX^h_{\rmZ}-\rmZ)) \wedge  {\Sigma^{\infty}_{\T}}  \rmX^h_{\rmZ}/(\rmX^h_{\rmZ}-\rmZ) } \ar@<1ex>[r]^(.8){e } & {{\mbS_k }}\\
             {\rmD({\Sigma^{\infty}_{\T}}Z_+) \wedge {\Sigma^{\infty}_{\T}}Z_+ } \ar@<1ex>[r]^{e } \ar@<1ex>[u] & {{\mbS_k} } \ar@<1ex>[u]^{id} ,}
\end{equation} \ee
which is implied by the commutativity of the square:
\be \begin{equation}
  \label{big.diagm.3.2}
  \xymatrix{{\rmD({\Sigma^{\infty}_{\T}}  \rmX^h_{\rmZ}/(\rmX^h_{\rmZ}-\rmZ))_{|\rmW_{\bullet}} \wedge  {\Sigma^{\infty}_{\T}}  (\rmX^h_{\rmZ}/(\rmX^h_{\rmZ}-\rmZ))_{|\rmW_{\bullet}} } \ar@<1ex>[r]^(.8){e } & {{\mbS_k}}\\
             {\rmD({\Sigma^{\infty}_{\T}}\rmW_{\bullet,_+}) \wedge {\Sigma^{\infty}_{\T}}\rmW_{\bullet,+} } \ar@<1ex>[r]^{e } \ar@<1ex>[u] & {{\mbS_k} } \ar@<1ex>[u]^{id} .}
\end{equation} \ee
In view of the identification ${\Sigma^{\infty}_{\T}}  (\rmX^h_{\rmZ}/(\rmX^h_{\rmZ}-\rmZ))_{|\rmW_{\bullet}} ={\Sigma^{\infty}_{\T}}\rmW_{\bullet, +} \wedge \T^c$,
the last diagram now identifies with
\be \begin{equation}
  \label{big.diagm.3.3}
  \xymatrix{{(\rmD({\Sigma^{\infty}_{\T}} \rmW_{\bullet,_+}) \wedge {\Sigma^{\infty}_{\T}}\rmW_{\bullet,+}) \wedge (\rmD({\Sigma^{\infty}_{\T}}\T^c) \wedge {\Sigma^{\infty}_{\T}}\T^c)} \ar@<1ex>[r]^(.8){e_{\rmW_{\bullet,+}}\wedge e_{\T^c} } & {{\mbS_k}}\\
             {(\rmD({\Sigma^{\infty}_{\T}}\rmW_{\bullet,_+}) \wedge {\Sigma^{\infty}_{\T}}\rmW_{\bullet,+}) \wedge {\mbS_k} } \ar@<1ex>[r]^{e_{\rmW_{\bullet,+} }} \ar@<1ex>[u]^{id \wedge (\tau\circ c_{\T^c})} & {{\mbS_k} } \ar@<1ex>[u]^{id} ,}
\end{equation} \ee
where $c_{\T^{\rmc}}$ ($e_{\T^{\rmc}}$) is the co-evaluation (evaluation, \res) map for $\T^{\rmc}$. The above square commutes precisely when $\tau_{\T}=1$ in $\pi_{0,0}(\mbS_k)\cong {\rm GW}(k)$, with the last isomorphism holding because we have assumed $k$ is perfect: see
\cite[Theorem 6.2.2]{Mo4}.  
\vskip .2cm
To see this, first observe that 
the multiplicative property of the trace as in Lemma ~\ref{mult.prop} shows that 
\be \begin{equation}
  \label{mult.prop.1}
  \tau_{\T^{\rmc} \wedge \rmZ_+}  = (\tau_{{\Sigma^{\infty}_{\T}}\T})^{\wedge ^{\rmc}} \wedge \tau_{{\Sigma^{\infty}_{\T}} \rmZ_+} = \tau_{{\Sigma^{\infty}_{\T}} \rmZ_+}
\end{equation} \ee
as classes in $\pi_{0,0}(\mbS_k)$, where the last identification makes use of the assumption that $\sqrt{-1} \in k$.
\vskip .1cm
In general, it is known that the class of $\tau_{{\Sigma^{\infty}_{\T}}\T}= <-1>$ in the Grothendieck-Witt group ${\rm GW}(k)$, which identifies with $\pi_{0,0}(\mbS_k)$, in view of
\cite[Theorem 6.2.2]{Mo4}. (Here it may be important to recall that $\T$ is the pointed simplicial presheaf ${\mathbb P}^1$ pointed by $\infty$.)
The assumption that $\sqrt{-1} \in k$ implies that the  quadratic form $<1, -1>$ identifies with the  quadratic form
$<1,1>$,  and the  quadratic form $<-1>$ identifies with the  quadratic form $<1>$ as classes in ${\rm GW}(k)$. (See, for example, \cite[p. 44]{Sz}.)
This implies that  $\tau_{{\Sigma^{\infty}_{\T}}\T} = \tau_{\mbS_k} =1$ in $\pi_{0,0}(\mbS_k)$. This proves ~\eqref{mult.prop.1}. \index{$\sqrt -1$}
\vskip .2cm
Therefore, the last square commutes under the assumption that the base field $\k$ contains a 
$\sqrt{-1}$. Therefore, under the same hypothesis, the
square in ~\eqref{big.diagm.3} also commutes. Observe that the composition of the maps in the top rows of the squares
~\eqref{big.diagm.1}, ~\eqref{big.diagm.2} and ~\eqref{big.diagm.3} define the pre-transfer $\tr_{\rmX^h_{\rmZ}}'$, while the
 composition of the maps in the bottom rows of the squares ~\eqref{big.diagm.1}, ~\eqref{big.diagm.2} and ~\eqref{big.diagm.3} 
 define the pre-transfer $\tr_{\rmZ}'$. Therefore, $\tr_{\rmX^h_{\rmZ}}' = i^h \circ \tr_{\rmZ}'$.
 \vskip .2cm

In view of the identification in ~\eqref{EP}, 
the product with $\rmU_{\bullet}$ of the composition of the maps forming the top rows in diagrams ~\eqref{big.diagm.1}, ~\eqref{big.diagm.2} and  ~\eqref{big.diagm.3}
defines the map $(\rmE{\underset {\rmG} \times} \tr_{\rmX^h_{\rmZ}}')_{|\rmU_{\bullet}}$ while the product with $\rmU_{\bullet}$ of the composition of the maps forming
the bottom-rows in diagrams ~\eqref{big.diagm.1}, ~\eqref{big.diagm.2} and ~\eqref{big.diagm.3}
defines the map $(\rmE{\underset {\rmG} \times} \tr_{\rmZ}')_{|\rmU_{\bullet}}$. 
Moreover, at this point, on applying the functor $\RHom(\quad, \rmM)$ 
for a  spectrum $\rmM$ in $\Spt(\k_{\rm mot})$ that
has the rigidity property to the diagrams obtained by taking the product of the above diagrams with $\rmU_{\bullet}$,
the resulting vertical maps are all  weak-equivalences. 
Therefore, this
completes the proof of the Proposition, when the given cover $\rmU$ over which $\rmp: \rmE \ra \rmB$ is trivial is a Zariski open cover.
\vskip .2cm
Next we consider the case where the above cover $\rmU$ is an \'etale cover of $\rmB$. In this case, recall that 
$\epsilon: (\Speck)_{et} \ra (\Speck)_{Nis}$ denotes the morphism of sites. (At this point one may make use of the model structures
in Construction ~\ref{more.rigidity}.)
  Now ${\rm L}\epsilon^*$ applied to the composition of maps forming the top rows in diagrams ~\eqref{big.diagm.1}, ~\eqref{big.diagm.2} and  ~\eqref{big.diagm.3}
defines the pre-transfer ${\rm L}\epsilon^*(\tr_{\rmX^h_{\rmZ}}')$ while ${\rm L}\epsilon^*$ applied to the composition of the maps forming the bottom rows in
 diagrams  ~\eqref{big.diagm.1}, ~\eqref{big.diagm.2} and  ~\eqref{big.diagm.3} defines the pre-transfer ${\rm L}\epsilon^*(\tr_{\rmZ}')$. 
 Therefore, ${\rm L}\epsilon^*(\tr_{\rmX^h_{\rmZ}}') = {\rm L}\epsilon ^*(i^h) \circ {\rm L}\epsilon ^*(\tr_{\rmZ}')$ and hence 
 the same holds on applying the Borel-construction, i.e.,
 \[(\rmE\times_{\rmG}{\rm L}\epsilon^*(\tr_{\rmX^h_{\rmZ}}'))_{\rmU_{\bullet}} = (\rmE\times_{\rmG}{\rm L}\epsilon ^*(i^h) \circ {\rm L}\epsilon ^*(\tr_{\rmZ}'))_{\rmU_{\bullet}}.\]
 Now it is clear one again obtains equality on applying ${\rm R}\epsilon_*$ to both sides. Therefore, essentially the same arguments as in the case where $\rmU$ is a Zariski open cover of $\rmB$ completes the proof.
\end{proof}

\begin{lemma}
 \label{loc.extend}
 Let $i:\rmZ \ra \rmX$ denote a closed regular immersion of smooth schemes over $\k$  and with the normal bundle associated $i$ being $\cN$. Then there exists a natural map
 \be \begin{equation}
 \label{map.on.Thom.sps}
 \RHom({\Sigma^{\infty}_{\T}}\rmZ_+, {\Sigma^{\infty}_{\T}}\rmZ_+) \ra \RHom( {\Sigma^{\infty}_{\T}}\Th(\cN), {\Sigma^{\infty}_{\T}}\Th(\cN)) 
 \end{equation} \ee
 which is a weak-equivalence. Here $\RHom$ denotes the derived internal hom in $\Spt(\k_{\rm mot})$ or $\Spt(\k_{et})$.
\end{lemma}
\begin{proof}
First observe that $\rmZ = \sqcup_i \rmZ_i$, where $\rmZ_i$ is a connected component of $\rmZ$. Therefore, we may assume without loss of generality  that 
$\rmZ$ has pure codimension $c$ in $\rmX$. In view of the adjunction between the internal hom, $\Hom$ and the smash product, $\wedge$, we will first show that there is a natural map
 \be \begin{equation}
   \label{map.Thom.sps.1}
{\rm Hom}(K \wedge {\Sigma^{\infty}_{\T}}\rmZ_+, {\Sigma^{\infty}_{\T}}\rmZ_+) \ra Hom(K \wedge {\Sigma^{\infty}_{\T}}\Th(\cN), {\Sigma^{\infty}_{\T}}\Th(\cN)),
    \end{equation} \ee
\vskip .1cm \noindent
 where ${\rm Hom}$ denotes the
 external hom in the category $\Spt(\k_{\rm mot})$ and for every $\rmK \eps \Spt(\k_{\rm mot})$.
The adjunction between $\wedge$ and the internal $\hom$, will then show this induces a natural map 
 \[{\rm Hom}(K, \Hom({\Sigma^{\infty}_{\T}}\rmZ_+, {\Sigma^{\infty}_{\T}}\rmZ_+)) \ra Hom(K, \Hom ({\Sigma^{\infty}_{\T}}\Th(\cN), {\Sigma^{\infty}_{\T}}\Th(\cN))\]
 for all $K \eps  {\Spt}(\k_{\rm mot})$, and therefore a natural map 
 \[\Hom({\Sigma^{\infty}_{\T}}\rmZ_+, {\Sigma^{\infty}_{\T}}\rmZ_+)\ra \Hom ({\Sigma^{\infty}_{\T}}\Th(\cN), {\Sigma^{\infty}_{\T}}\Th(\cN)).\]
 By making use of the the injective model structures on ${\Spt}(\k_{\rm mot})$ (as in \cite[5.2]{CJ23-T1}), we may assume that every object is
 cofibrant, and therefore, the above map will then induce a natural map (by taking fibrant replacements):
 \be \begin{equation}
 \label{internal.hom.weak.eq}
 \RHom({\Sigma^{\infty}_{\T}}\rmZ_+, {\Sigma^{\infty}_{\T}}\rmZ_+)\ra \RHom ({\Sigma^{\infty}_{\T}}\Th(\cN), {\Sigma^{\infty}_{\T}}\Th(\cN)).
 \end{equation} \ee
  Moreover, the fact the above map is a weak-equivalence will follow by showing that, on working locally on the Zariski topology of $\rmZ$, we reduce
  to the case where the normal bundle $\cN$ is in fact trivial, where the calculation reduces to the following: 
\be \begin{align}
     \label{Thom.sp.local}
 \RHom({\Sigma^{\infty}_{\T}}\T^{\rmc} \wedge \rmZ_+, {\Sigma^{\infty}_{\T}}\T^{\rmc} \wedge \rmZ_+) \simeq \RHom({\Sigma^{\infty}_{\T}}\rmZ_+, {\Sigma^{\infty}_{\T}} \rmZ_+). \notag
\end{align} \ee
 \vskip .1cm
 Therefore, it suffices to show that there is natural map as in ~\eqref{map.Thom.sps.1}.
 Recall that the Thom-space $\Th(\cN)$ is defined as the pushout:
 \be \begin{equation}
   \label{Thom.sp.0}
   \xymatrix{{Proj(\cN)} \ar@<1ex>[r] \ar@<1ex>[d] & {Proj(\cN \oplus 1)} \ar@<1ex>[d]\\
             {Spec \, \k} \ar@<1ex>[r] & {\Th(\cN)}.}
\end{equation} \ee
\vskip .1cm
One may observe that this pushout may be also obtained in two stages, by taking the pushout of the first diagram and then the second in:
\be \begin{equation}
   \label{Thom.sp.1}
   \xymatrix{{Proj(\cN)} \ar@<1ex>[r] \ar@<1ex>[d] & {Proj(\cN \oplus 1)} \ar@<1ex>[d]\\  
             {\rmZ} \ar@<1ex>[r] & {\rmS(\cN\oplus 1)},}                                                     
\quad \xymatrix{{\rmZ} \ar@<1ex>[r] \ar@<1ex>[d] & {\rmS(\cN \oplus 1)} \ar@<1ex>[d]\\
                {Spec \, \k} \ar@<1ex>[r] & {\Th(\cN)}.}
             \end{equation} \ee
\vskip .1cm
Next let $\{\rmU_i|i=1, \cdots, n\}$ denote a Zariski open cover of $\rmZ$ so that
$\cN$ is trivial on this cover, that is, $\cN_{|\rmU_i} = \rmU_i \times {\mathbb A}^{\rmc}$, for each $i$. Here we assume that
$c$ is the codimension of $\rmZ$ in $\rmX$. 
Let $\U= \sqcup _i \rmU_i$ and $\U_{\bullet} = cosk_0^{\rmZ}(\U)$, so that in degree $n$,
$ \U_n = (\U{\underset {\rmZ} \times} \cdots {\underset {\rmZ} \times} \U)$, which is the $n$-fold
fibered product of $\U$ with itself over $\rmZ$. Observe that now one has an isomorphism of simplicial schemes
\be \begin{equation}
\label{SN.isom.1}
   \phi:  \U_{\bullet} {\underset {Spec \, \k} \times } \T^{\rmc} {\overset {\cong} \ra} \rmS(\cN \oplus 1)_{|\U_{\bullet}},
    \end{equation} \ee
\vskip .1cm \noindent
where $\rmS(\cN \oplus 1)_{|\U_{\bullet}}$ denotes the pull-back of $\rmS(\cN \oplus 1)$ to $\U_{\bullet}$. This isomorphism defines an isomorphism of
simplicial objects of spectra:
\be \begin{equation}
\label{SN.isom.2}
   {\Sigma^{\infty}}_{ \T}\phi_+: {\Sigma^{\infty}_{\T}} \U_{\bullet, +} {\wedge } \T^{\rmc}_+  =
   {\Sigma^{\infty}_{\T}}(\U_{\bullet} {\underset {Spec \, \k} \times } \T^{\rmc})_+
    {\overset {\cong} \ra} {\Sigma^{\infty}_{\T}}\rmS(\cN \oplus 1)_{|\U_{\bullet}, +}.
    \end{equation} \ee
\vskip .1cm \noindent
Now consider the cosimplicial simplicial sets 
\[Hom(K \wedge {\Sigma^{\infty}_{\T}}\U_{\bullet, +}, {\Sigma^{\infty}_{\T}}\U_{\bullet, +}) \mbox{ and } Hom(K \wedge {\Sigma^{\infty}_{\T}}\rmS(\cN\oplus 1)_{|\U_{\bullet}}, {\Sigma^{\infty}_{\T}}\rmS(\cN \oplus 1)_{|\U_{\bullet}}).\]
Sending an $f: K \wedge {\Sigma^{\infty}_{\T}}\U_{n, +} \ra {\Sigma^{\infty}_{\T}}\U_{m, +}$ to ${\Sigma^{\infty}_{\T}}\phi_+(f \wedge id_{\T^c,+})$ (which denotes the induced map
$K \wedge {\Sigma^{\infty}_{\T}}\rmS(\cN\oplus 1)_{|\U_{\bullet,+}} \ra  {\Sigma^{\infty}_{\T}}\rmS(\cN \oplus 1)_{|\U_{\bullet,+}}$ defined by making use of the isomorphism 
 $ {\Sigma^{\infty}_{ \T}}\phi_+$)
defines a map 
\[Hom(K \wedge {\Sigma^{\infty}_{\T}}\U_{\bullet, +}, {\Sigma^{\infty}_{\T}}\U_{\bullet, +}) \ra Hom(K \wedge {\Sigma^{\infty}_{\T}}\rmS(\cN\oplus 1)_{|\U_{\bullet,+}}, {\Sigma^{\infty}_{\T}}\rmS(\cN \oplus 1)_{|\U_{\bullet,+}})\]
which one may verify is compatible with the cosimplicial simplicial structure on either side. Moreover, one also
obtains a commutative diagram
\be \begin{equation}
\label{section.s}
\xymatrix{{K \wedge {\Sigma^{\infty}_{\T}}\U_{n, +}} \ar@<1ex>[r]^f \ar@<1ex>[d]^s & { {\Sigma^{\infty}_{\T}}\U_{m, +}} \ar@<1ex>[d]^s\\
            {K \wedge {\Sigma^{\infty}_{\T}}\rmS(\cN\oplus 1)_{|\U_{n,+}}} \ar@<1ex>[r]^{\phi(f \times id_{\T^c})} &  { {\Sigma^{\infty}_{\T}}\rmS(\cN\oplus 1)_{|\U_{m,+}}},}
\end{equation}  \ee         
where $s$ denotes the canonical section. Therefore, collapsing the section $s$ defines a map
\[Hom(K \wedge {\Sigma^{\infty}_{\T}}\U_{\bullet, +}, {\Sigma^{\infty}_{\T}}\U_{\bullet, +}) \ra Hom(K \wedge {\Sigma^{\infty}_{\T}}\Th(\cN_{|\U_{\bullet}}), {\Sigma^{\infty}_{\T}}\Th(\cN_{|\U_{\bullet}}))\]
of cosimplicial simplicial sets. Since this is functorial in $K$, it follows that this defines a map of cosimplicial simplicial spectra
 of internal  homs:
\[\RHom({\Sigma^{\infty}_{\T}}\U_{\bullet, +}, {\Sigma^{\infty}_{\T}}\U_{\bullet, +}) \ra \RHom({\Sigma^{\infty}_{\T}}\Th(\cN_{|\U_{\bullet}}), {\Sigma^{\infty}_{\T}}\Th(\cN_{|\U_{\bullet}})).\]
The proof of the proposition may now be completed by observing the weak-equivalences:
\[\RHom({\Sigma^{\infty}_{\T}}\rmZ_+, {\Sigma^{\infty}_{\T}}\rmZ_+)  \simeq \holimD \hocolimD \RHom({\Sigma^{\infty}_{\T}}\U_{\bullet, +}, {\Sigma^{\infty}_{\T}}\U_{\bullet, +}) \mbox{ and }\]
\[\RHom({\Sigma^{\infty}_{\T}}\Th(\cN), {\Sigma^{\infty}_{\T}}\Th(\cN)) \simeq  \holimD \hocolimD \RHom({\Sigma^{\infty}_{\T}}\Th(\cN_{|\U_{\bullet}}), {\Sigma^{\infty}_{\T}}\Th(\cN_{|\U_{\bullet}})) .\]
Here we are making use of the weak-equivalences $\hocolimD {\Sigma^{\infty}_{\T}}\U_{\bullet, +} \simeq {\Sigma^{\infty}_{\T}}\rmZ_+$ (
see for example, \cite{DHI}) and
\be \begin{align}
     \label{hypercover.weak.eq}
     \hocolimD {\Sigma^{\infty}_{\T}}\Th(\cN_{|\U_{\bullet}}) &\simeq \hocolimD ({\Sigma^{\infty}_{\T}}S(\cN \oplus 1)_{|\U_{\bullet}}/s({\Sigma^{\infty}_{\T}}(\U_{\bullet, +})))\\
                                                   &\simeq \hocolimD ({\Sigma^{\infty}_{\T}}S(\cN \oplus 1)_{|\U_{\bullet}})/s(\hocolimD {\Sigma^{\infty}_{\T}}(\U_{\bullet, +})) \notag\\
                                                   &\simeq {\Sigma^{\infty}_{\T}}S(\cN \oplus 1)/{\Sigma^{\infty}_{\T}}\rmZ_+  \simeq {\Sigma^{\infty}_{\T}} \Th(\cN) \notag 
\end{align}\ee
\vskip .1cm \noindent
where $s$ is the section considered in ~\eqref{section.s}. Finally we observe that the homotopy colimit in the left argument
pulls out of the $\RHom( \quad, \quad)$ as a homotopy inverse limit, while the homotopy colimit in the right argument
pulls out as a homotopy colimit, since the left argument of the $\RHom(\quad, \quad)$ is a compact object.
\end{proof}
\vskip .2cm

\section{\bf The rigidity property, Motivic Tubular Neighborhoods and Henselization along smooth closed subschemes} \index{motivic tubular neighborhood} \index{Henselization}
We begin this section by discussing the rigidity property in some detail, giving various criteria that ensure rigidity.
The following proposition lists a small sample of convenient criteria that ensure that a motivic spectrum $\rmM$ has the rigidity property as
in Definition ~\ref{rigid.prop}.
\begin{proposition} \index{rigidity}
\label{rigid.props.1}
In {\rm (i)} through {\rm (iv)}, let $\rmM$ denote a motivic spectrum so that  there exists a prime $\ell \ne char(k)$, so that the homotopy groups
 of the spectrum $\rmM$ are all $\ell$-primary torsion. 
 
\begin{enumerate}[\rm(i)]
 \item  The base field is infinite and $\rmM$ defines an orientable motivic cohomology theory, that is, one that has a theory of Chern classes.
 \item The base field $k$ is algebraically or quadratically closed and of characteristic $0$ with $\ell \ne 2$: there are no further restrictions on the motivic spectrum $\rmM$.
 \item The base field $k$ is infinite, non-real (i.e., not formally real or equivalently $-1$ is a sum of squares) and of characteristic $0$ with $\ell \ne 2$: there are no further restrictions on the motivic spectrum $\rmM$.
\item The base field $k$ is infinite, non-real, of characteristic different from $2$ and the prime $\ell \ne 2$: there are no further restrictions on the motivic spectrum $\rmM$.
\end{enumerate}
 
If any one of the above hypotheses are satisfied, then $\rmM$ has the rigidity property (as in Definition ~\ref{rigid.prop}).
\vskip .1cm
{\rm (v)} Alternatively, if the base field $k$ is of characteristic $0$, $\phi$ is a class in the Grothendieck-Witt group of the field $k$, so that $rank(\phi)$ is invertible in $k$ and
the spectrum $\rmM$ is $\phi$-torsion, that is, $\phi\rmM=0$, then again $\rmM$ has the rigidity property (as in Definition ~\ref{rigid.prop}).
\vskip .1cm
In particular, the spectrum representing algebraic K-theory with finite coefficients prime to the
  characteristic has the rigidity property.
\end{proposition}
\begin{proof} The fact that one has the above rigidity property for the spectrum representing algebraic K-theory follows from
Gabber's theorem which holds for all Hensel pairs: see \cite[Theorem 1]{Gab}. The remaining statements need to be deduced from what is
in the literature on rigidity: statements (i) ((ii) and  (iii)) when $x$ is a $k$-rational point of a smooth variety is stated
in  \cite[Theorem 0.3, Corollary 0.4]{HY07}, as well as \cite[Theorem 1.5]{Y04} and \cite[Corollary 2.6]{Y11}. (See also \cite[Theorem 1.13]{PY02}.) Observe that, in Definition ~\ref{rigid.prop},
one does not require the point $x$ to be a $k$-rational point. Therefore, we proceed to show that the above rigidity property can be deduced
from the corresponding statement for the case $x$ is a $k$-rational point.
\vskip .1cm
Next let $\rmZ$ denote the closure of the given point $\x$ and let $\z= \x$, but viewed as a point of $\rmZ$. 
By replacing $\rmX$ and $\rmZ$ by open subschemes, we may assume without loss of generality that $\rmZ$ is
smooth and $\z$ denotes the generic point of $\rmZ$.
The local structure discussed below in Lemma ~\ref{local.struct} shows that there is a Zariski open neighborhood 
$\rmU_{\z}$ of $\z$ in $\rmX$ and an 
\'etale map $q_{\z}:\rmU_{\z}  \ra {\mathbb A}^n$, (where $n= dim_k(\rmX)$),
so that $\rmU_{\z} \cap \rmZ = q_{\z}^{-1}({\mathbb A}^{n-c} \times \{0\})$, (where $c= codim _{\rmX}(\rmZ)$). Moreover, there is then a smaller open
$\rmV_{\z}$ in $\rmU_{\z}$ which is a Nisnevich neighborhood of $\rmU_{\z} \cap \rmZ$ in $\rmU_{\z}$ and also of $(\rmU_{\z} \cap Z )\times \{0\}$ in 
$(\rmU_{\z} \cap Z) \times {\mathbb A}^c$, in the sense that the conditions in Lemma ~\ref{local.struct}(ii) are satisfied.
Then, since Nisnevich neighborhoods of the form $\rmW_{\z} \times_k \rmW'_0$, where $\rmW_{\z}$ is a Nisnevich neighborhood
of $z$ in $\rmZ$ and $\rmW'_0$ is a Nisnevich neighborhood of $0$ in ${\mathbb A}^c$, are cofinal in the system of all Nisnevich neighborhoods of
 the point $\z \times {\rm 0} $ in $\rmZ \times {\mathbb A}^c$, we obtain
 \be \begin{equation}
      \label{stalk.ident}
      \O_{\rmX, \x}^h \cong \O_{\rmZ, \z}^h \otimes_k \O_{{\mathbb A}^c, 0}^h.
 \end{equation} \ee
\vskip .1cm \noindent
Since $\z$ is the generic point of $\rmZ$, clearly $\O_{\rmZ, \z}^h \cong k(\z)$. Thus $\O_{\rmX, \x}^h \cong k(\z) \otimes_k \O_{{\mathbb A}^c, {\rm 0}}^h$.
At this point, we may consider the scheme ${\rm Spec} \, \k(\z) {\underset {{\rm Spec} \, \k} \times} {\mathbb A}^c$: clearly $\z {\underset {{\rm Spec} \, \k} \times}{\rm 0}$
is a $k(\z)$ rational point of the scheme ${\rm Spec} \, \k(\z) {\underset {{\rm Spec} \, \k} \times} {\mathbb A}^c$. 
\vskip .1cm
It is observed on \cite[p. 441]{HY07} that any generalized orientable motivic cohomology theory is normalized with respect to any field.
Therefore, \cite[Theorem 0.3 and Corollary 0.4]{HY07} apply to prove the statement in (i).
The field $k(\z)$ and the fraction field of
$k(\z) {\underset k \otimes}\O_{{\mathbb A}^c, {\rm 0}}^h$ satisfy the hypotheses of \cite[Corollary 2.6]{Y11}, which proves the statements in (ii) and (iii).
The statement in (iv) also follows from \cite[Corollary 2.6]{Y11}, once the restriction that the field of fractions of the Hensel ring 
be perfect is removed. An analysis of the proof of \cite[Corollary 2.6]{Y11} shows that this condition is put in because of the restriction
that the field be perfect in Morel's theorem as in \cite[Theorem 6.2.2]{Mo4}. By \cite[Theorem 10.12]{BH}, the above assumption is no longer
 needed in the above Theorem of Morel. The fifth statement appears in \cite[Corollary 1.3]{AD} where the hypothesis
 is that the base field is perfect. However, in order to apply this result to residue fields of non-closed points,
 one needs to assume that such residue fields are also perfect, which is guaranteed by the assumption that the characteristic is $0$.

\end{proof}
\begin{remark}
 In view of \cite{AD} and \cite{BH}, it seems possible to obtain more general criteria than those listed in Proposition 
 ~\ref{rigid.props.1} which would ensure rigidity in the sense of Definition ~\ref{rigid.prop}. The purpose of Proposition
 ~\ref{rigid.props.1} is not to provide the most general criteria to ensure such rigidity, but to give a small sample of
 convenient criteria to ensure rigidity.
\end{remark}
 \vskip .2cm
 Let $\rmM \eps \Spt(\k_{mot})$. Then one has Voevodsky's slice tower (see \cite{Voev00}): $\{\rmf_n \rmM|n\}$, where $\rmf_{n+1}\rmM$ is the $n$-th connective cover of $\rmM$. Let $s_{\le n} \rmM $ be the homotopy cofiber of the map 
	$\rmf_{n+1}\rmM \ra \rmM$. Then, as shown in \cite{Pel}, the diagram
\[\xymatrix{{\cdots} \ar@<1ex>[d] \ar@<1ex>[r] & {\rmf_{n+1} \rmM}  \ar@<1ex>[d] \ar@<1ex>[r] & {\rmf_n \rmM} \ar@<1ex>[d] \ar@<1ex>[r] & {\cdots} \ar@<1ex>[d] \\
            {\cdots} \ar@<1ex>[d] \ar@<1ex>[r]^{id} & {\rmM} \ar@<1ex>[d] \ar@<1ex>[r]^{id} & {\rmM} \ar@<1ex>[d] \ar@<1ex>[r]^{id} &{\cdots} \ar@<1ex>[d] \\
            {\cdots} \ar@<1ex>[r] & {s_{\le n}\rmM} \ar@<1ex>[r] & {s_{\le n-1}\rmM} \ar@<1ex>[r] &{\cdots} }
\]
admits a lifting to $\Spt(\k_{mot})$.
 
\vskip .2cm
\begin{proposition}\index{rigidity: slices}
 \label{slices.rigid} Let $\rmM $ denote a motivic spectrum. If the homotopy groups of 
 $\rmM$ are all $\ell$-primary torsion, for some prime $\ell \ne char (k)$, then
 the slices of $\rmM$ have the rigidity property. In particular, if the spectrum $\rmM$ has the rigidity property as in Definition ~\ref{rigid.prop}, then all
 its slices have the rigidity property.
\end{proposition}
\begin{proof} We observe from \cite[Theorem 2.4]{Pel11} that the slices of any motivic spectrum are orientable.
Therefore, in view of Proposition ~\ref{rigid.props.1}(i), it suffices to show that if the spectrum $\rmM$ is such that its homotopy groups are all $\ell$-primary torsion for some
 prime $\ell \ne char (k)$, then the same property holds for its slices. 
For this one needs to recall the construction of the ${\mathbb P}^1$-slices of a motivic spectrum as in \cite[sections 8, 9]{Lev}.
 First one shows that the $\Omega_{\T}$-spectrum, $\widehat \rmM$, associated to $\rmM$ also has its homotopy groups all $\ell$-primary torsion: this follows readily from
 the fact that $\widehat \rmM_n = \colimm \Omega_{\T}^m\rmM_{n+m}$. It follows that $\widehat \rmM_n$ has its homotopy groups all $\ell$-primary torsion, for each fixed $n$.
 Next one constructs a bi-spectrum by taking the $\rmS^1$-suspension spectrum,  $\Sigma^{\infty}_{\rmS^1}\widehat \rmM_n$, of each $\widehat \rmM_n$: 
 $\{\Sigma^{\infty}_{\rmS^1}\widehat \rmM_n|n\ge 0\}$. One may see readily that each of the $\rmS^1$-suspension spectra, $\Sigma^{\infty}_{\rmS^1}\widehat \rmM_n$ also
 has its homotopy groups all $\ell$-primary torsion. Finally one applies the construction of the slices as in \cite[8.3]{Lev} in terms of the slices
  of the $\rmS^1$-spectra, $\Sigma^{\infty}_{\rmS^1}\widehat \rmM_n$. Thus, we reduce to showing that if the homotopy groups of the $\rmS^1$-spectrum $\X$
  are all $\ell$-primary torsion, then its $\rmS^1$-slices also  have their homotopy groups all $\ell$-primary torsion: this is clear from the explicit construction of such slices 
   as in \cite[2.1]{Lev}.
\end{proof}

\begin{corollary} \index{rigidity: slice completed spectra}
 \label{slice.completed.rigid}
 Let $\rmM $ denote a motivic spectrum so that the homotopy groups of 
 $\rmM$ are all $\ell$-primary torsion, for some prime $\ell \ne char (k)$. Then the slice-completed spectrum $\holimn s_{\le n}\rmM$
 has the rigidity property.
\end{corollary}
\begin{proof} The proof is clear in view of Proposition ~\ref{slices.rigid}.
\end{proof}

\vskip .2cm
\begin{lemma} \index{rigidity: \'etale spectra}
 \label{rigid.et.Nis}
 Let $\rmM$ denote a motivic ring spectrum whose homotopy groups are all $\ell$-primary torsion, for a fixed prime $\ell \ne char (k)$.
 Then if $\rmM$ has the rigidity property as in Definition ~\ref{rigid.prop}, its pull-back $\epsilon^*(\rmM)$ to the \'etale site
  also has the rigidity property
\end{lemma}
\begin{proof}
 Clearly for every  Hensel ring $\rmR$ with residue field $\rmK$, the map $\Gamma ({\rm Spec} \, \rmR, \rmM) \ra \Gamma ({\rm Spec}\, \rmK, \rmM)$ is a weak-equivalence, since
 $\rmM$ has the rigidity property (on the Nisnevich site).
 The fact that, if $\rmK_1 \subseteq \rmK_2 $ is an extension of algebraically closed fields, then the induced map $\Gamma ({\rm Spec} \, \rmK_1, \rmM) \ra \Gamma({\rm  Spec} \, \rmK_2, \rmM)$
 is a weak-equivalence is shown in \cite[Theorem 1.10]{Y04}. To see that the same holds when $\rmK_1$ and $\rmK_2$ are only separably closed,
 one observes that for any purely inseparable field extension $\rmK \subseteq \rmK'$ of fields containing the base field $k$, $\Gamma ({\rm Spec} \, \rmK, \rmM) \simeq \Gamma ({\rm Spec }\, \rmK', \rmM)$
 as the homotopy groups of $\rmM$ are all $\ell$ primary torsion and the degree of the field extension is prime to $\ell$.
\end{proof}

\subsection{Motivic tubular Neighborhoods and Henselization along smooth closed subschemes}
\label{m.tub.nbds}
We next discuss gadgets we call motivic tubular neighborhoods: we model this on
the \'etale tubular neighborhoods that have been around since \cite{Cox}. Given a smooth scheme $\rmX$, a {\it rigid Nisnevich cover}
of $\rmX$ is a map of schemes $\rmU \ra \rmX$, which is a Nisnevich cover, and in addition, $\rmU$ is a disjoint union of 
pointed \'etale separated maps $\rmU_x, \rmu_{\rmx} \ra \rmX, \rmx$, where each $\rmU_{\rmx}$ is {\it connected}, and
as $\rmx$ varies over points of $\rmX$, so that the above map
induces an isomorphism of residue fields $\k(\rmu_{\rmx}) \cong \k(\rmx)$. 
Given two rigid Nisnevich  
covers $\rmU \ra \rmX$ and $\rmV \ra \rmZ$, the {\it rigid product} $\rmU {\overset R \times} \rmV$ is given as the disjoint union 
 of $(\rmU_x \times \rmV_y)_0$ which is the connected component of $\rmU_x \times \rmV_y$ 
 indexed by the point $\rmx \times \rmy$ of $\rmX \times \rmZ$.  
 \vskip .1cm
 Given a scheme $\rmX$, a {\it hypercover} $\rmU_{\bullet}$ of $\rmX$ is a simplicial scheme $\rmU_{\bullet}$ together 
 with an \'etale map $\rmU_{\bullet} \ra \rmX$, so that (i) $\rmU_0 \ra \rmX$ is a Nisnevich cover and  (ii)
 the induced map $\rmU_t \ra (cosk^{\rmX}_{t-1}U_{\bullet})_t$ is a Nisnevich cover, for each $t \ge 1$. 
 Such a hypercover is a {\it rigid hypercover} if the given maps in (i) and (ii) are rigid Nisnevich covers. 
 One may readily show that the category of rigid Nisnevich hypercovers of a given scheme $\rmX$ is a left directed
 category: see \cite[section 1]{Cox}. Here it may be important to point out that this category itself (and not the associated homotopy category) is a left directed category,
 and this is because we are working with {\it rigid} hypercoverings. This category will be denoted $\HRR(\rmX)$. 
 \vskip .1cm
 \begin{definition}(Motivic tubular neighborhoods)
 \label{tub.nbds}
 Let $\rmZ$ denote a closed smooth subscheme of a smooth scheme $\rmX$. We define the {\it (rigid) motivic tubular neighborhood}
 of $\rmZ$ in $\rmX$ to be the inverse system of simplicial schemes $\rmN^{\rmZ}_{\bullet}(\rmU_{\bullet})$ for which there exists
 a rigid Nisnevich hypercover $\rmU_{\bullet}$ of $\rmX$ so that $\rmN^{\rmZ}_n(\rmU_{\bullet})= \sqcup _{i} \rmU_{n,i}$, where the sum runs over
  $\rmU_{n,i}$, which are
 connected components of $\rmU_n$ with the property that $\rmU_{n, i}{\underset {\rmX} \times} \rmZ \ne \phi$.  One may readily see that the
 motivic tubular neighborhood of $\rmZ$ in $\rmX$ is a left-directed category. This will be denoted $t_{\rmX/\rmZ}$.
 See \cite[section 1]{Cox} for similar definitions of \'etale tubular neighborhoods. 
 \end{definition}
\begin{remark} The inverse system of all rigid Nisnevich neighborhoods of $\rmZ$ in $\rmX$ corresponds to the 
 Henselizaton of $\rmX$ along $\rmZ$. This will become clear from Theorem ~\ref{coh.tub.nbd}.
\end{remark}
Next we will provide the following Lemma, whose proof is skipped as it follows exactly as in \cite[Lemma 1.2]{Cox}.
\begin{lemma}
 \label{Cox.lemma1.2}
 Given a $\rmV_{\bullet}$ in $t_{\rmX/\rmZ}$, and a separated \'etale map $\phi:\rmW \ra  \rmV_n$,  and so
  that the induced map $\rmW{\underset {\rmX} \times} \rmZ \ra \rmV_n{\underset {\rmX} \times} \rmZ$ is a Nisnevich cover, there is a map
  $\phi_{\bullet}:\rmW_{\bullet} \ra \rmV_{\bullet}$ in $t_{\rmX/\rmZ}$, so that $\phi_n$ factors through the map $\phi$.
\end{lemma}

 Next we recall that the main model structure we use on the category ${\widetilde \Spt}^{\rmG}(\k_{\rm mot})$ is defined as follows.
 First we start with the injective model structure on the category $\Spc_*(\k)$ of pointed simplicial presheaves, where cofibrations (weak-equivalences) are 
 section-wise cofibrations (weak-equivalences, \res) of pointed simplicial sets, and 
 fibrations are defined by the right lifting property with respect to
  maps that are trivial cofibrations. Then we localize this model structure by inverting maps that are both stalk-wise weak-equivalences
   and also maps of the form ${\mathbb A}^1 \times \rmU \ra \rmU$, for any $\rmU$ in the site. We will
   call the resulting category, the category of {\it motivic spaces} and denoted $\Spc_*(\k_{\rm mot})$.
   Recall that the objects of the category ${\widetilde \Spt}^{\rmG}$
   are $\Spc_*(\k_{\rm mot})$-enriched functors $\Sph^{\rmG} \ra \Spc_*(\k_{\rm mot})$, where $\Spc_*(\k_{\rm mot})$ is provided with the above model
   structure. We start with the {\it level-wise injective model structure}
   on this category, where the cofibrations (weak-equivalences) are maps $\phi:\X' \ra \X$ for which the induced map 
   $\phi(\rmT_{\rmV}) :\X'(\rmT_{\rmV}) \ra \X (\rmT_{\rmV})$ is a cofibration (weak-equivalence, \res) for every $\rmT_{\rmV}$. Finally we
    obtain the corresponding stable model structure, where the fibrant objects are the $\Omega$-spectra. A map between
    two fibrant spectra $\rmM' =\{\rmM'(\rmT_{\rmV})|\rmV\} \ra \rmM =\{\rmM(\rmT_{\rmV})|\rmV\}$ is a weak-equivalence if and only if the map $\rmM'(\rmT_{\rmV}) \ra \rmM(\rmT_{\rmV})$ is
    a weak-equivalence for each $\rmV$, in the level-wise injective model structure, which implies (in view of the above discussion)
    that it is a stalk-wise weak-equivalence  of ${\mathbb A}^1$-localized simplicial presheaves.
    \vskip .1cm
    We also introduced in \cite[Definition 4.6]{CJ23-T1} the category $\Spt(\k_{mot})$ whose objects are sequences $\{\rmM_n|n \ge 0\}$ of
    motivic spaces, together with a compatible family of structure maps $\rmS^1 \wedge \rmM_n \ra \rmM_{n+1}$, $n \ge 0$. We point out 
    the following important fact: 
    one has a Quillen equivalence between the model category $\Spt(\k_{mot})$ of motivic spectra
    and the model category ${\widetilde \Spt}^{\rmG}(\k_{\rm mot})$. See \cite[Proposition 6.2]{CJ23-T1}. 
    \vskip .1cm
    We will put the {\it level-wise injective} model structure on the category $\Spt(\k_{mot})$, where cofibrations and weak-equivalences
    are defined level-wise and fibrations defined by the right-lifting property with respect to trivial cofibrations. 
    A motivic {\it $\Omega$-bi-spectrum}
    ${\mathcal S}$ is given by a sequence $\{{\mathcal S}_n|n \ge 0\}$ of motivic $\rmS^1$-spectra, together with 
    compatible weak-equivalences $\{{\mathcal S}_n \ra \Omega_{\T}({\mathcal S}_{n+1})|n \ge 0\}$. One may define motivic bi-spectra similarly, 
    by just relaxing the condition that the maps ${\mathcal S}_n \ra \Omega_{\T}({\mathcal S}_{n+1})$ are weak-equivalences. With suitable model
    structures, the above categories of spectra are all Quillen-equivalent.
    (See \cite[section 8]{Lev}, where such
    spectra are called by a slightly different name. The relationship between various categories of spectra is discussed
    there.) 
    \vskip .1cm
    We define {\it rigidity} for motivic $\rmS^1$-spectra just as in Definition ~\ref{rigid.prop}.
    A useful observation for us is the following: given a motivic $\rmS^1$-spectrum $\rmM =\{\rmM_n|n\ge 0\}$, so that (i) each $\rmM_n$ is ${\mathbb A}^1$-local and stalk-wise fibrant, one may obtain a fibrant replacement (that is, fibrant in
    the level-wise injective model structure on $\Spt(\k_{mot})$ by applying the canonical Godement resolution $\rmG^{\bullet}$
    (which produces a cosimplicial object) and then taking its homotopy inverse limit): the resulting motivic spectrum will be
    denoted $\rmG(\rmM)$. 
    \begin{proposition}
     \label{ss.s1.spectra} \index{Nisnevich hypercover}
     \begin{enumerate}[\rm(i)]
     \item Given any motivic $\rmS^1$-spectrum $\rmM =\{\rmM_n|n \ge 0\}$, there exists a spectral sequence
     \[\rmE_2^{s,t} = \colimalpha \rmH^s(\Gamma(\rmU_{\bullet}^{\alpha}, \pi_t(\rmM))) \Ra \pi_{-s+t} (\holimD \colimalpha \Gamma (\rmU_{\bullet}^{\alpha}, \rmM))\]
     where the colimit is taken over all rigid Nisnevich hypercovers of the given smooth scheme $\rmX$. This spectral
     sequence converges strongly as $\rmE_2^{s,t} \cong \rmH^s_{Nis}(\rmX, a\pi_t(\rmM)) =0$ for all $s> dim (X)$,
      where $a\pi_t(\rmM)$ denotes the associated abelian sheaf.
     \vskip .1cm
     \item If $\rmM$ is a motivic $\rmS^1$-spectrum as in (i) and is replaced by its fibrant replacement in the injective model structure on $\Spt(\k_{mot})$ 
     (for example, its Godement resolution $\rmG(\rmM)$), then the map from the spectral sequence
     \[\rmE_2^{s,t} = \rmH^s(\Gamma (\rmU_{\bullet}^{\alpha_0}, \pi_t(\rmG(\rmM)))) \Ra \pi_{-s+t} (\holimD \Gamma (\rmU_{\bullet}^{\alpha_0}, \rmG(\rmM)))\]
     for a fixed rigid Nisnevich hypercover $\rmU_{\bullet}^{\alpha_0}$, to the spectral sequence in (i), is an isomorphism.
    \end{enumerate}
    \end{proposition}
\begin{proof} (i) The existence of the spectral sequence and the identification of the $E_2$-terms follow
 readily in view of the discussion in \cite[Proposition 1.16]{Th85}. At this point, one knows that cohomology on any site with respect to
 an abelian presheaf may be computed using hypercoverings in that site: see \cite{SGA4}. Therefore, the $\rmE_2^{s,t}$-term in (i)
 identifies with $\rmH^s_{Nis}(X, a\pi_t(\rmM))$, where $ a\pi_t(\rmM)$ denotes the sheaf associated to the
  presheaf $\pi_t(\rmM)$. 
 Finally one uses the fact that the Nisnevich site of the smooth scheme $\rmX$ has cohomological dimension given by $dim(\rmX)$
 to complete the proof of (i). 
 \vskip .1cm
 (ii) We will assume that $\rmM$ is replaced by $\rmG(\rmM)  = \holimD \rmG^{\bullet}(\rmM)$. 
 Then the $\rmE_2^{s,t}$-term of the spectral sequence in (ii) identifies with
 $\rmH^s(\Gamma(\rmU_{\bullet}^{\alpha_0}, \pi_t(\rmG(\rmM)))) \cong 
 \rmH^s(\Gamma(\rmU_{\bullet}^{\alpha_0}, \rmG^{\bullet} \pi_t(\rmM))) \cong  
 \rmH^s_{Nis}(\rmX, a\pi_t(\rmM))$. Thus these $\rmE_2^{s,t}$-terms do not depend on the choice of the rigid
 Nisnevich hypercover, and also identifies with the $\rmE_2^{s,t}$-term of the spectral sequence in (i). Since both
 spectral sequences converge strongly, we obtain an isomorphism on the abutments, thereby proving (ii). 
\end{proof}
\vskip .2cm

Next we recall the definition of {\it Hensel pairs} for affine schemes from \cite[Definition 15.11.1]{St}. 
Accordingly an (affine) Hensel pair is given by a pair $(\rmA, \rmI)$ where $\rmA$ is a commutative ring with $1$ and
$\rmI$ is an ideal in $\rmA$ contained in the Jacobson radical of $\rmA$, so that for any monic polynomial $ f \in \rmA[\rmT]$ and factorization $\bar \rmf = \bar \rmg. \bar \rmh$ with $\bar \rmg, \bar \rmh \in (\rmA/\rmI) [\rmT]$, which is monic and generating the unit 
ideal in $\rmA/\rmI[\rmT]$, there exists a factorization $\rmf=\rmg. \rmh$ in $\rmA[\rmT]$ with $\rmg, \rmh$
monic and $\bar \rmg$ ( and $\bar \rmh$) the image of $\rmg$ ($\rmf$) in $\rmA/\rmI[\rmT]$. 
\vskip .1cm
It is important to view the above {\it affine Hensel pair as the affine scheme given by} $(Spec \, (\rmA/\rmI), \rmA)$, that is,
 the underlying topological space is $Spec (\rmA/\rmI)$ and the structure sheaf is the one defined by the ring $\rmA$.
\vskip .1cm
\begin{definition} (Hensel pairs and Henselization) \index{Hensel pair} \index{Henselization}
 Let $\rmX$ be a given scheme and $\rmZ$ a closed subscheme of $\rmX$. Then ($\rmX$, $\rmZ)$ is a {\it Hensel pair}
 if for every affine open cover $\{\rmU_i|i\}$ of $\rmX$, $(\rmU_i, \rmU_i \cap \rmZ)$ is Hensel pair as in
 \cite[Definition 15.11.1]{St}. Given a scheme $\rmX$ and a closed subscheme $\rmZ$ of $\rmX$,
 the Henselization of $\rmX$ along $\rmZ$ is the scheme obtained by gluing the Henselization of the
 affine schemes $\{\rmU_i|i\}$ along the closed subschemes $\rmU_i \cap \rmZ$. This will be denoted $\rmX_{\rmZ}^h$.
 \vskip .1cm
 It is important to view the scheme $\rmX_{\rmZ}^h$ as given by the underlying space $\rmZ$ and provided with the structure sheaf
 $\O_{\rmX, \rmZ}^h$ which is obtained by the gluing $(\rmZ \cap \rmU_i, \O_{\rmU_i, \rmU_i \cap \rmZ}^h)$ for 
 any affine open cover $\{\rmU_i|i\}$ of the scheme $\rmX$. (See \cite[p. 213]{Cox}
\end{definition}
\begin{lemma}
 \label{funct.Hensel}
 If 
 \[ \xymatrix{{\rmZ'} \ar@<1ex>[r] \ar@<1ex>[d] & {\rmX'} \ar@<1ex>[d]\\
               {\rmZ} \ar@<1ex>[r] & {\rmX}}
 \]
is a cartesian square where the horizontal maps are closed immersions, one obtains an induced map
$\rmX_{\rmZ'}'^h \ra \rmX_{\rmZ}^h$.
\end{lemma}
\begin{proof} This is skipped as it is an easy exercise to complete from the definition of Henselization: one may
in fact reduce to the case where all the schemes are affine.
\end{proof}
\vskip .1cm
\begin{theorem} \index{rigidity}
 \label{coh.tub.nbd}
 Let $i: \rmZ \ra \rmX$ denote a closed immersion of smooth schemes. 
 \begin{enumerate}[\rm(i)]
 \item Then for any motivic spectrum $\rmM=\{\rmM_n|n \}$ in 
 $\Spt(\k_{mot})$, one obtains a weak-equivalence
 \[\holimD \colimalpha \Gamma(\rmN^{\rmZ}(\rmU_{\bullet}^{\alpha}), \rmM) \simeq \H_{\Nis}(\rmX_{\rmZ}^h, \rmM)\]
 where $\rmU_{\bullet}^{\alpha}$ varies among all hypercoverings of $\rmX$, $\rmX_{\rmZ}^h$ denotes the Henselization of the
  scheme $\rmX$ along $\rmZ$, and $\H_{\Nis}(\rmX_{\rmZ}^h, \rmM)$ denotes the
 spectrum $\Gamma(\rmX_{\rmZ}^h, G(\rmM))$, with ${\rm G}(\rmM)$ denoting a fibrant replacement of $\rmM$ as in Proposition ~\ref{ss.s1.spectra}. 
 \vskip .1cm
 \item If in addition, the spectrum $\rmM$  has the rigidity property in Definition ~\ref{rigid.prop}, then
 one obtains the weak-equivalence: 
 \[\H_{\Nis}(\rmX_{\rmZ}^h, \rmM) \simeq \H_{Nis}(\rmZ, \rmM).\]
 \vskip .1cm
 \item If $\rmX$ and $\rmZ$ are provided with the action of a linear algebraic group $\rmG$ with the map $i$ $\rmG$-equivariant, the same
 conclusions also hold for any spectrum $\rmM \eps {\widetilde \Spt}^{\rmG}(\k_{\rm mot})$.
 \vskip .1cm
 \item Corresponding results also hold for hypercohomology computed on the \'etale site, provided the base field $\k$ has finite $\ell$-cohomological dimension
 for some prime $\ell \ne char(\k)$ and the homotopy groups of the spectrum $\rmM$ are $\ell$-primary torsion.
 \end{enumerate}
\end{theorem}
\begin{proof} 
First, making use of the relationship between motivic $\rmS^1$-spectra and motivic $\T$-spectra, we observe that it suffices to prove the first 
two statements for motivic $\rmS^1$-spectra. Therefore, we will assume that $\rmM$ denotes a motivic $\rmS^1$-spectrum throughout the proof 
of the first two statements. Throughout this proof, we will also adopt the following notational conventions. 
For a smooth scheme $\rmY$, we let $\rmM_{|\rmY}$ denote the restriction of $\rmM$ to the {\it small} Nisnevich site of the scheme $\rmY$. 
Given a presheaf $\rmP$ on the small Nisnevich site of the scheme $\rmX$, we will let $i^{-1}(\rmP)$ denote the
 restriction of $\rmP$ to the small Nisnevich site of the closed subscheme $\rmZ$.
 \vskip .2cm
 As in Proposition ~\ref{ss.s1.spectra}, one obtains  spectral sequences:
 \be \begin{align}
 \label{comp.ss.1}
 \rmE_2^{s,t}(1) = \colimalpha\rmH^s(\Gamma(\rmN^{\rmZ}(\rmU_{\bullet}^{\alpha}), \pi_t(\rmM_{|\rmX})) &\Ra \pi_{\rm -s+t}(\holimD \colimalpha \Gamma(\rmN^{\rmZ}(\rmU_{\bullet}^{\alpha}), \rmM_{|\rmX})) \mbox{ and }\\
\rmE_2^{s,t}(2) =\rmH^s(\rmZ, \pi_t(i^{-1}\rmM_{|\rmX})) &\Ra  \pi_{\rm -s+t}\H(\rmZ, i^{-1}\rmM_{|\rmX})\notag.
\end{align} \ee
Since every scheme that appears in the simplicial scheme $\rmN^{\rmZ}(\rmU_{\bullet}^{\alpha})$ in each degree belongs to the small Nisnevich site of $\rmX$,
one may identify the first spectral sequence with
\[\rmE_2^{s,t}(1) = \colimalpha\rmH^s(\Gamma(\rmN^{\rmZ}(\rmU_{\bullet}^{\alpha}), \pi_{\rm t}(\rmM)) \Ra \pi_{\rm -s+t}(\holimD \colimalpha \Gamma(\rmN^{\rmZ}(\rmU_{\bullet}^{\alpha}), \rmM)).\]
\vskip .1cm
In addition, there is also a third spectral sequence:
\be \begin{align}
 \label{comp.ss.2}
 \rmE_2^{s,t}(3) = \rmH^s(\rmZ, \pi_t(\rmM))\cong  \rmH^s(\rmZ, \pi_{\rm t}(\rmM_{|\rmZ})) &\Ra \pi_{\rm -s+t}\H(\rmZ, \rmM) \cong \pi_{-s+t}\H(\rmZ, \rmM_{|\rmZ}).
\end{align} \ee
\vskip .1cm
We reduce to showing there are natural maps of these spectral 
sequences inducing  an isomorphism at the $\rmE_2$-terms, and that all three of these spectral sequences converge strongly. First making use of
Lemma ~\ref{Cox.lemma1.2}, we obtain the identification:
\[\colimalpha\rmH^s(\Gamma(\rmN^{\rmZ}(\rmU_{\bullet}^{\alpha}), \pi_t(\rmM_{|\rmX})) \cong \colimalpha\rmH^s_{Nis}(\rmN^{\rmZ}(\rmU_{\bullet}^{\alpha}), \pi_{\rm t}(\rmM_{|\rmX})),\]
where the term on the left (right) denotes the cohomology of the co-chain complex $\Gamma(\rmN^{\rmZ}(\rmU_{\bullet}^{\alpha}), \pi_{\rm t}(\rmM_{|\rmX})$
(the Nisnevich hypercohomology of $\rmN^{\rmZ}(\rmU_{\bullet}^{\alpha})$ with respect to the abelian sheaf $\pi_{\rm t}(\rmM_{|\rmX})$, \res). 
Observe that $\rmN^{\rmZ}(\rmU_{\bullet}^{\alpha}){\underset {\rmX} \times} \rmZ$ is a Nisnevich hypercover of $\rmZ$, so that one obtains
a natural map
\[\colimalpha \rmH^s_{Nis}(\rmN^{\rmZ}(\rmU_{\bullet}^{\alpha}), \pi_{\rm t}(\rmM_{|\rmX})) \ra \colimalpha \rmH^s_{Nis}(\rmN^{\rmZ}(\rmU_{\bullet}^{\alpha}){\underset {\rmX} \times} \rmZ, \pi_{\rm t}(\rmM_{|\rmX})) \cong \rmH^s_{Nis}(\rmZ, \pi_t(i^{-1}(\rmM_{|\rmX}))).\]
This provides a map between the first two spectral sequences. To prove that this map will be an isomorphism at the $\rmE_2$-terms, exactly the same
arguments as in the proof of \cite[Theorem 1.3]{Cox} carry over from the \'etale framework to the Nisnevich framework, we are considering. 
(One may observe that \cite[Theorem 1.3]{Cox} depends strongly on \cite[Lemma 1.2]{Cox}, which is the precise analogue of Lemma ~\ref{Cox.lemma1.2}.)
Now it is clear that $\rmE_2^{s,t}=0$, for all $s> dim_k(\rmZ)$, so that both these spectral sequences converge strongly providing the 
required isomorphism at the abutments. Observe that the stalk of $\pi_t(\rmM_{|\rmX})$ at a point $\z \eps \rmZ$ identifies with $\pi_t(\Gamma(Spec \, \O_{\rmX, \z}^h, \rmM))$, so that 
making use of the Godement resolution, we obtain the identification 
\be \begin{equation}
\label{rigid.identity.1}
\rmH^s_{Nis}(\rmZ, \pi_{\rm t}(i^{-1}\rmM_{|\rmX})) \cong \rmH^s_{Nis}(\rmX^h_{\rmZ}, \pi_{\rm t}(\rmM_{|\rmX})).
\end{equation} \ee
\vskip .1cm
This proves the first statement. Next we consider the second statement.
Observe that there is a map from the second spectral sequence in ~\eqref{comp.ss.1} to the spectral sequence in ~\eqref{comp.ss.2}. As both spectral
 sequences converge strongly,  it suffices to show that the obvious map of sheaves $\pi_{\rm t}(\rmM_{|\rmX}) \ra \pi_{\rm t}(\rmM_{|\rmZ})$ is an isomorphism stalk-wise at every point of $\rmZ$.
 As observed above,  the stalk of $\pi_t(\rmM_{|\rmX})$ at a point $\z \eps \rmZ$ identifies with 
 $\pi_{\rm t}(\Gamma({\rm Spec} \, \O_{\rmX, \z}^h, \rmM))$. The stalk of $\pi_{\rm t}(\rmM_{|\rmZ})$ 
   at the same point $\z$ identifies with 
 $\pi_{\rm t}(\Gamma({\rm Spec} \, \O_{\rmZ, \z}^h,  \rmM))$. By the assumed rigidity property of $\rmM$, both of the above groups identify with 
 $\pi_{\rm t}(\Gamma({\rm Spec}\, \k(\z), \rmM))$. Therefore, the
 required isomorphism follows from the assumed rigidity property of the spectrum $\rmM$ and the isomorphism in ~\eqref{rigid.identity.1}.
 This completes the proof of the second statement.
 The third statement now follows in view of
the Quillen-equivalence of model categories between ${\widetilde \Spt}^{\rmG}(\k)$ and the model category of motivic spectra 
established in \cite[Proposition 6.2]{CJ23-T1}.
\vskip .1cm
Next we consider the statement in (iv). One can see that essentially the same spectral sequences exist on the \'etale site: their strong
convergence is guaranteed by the assumption that the base field $\k$ has finite $\ell$-cohomological dimension.  Now, the main point is to show
that one obtains a weak-equivalence 
\be \begin{equation}
\label{Hensel.et}
 \H_{et}(\rmX_{\rmZ}^h, \epsilon^*(\rmM)) \simeq \H_{et}(\rmZ, \epsilon^*(\rmM)).
\end{equation} \ee
 For a smooth scheme $\rmY$, we let $\epsilon^*(\rmM)_{|\rmY}$ denote the restriction of $\epsilon^*(\rmM)$ to the {\it small} \'etale site of the scheme $\rmY$. 
Given a presheaf $\rmP$ on the small \'etale site of the scheme $\rmX$, we will let $i^{-1}(\rmP)$ denote the
 restriction of $\rmP$ to the small \'etale site of the closed subscheme $\rmZ$.
As the space underlying the scheme $\rmX^h_{\rmZ}$ is just the space underlying the scheme $\rmZ$, the left-hand-side of ~\eqref{Hensel.et}
identifies with $\H_{et}(\rmZ, i^{-1}(\epsilon^*(\rmM)_{|\rmX}))$. The right-hand-side of ~\eqref{Hensel.et} identifies with 
$\H_{et}(\rmZ, \epsilon^*(\rmM)_{|\rmZ})$. Now the stalk of $i^{-1}(\epsilon^*(\rmM)_{|\rmX})$ at a geometric point $\bar \z$, corresponding to a point $\z \in \rmZ$, is given 
by $\Gamma({\rm Spec} \, (\O_{\rmX,\z}^{sh}), \rmM)$ while the stalk of $\epsilon^*(\rmM)_{|\rmZ}$ at the same geometric point $\bar z$
is given by $\Gamma ({\rm Spec} \, (\O_{\rmZ, \z}^{sh}), \rmM)$. By the assumed rigidity property of $\rmM$, both of these
identify with $\Gamma({\rm Spec} \, ({\overline {k(\z)}}), \rmM)$, where ${\overline {k(\z)}}$ denotes the separable closure of $ {k(\z)}$. This then provides the required 
weak-equivalence of the \'etale hypercohomology spectra in ~\eqref{Hensel.et}, as the corresponding spectral sequences that compute the homotopy groups of the hypercohomology spectra converge strongly.
\end{proof}
\section{\bf More on Nisnevich neighborhoods}\index{rigidity: model structures}
\begin{lemma} \index{Nisnevich neighborhood}
 \label{local.struct}
  Let $i: \rmZ \ra \rmX$ denote a closed immersion of smooth schemes of finite type over $k$ of pure codimension $c$ and $\rmX$ is of pure dimension $n$. Then the following hold.
  \begin{enumerate}[\rm(i)]
   \item For every point $\z \eps \rmZ$, there exists a Zariski neighborhood $\rmU_{\z}$ of $\z$ in $\rmX$ 
   and an \'etale map $q_{\z}: \rmU_{\z} \ra {\mathbb A}^n$, so that one has the 
 cartesian square:
 \[\xymatrix{{\rmU_{\z} \cap \rmZ} \ar@<1ex>[r] \ar@<1ex>[d]^{q'_{\z}} & {\rmU_{\z}} \ar@<1ex>[d]^{q_{\z}}\\
             {{\mathbb A}^{n-c}} \ar@<1ex>[r] & {{\mathbb A}^n}.}
 \]
 \item For every point $\z \eps \rmZ$, there exists a commutative square 
 \[\xymatrix{{\rmV_{\z}} \ar@<1ex>[r] \ar@<1ex>[d]^{q_{\rmV_{\z}}} & {\rmU_{\z}} \ar@<1ex>[d]^{q_{\z}}\\
              {(\rmU_{\z} \cap \rmZ) \times {\mathbb A}^c} \ar@<1ex>[r]^{p_{\z}} & {{\mathbb A}^n}}
 \]
 so that $\rmV_{\z} {\underset {(\rmU_{\z} \cap \rmZ) \times {\mathbb A}^c} \times} ((\rmU_{\z} \cap \rmZ)\times \{0\}) \cong (\rmU_{\z} \cap \rmZ)$ and $\rmV_{\z}{\underset {\rmU_{\z}} \times}(\rmU_{\z} \cap \rmZ) \cong \rmU_{\z} \cap \rmZ$,
 that is, $\rmV_{\z}$ is a Nisnevich neighborhood of $(\rmU_{\z} \cap \rmZ) \times \{0\}$ in $(\rmU_{\z}\cap \rmZ)\times {\mathbb A}^c$ and that $\rmV_{\z}$ is a Nisnevich neighborhood of 
 $\rmU_{\z} \cap \rmZ$  in $\rmU_{\z}$.
 \end{enumerate}
\end{lemma}
\begin{proof}
For each $y \eps \rmZ$, one knows by \cite[IV.4, 17.12.2]{EGA}
 that there exists a Zariski open neighborhood $\rmU_{\z}$ of $y$ in $\rmX$ and an \'etale map $q_{\z}: \rmU_{\z} \ra {\mathbb A}^n$
so that one obtains the first cartesian square in the lemma.
\vskip .1cm
 Let $p_{\z}= q'_{\z} \times id: (\rmU_{\z} \cap \rmZ) \times {\mathbb A}^c \ra {\mathbb A}^{n-c} \times {\mathbb A}^c = {\mathbb A}^n$ denote 
 the induced map. Let $\rmV'_{\z}$ be defined by the cartesian square:
 \[\xymatrix{{\rmV'_{\z}} \ar@<1ex>[r] \ar@<1ex>[d]^{q'_{U_{\z}}} & {\rmU_{\z}} \ar@<1ex>[d]^{q_{\z}}\\
              {(\rmU_{\z} \cap \rmZ) \times {\mathbb A}^c} \ar@<1ex>[r]^{p_{\z}} & {{\mathbb A}^n},}
 \]
that is, $\rmV'_{\z} = \rmU_{\z}{\underset {{\mathbb A}^n} \times} ((\rmU_{\z} \cap \rmZ) \times {\mathbb A}^c)$. Now one may observe the commutative diagram 
\[\xymatrix{{(\rmU_{\z} \cap \rmZ ){\underset {{\mathbb A}^{n-c}} \times}(\rmU _{\z} \cap  \rmZ)} \ar@<1ex>[r] \ar@<1ex>[d]& {(\rmU_{\z} \cap \rmZ)} \ar@<1ex>[r]^i \ar@<1ex>[d]^{q'_{\z}} & {\rmU_{\z}} \ar@<1ex>[d]^{q_{\z}}\\
            {(\rmU_{\z} \cap \rmZ)} \ar@<1ex>[r]^{q'_{\z}} & {{\mathbb A}^{n-c}} \ar@<1ex>[r] &{{\mathbb A}^n}}
\]
where both the squares are cartesian, which provides the isomorphism: $(\rmU_{\z} \cap \rmZ){\underset {{\mathbb A}^{n-c}} \times} (\rmU_{\z} \cap \rmZ) \cong \rmU_{\z} {\underset {{\mathbb A}^n} \times} (\rmU_{\z} \cap \rmZ)$. We call this scheme
 $\rmW'_{\z}$. Then one observes the isomorphisms:
\be \begin{align}
\label{key.ident.deform.norm.cone}
\rmV'_{\z} {\underset {(\rmU_{\z} \cap \rmZ) \times {{\mathbb A}^c}} \times}((\rmU \cap  \rmZ) \times \{0\}) &\cong \rmU_{\z} {\underset {{\mathbb A}^n} \times} ((\rmU_{\z} \cap \rmZ )\times {{\mathbb A}^c}) {\underset {(\rmU_{\z} \cap \rmZ) \times {{\mathbb A}^c}} \times} ((\rmU_{\z} \cap \rmZ)\times \{0\}) \\
\cong \rmU_{\z} {\underset {{\mathbb A}^n} \times}((\rmU_{\z} \cap  \rmZ)\times \{0\}) &= (\rmU_{\z} \cap \rmZ){\underset {{\mathbb A}^{n-c}} \times}(\rmU_{\z} \cap  \rmZ) =\rmW'_{\z} \notag.
\end{align} \ee
Next one observes that the map $q'_{\z}: (\rmU_{\z} \cap \rmZ) \ra {\mathbb A}^{n-c}$ is
\'etale, which implies the diagonal map $\Delta: (\rmU_{\z} \cap \rmZ) \ra (\rmU_{\z} \cap \rmZ){\underset {{\mathbb A}^{n-c}} \times} (\rmU_{\z} \cap \rmZ)$ is an open immersion. 
One may observe that 
$(\rmU_{\z} \cap \rmZ){\underset {{\mathbb A}^{n-c}} \times} (\rmU_{\z} \cap \rmZ) \times\{0\}$ is closed in $\rmV'_{\z}$.
Let $\rmZ_{\z}$ denote
$ (\rmU_{\z} \cap \rmZ){\underset {{\mathbb A}^{n-c}} \times} (\rmU_{\z} \cap \rmZ) - \Delta (\rmU_{\z} \cap \rmZ)$, which is therefore closed in $\rmV'_{\z}$ and $\rmW'_{\z}$. Let $\rmV_{\z} = \rmV'_{\z} - \rmZ_{\z}$ and $\rmU_{\z} \cap \rmZ = \rmW'_{\z} -\rmZ_{\z}$. 
Then one may see that, with the above choice of $\rmV_{\z}$, one obtains the commutative square in (ii).
 \end{proof}

\section{\bf Applications of the Additivity (and Multiplicativity) of the transfer}
\vskip .3cm
Classically, several of the applications of the transfer, such as various double coset formulae for actions of compact Lie groups were first obtained by special
arguments, such as in \cite{Fesh}. The work \cite{LMS} showed that all such results could be deduced by proving the additivity of
the transfer.  In fact, the full strength of the Becker-Gottlieb transfer and its full range of applications stem from the additivity
of the pre-transfer, the transfer and the trace. 
\vskip .2cm
In the present section, making use of the motivic and \'etale Becker-Gottlieb transfer constructed in \cite{CJ23-T1}, and with the additivity for the 
transfer and trace established in section 4, we carry out a similar
program in the motivic and \'etale framework. The analogue of the statement that the Euler characteristic of $\rmG/\NT$ is $1$
in singular cohomology for compact Lie groups is a conjecture due to Morel, that a suitable motivic Euler characteristic
in the Grothendieck-Witt group is $1$, for $\rmG/\NT$, where $\rmG$ is a split connected reductive group and $\NT$ is the 
normalizer of a maximal torus in $\rmG$. As pointed out above, this was proven in \cite[Theorem 1.2]{JP23}, under the 
hypothesis that the base field has a square root of $-1$, by deducing it from
the {\it additivity of the motivic trace} and then proving such an additivity theorem for the motivic trace.
\vskip .2cm
\subsection{The generic torus slice theorem and applications}
 We begin by discussing the following Proposition, which seems to be rather well-known. (See for example, \cite[Proposition 4.10]{Th86} or
 \cite[(3.6)]{BP}.)
\begin{proposition} \index{generic torus slice theorem}
 \label{struct.T.actions}
Let $\rmT$ denote a split torus acting on a smooth scheme $\rmX$ all defined over the given perfect base field $k$. 
\vskip .1cm
Then the following hold.
\vskip .1cm
$\rmX$ admits a  decomposition into a disjoint union of finitely many locally closed, $\rmT$-stable subschemes $\rmX_j$ so that
 \be \begin{equation}
    \label{torus.decomp}
    \rmX_j \cong (\rmT/\Gamma _j) \times \rmY_j.
 \end{equation} \ee
 \vskip .2cm \noindent
 Here each $\Gamma_j$ is a sub-group-scheme of $\rmT$, each $\rmY_j$ is a scheme of finite type over $k$ which is also regular and on which $\rmT$ acts trivially with the isomorphism in ~\eqref{torus.decomp} being $\rmT$-equivariant. 
\end{proposition}
\begin{proof}
One may derive this from the generic torus slice theorem proved in \cite[Proposition 4.10]{Th86}, which says that if a split torus acts on a reduced separated  scheme of finite type over a perfect field, then the following are satisfied:
 \begin{enumerate}
   \item there is an open subscheme $\rmU$ which is regular and stable under the $\rmT$-action
   \item a geometric quotient $\rmU/\rmT$ exists, which is a regular scheme of finite type over $k$
   \item $\rmU $ is isomorphic as a $\rmT$-scheme to $\rmT/\Gamma \times \rmU/\rmT$ where $\Gamma$ is a diagonalizable subgroup scheme of $\rmT$ and $\rmT$ acts trivially on $\rmU/\rmT$. 
 \end{enumerate}
(See also \cite[(3.6)]{BP} for a similar decomposition.) 
\end{proof}
Next we consider the following theorem.
 \begin{theorem} \index{motivic trace: torus actions}
 \label{torus.act} 
 We will assume throughout this theorem that the base field $\k$ is infinite and contains a $\sqrt{-1}$. Under the assumption that the base field $k$ is of characteristic $0$,  the following hold, where $\tau_{\rmX_+}$ denotes the trace associated to the scheme $\rmX$:
  \begin{enumerate}[\rm(i)]
      \item $\tau_{{{\mathbb G}_{\it m}}_+}=0$ in the Grothendieck-Witt ring of the base field $k$. More generally, if $\rmT$ is a split torus, $\tau_{\rmT_+}=0$  in the Grothendieck-Witt ring of $k$.
      \item Let $\rmT$ denote a split torus acting on a smooth scheme $\rmX$. Then $\rmX^{\rmT}$ is also 
            smooth, and $\tau_{\rmX_+}= \tau_{\rmX^{\rmT}_+}$ in the Grothendieck-Witt ring of $k$. 
  \end{enumerate}      
 \vskip .1cm \noindent
    If the base field is of positive characteristic, the corresponding assertions hold
      with the Grothendieck-Witt ring of $k$ replaced by the Grothendieck-Witt ring of $k$ with the prime $p$ inverted.
  \vskip .2cm \noindent  
  Assume $\rmM$ denotes a motivic spectrum that has the rigidity property as in Definition ~\ref{rigid.prop} and that $\rmT = {\mathbb G}_{\it m}$. Then if ${\rm RHom}(\quad, \rmM)$ denotes the derived (external) hom in the category of spectra, 
   \begin{enumerate}
  \item [\rm (iii)]  ${\rm RHom}(j \circ \tr_{{{\mathbb G}_{\it m}}_+}', \rmM)$ is trivial, where $j: {\mathbb G}_{\it m} \ra {\mathbb A}^1$ in the open immersion.
  \item [\rm (iv)] Let $\rmT$ act on a smooth scheme $\rmX$ so that for each $\rmT$-orbit $\rmT/{\rm \Gamma}_j$  with the orbit $\rmT/{\rm \Gamma}_j \cong {\mathbb G}_{\it m}$, the locally closed immersion
$\rmT/{\rm \Gamma}_j \times \rmY_j \ra \rmX$ as in ~\eqref{torus.decomp}  factors through a map ${\mathbb A}^1 \times \rmY_j \ra \rmX$. Then
   ${\rm RHom}(\tr_{\rmX_+}', \rmM)  \simeq {\rm RHom}({\it i} \circ \tr_{\rmX^{\rmT}_+}', \rmM)$, where $i: \rmX^{\rmT} \ra \rmX$ is the inclusion.
\item[\rm (v)] Moreover, under the assumptions of {\rm (iv)},  if $\rmG$ is a linear algebraic group that is special and acting on the scheme $\rmX$ commuting with the action of a split torus $\rmT$ so that the decomposition of $\rmX$
in ~\eqref{torus.decomp} is $\rmG$-stable, and $\rmE \ra \rmB$ is 
a $\rmG$-torsor, 
 then $tr_{\rmX}^*= {\rm RHom}(\rmE{\underset {\rmG} \times}\tr_{\rmX}'^{\rmG}, \rmM) \simeq  tr_{\rmX^{{\mathbb G}_{\it m}}}^* \circ {\it i}^* ={\rm RHom}(i \circ (\rmE{\underset {\rmG} \times}\tr_{\rmX^{\rmT}}'^{\rmG}), \rmM)$, where ${\it i}: \rmE{\underset {\rmG} \times}\rmX^{\rmT} \ra \rmE{\underset {\rmG} \times}\rmX$ is the inclusion.
   \end{enumerate}
 \end{theorem}
 \begin{proof} First observe from Definition ~\ref{pretransfer.3}, that the trace $\tau_{\rmX}$ associated to any smooth scheme $\rmX$ is a map
 $\mbS_{\k} \ra \mbS_{\k}$: as such, we will identify $\tau_{\rmX}$ with the corresponding class $\tau_{\rmX}^*(1)$ in the Grothendieck Witt-ring of the base field.
 We will only consider the proofs when the base field is of characteristic $0$, since the proofs in the positive characteristic case are entirely similar. However, it is important to point out that in positive characteristics $p$, it is important to invert $p$: for otherwise, one no longer has a theory of Spanier-Whitehead duality.
 \vskip .1cm 	
 (i) and (iii). We observe that the scheme ${\mathbb A}^1$ is the disjoint union of 
 the closed point $\{0\}$ and ${\mathbb G}_{\it m}$. If $i_1: \{0\} \ra {\mathbb A}^1$ and $j_1: {\mathbb G}_{\it m} \ra {\mathbb A}^1$ are the corresponding immersions, \cite[Theorem 2.9(ii) and (iii)]{JP23} and Theorem ~\ref{add.transf}(ii) and (iii)
  show that
 	\be\begin{equation}
 	  \label{torus.1}
 	   \tau_{{\mathbb A}^1_+} = \tau_{\{0\}_+} +  \tau_{{{\mathbb G}_{\it m}}_+} \mbox { and } {\rm RHom}(\tr_{{\mathbb A}^{\rm 1}_+}', \rmM) = {\rm RHom}(i_{\rm 1} \circ \tr_{\{{\rm 0 }_+\}}', \rmM) +   {\rm RHom}(j_{\rm 1} \circ tr_{{{\mathbb G}_{\it m}}_+}', \rmM).
 	\end{equation} \ee
 \vskip .1cm \noindent
 However, by ${\mathbb A}^1$-contractibility, $\tau_{{\mathbb A}^1_+} = \tau_{\{0\}_+}$ and $\tr_{{\mathbb A}^{\rm 1}_+}' = i_{\rm 1} \circ tr_{\{{\rm 0 }\}_+}'$. One may readily see this from the
 definition of the pre-transfer as in  Definition ~\ref{pretransfer.3}, which shows that both the
 pre-transfer $\tr_{\rmC_+}' = \tr_{\rmC_+}'(id)$ and hence the corresponding trace, $\tau_{\rmC_+} $ depend on $\rmC_+$ only up to its class in the motivic stable homotopy category.
 Therefore, $\tau_{{{\mathbb G}_{\it m}}_+}=0$ and ${\rm RHom}(j_1 \circ \tr_{{{\mathbb G}_{\it m}}_+}', \rmM)$ is trivial.
  Since $\rmT$ is a split torus, we may assume $\rmT = {\mathbb G}_{\it m} ^n$ for some 
 positive integer $n$. Then the multiplicative property of the trace and pre-transfer (see Proposition ~\ref{mult.prop}) prove that
  $\tau_{\rmT_+}=0$. These complete the proof of statements (i) and (iii).
 \vskip .2cm
 Therefore, we proceed to prove the statement in (ii) and (iv). First, we invoke Proposition ~\ref{struct.T.actions}
to conclude that $\rmX^{\rmT}$ is the disjoint union of the schemes $\rmX_j$ for which $\Gamma _j = \rmT$. 
 \vskip .1cm
 Let $i_j: \rmX_j \cong (\rmT/\Gamma _j) \times \rmY_j \ra \rmX$ denote the locally closed immersion.
 Next observe that the additivity of the trace proven as in  \cite[Theorem 2.9(ii) and (iii)]{JP23}, the additivity of pre-transfer proven in Theorem ~\ref{add.transf},  and the multiplicativity of the pre-transfer and trace proven in  
  \cite[Proposition 2.8]{JP23} along with the decomposition in  ~\eqref{torus.decomp} show that
 \be \begin{align}
    \label{tr.T}
    \tau_{\rmX_+} = {\Sigma} _j \tau_{\rmX_{j+}} &= {\Sigma} _j  (\tau_{\rmT/\Gamma_j+}) \wedge \tau_{\rmY_{j+}} \mbox{ and }\\
    {\rm RHom}(\tr_{\rmX_+}', \rmM)  \simeq {\Sigma} _j {\rm RHom}(i_j \circ tr_{\rmX_{j+}}', \rmM) &= {\Sigma} _j {\rm RHom}(i_j \circ (\tr_{\rmT/\Gamma_j+}'\wedge \tr_{\rmY_{j+}}'), \rmM). \notag
 \end{align} \ee
 \vskip .2cm
 Now statements (i) and (iii) in the theorem along with the assumptions in (iv) prove that the $j$-th summand on the right-hand-sides are trivial unless $\Gamma _j = \rmT$. But, then $\rmX^{\rmT}$ is the disjoint union of such $\rmX_j$. Finally the additivity of the trace and pre-transfer in 
 \cite[Theorem 2.9(ii) and (iii)]{JP23} and Theorem ~\ref{add.transf} applied once more to $\rmX^{\rmT}$ proves the sum of the non-trivial terms on the right-hand-side is 
 $\tau_{\rmX^{\rmT}}$ for the first equation and is given by ${\rm RHom}(i \circ \tr_{\rmX^{\rmT}_+}', \rmM)$ for the second equation.
 These prove the statements in (ii) and (iv).
 \vskip .1cm
 Finally, we consider the last statement. In view of the assumption that the actions by the linear
 algebraic group $\rmG$ and the split torus $\rmT$ on the scheme $\rmX$ commute and that the
  decomposition of $\rmX$ into the schemes $\rmX_j$ as in ~\eqref{torus.decomp} is stable by the action of
  $\rmG$, the weak-equivalence ${\rm RHom}(\tr_{\rmX_+}', \rmM) \simeq {\rm RHom}(i \circ \tr'_{\rmX^{\rmT}_+}, \rmM)$ obtained in (iv) implies the
   weak-equivalence ${\rm RHom}(\rmE{\underset {\rmG} \times}\tr_{\rmX_+}'^{\rmG}, \rmM) \simeq {\rm RHom}(i \circ (\rmE{\underset {\rmG} \times}\tr_{\rmX^{\rmT}_+}'^{\rmG}), \rmM)$ in (v), in view of Theorem ~\ref{add.transf}(iii) and (iv). (In fact, one may adopt an argument
   as in ~\eqref{EP} using a cover $\U=\{\rmU_i|i\}$ of $\rmB$ on which the torsor $\rmp: \rmE \ra \rmB$ is trivial.)
 \end{proof}
 \begin{remark}
  Here we provide an explanation of the condition in Theorem ~\ref{torus.act}(iv). The first observation is that then the origin in ${\mathbb A}^1$ corresponds
   to a fixed point $x$ for the ${\mathbb G}_{\it m}$-action on $\rmX$. The corresponding ${\mathbb G}_{\it m}$-orbit is then contained in a {\it slice} at $x$.
   The condition in Theorem ~\ref{torus.act}(iv) may now be interpreted as saying the fixed point $x$ is {\it an attractive fixed point} for the
   ${\mathbb G}_{\it m}$-action on $\rmX$, in the sense that all the weights for the induced ${\mathbb G}_{\it m}$ action on the Zariski tangent space $\rmT_{\it x}$ at $x$ lie
   in an open half-space. (See \cite[2.2, 2.3]{BJ} for further details.)
 \end{remark}

 \begin{corollary}
 	\label{equiv.fixed.pts}
 	Let $\rmX$ denote a smooth scheme provided with the action of a linear algebraic group $\rmG$ which we will assume is also special. Assume that 
 	 $\rmX$ is also provided with an action by ${\mathbb G}_{\it m}$ commuting 
 	with the action by $\rmG$ and that the hypotheses in Theorem ~\ref{torus.act}{\rm (v)} hold with $\rmT = {\mathbb G}_{\it m}$.  Let $\rmM$ denote a fibrant motivic spectrum that has the rigidity property. Then,
 	adopting the
 terminology as in \cite[8.3]{CJ23-T1},  that for a linear algebraic group $\rmG$, $\BG = \colimm \BG^{\it gm,m}$ and $\EG = \colimm \EG^{\it gm,m}$,
 	one obtains the homotopy commutative diagram
 	\[\xymatrix{ {\h(\EG{\underset {\rmG} \times}\rmX, \rmM)} \ar@<1ex>[r]^{i^*} \ar@<1ex>[d]^{tr_{\rmX}^*} &  {\h(\EG{\underset {\rmG} \times}\rmX^{{\mathbb G}_{\it m}}, \rmM)}  \ar@<1ex>[dl]^{tr_{\rmX^{{\mathbb G}_{\it m}}}^*} \\
 		{\h(\BG, \rmM)}} \]
 where $\h(\quad, \rmM)$ denotes the hypercohomology spectrum with respect to the motivic spectrum $\rmM$ and $i: \EG{\underset {\rmG} \times}\rmX^{{\mathbb G}_{\it m}} \ra \EG{\underset {\rmG} \times}\rmX$ is the map induced by the closed immersion $\rmX^{{\mathbb G}_{\it m}} \ra \rmX$.		
 \end{corollary}
 \begin{proof} We will show that there is a corresponding commutative diagram, when $\BG$ and $\EG$ are replaced by their finite
 dimensional approximations $\BG^{\it gm,m}$ and $\EG^{\it gm,m}$. Therefore let $m$ denote a fixed non-negative integer and let
 $\BG^{\it gm,m}$ denote the approximation of $\BG^{gm}$ to degree $m$ and let $\EG^{\it gm,m}$ denote its universal principal $\rmG$-bundle.
 \vskip .1cm
 Next we observe that $\rmX$ admits the decomposition $\rmX = (\rmX-\rmX^{{\mathbb G}_{\it m}}) \sqcup \rmX^{{\mathbb G}_{\it m}}$ and that this
 decomposition is stable under the action of $\rmG$ (as the action of $\rmG$ and ${\mathbb G}_{\it m}$ are assumed to commute). Moreover, 
 $\rmX- \rmX^{{\mathbb G}_{\it m}} \cong {\mathbb G}_{\it m} \times \rmY$, where $\rmY$ in fact denotes the geometric quotient $(\rmX - \rmX^{{\mathbb G}_{\it m}})/{\mathbb G}_{\it m}$.
By Theorem ~\ref{torus.act}(v),
 one obtains the identification of the $\rmG$-equivariant transfers
 \[tr_{\rmX}^*={\rm RHom}(\EG^{\it gm,m}{\underset {\rmG} \times}{\tr'}_{\rmX}^{\rmG}, \rmM) = {\rm RHom}(\EG^{\it gm,m}{\underset {\rmG} \times}({\it i} \circ {\tr'}_{\rmX^{{\mathbb G}_{\it m}}}^{\rmG}), \rmM) ={\it tr}_{\rmX^{{\mathbb G}_{\it m}}}^* \circ {\it i}^*.\]
Finally, taking the homotopy inverse limit as $m \ra \infty$ provides the homotopy commutative triangle in the corollary.
\end{proof}\qed
\vskip .2cm
In the following Proposition, we assume that the scheme $\rmX$ provided with commuting actions by a connected linear algebraic group $\rmG$ and 
a $1$-dimensional torus ${\mathbb G}_{\it m}$ is also {\it projective and smooth}. This enables us to draw the same conclusions as in Corollary
~\ref{equiv.fixed.pts} with less stringent hypotheses.
\begin{proposition}
 \label{fixed.pts.proj.case} \index{transfer: fixed points for torus actions}
 Let $\rmX$ denote a smooth {\it projective} variety over $\k$ provided with the action of a connected linear algebraic group $\rmG$, which we will assume is also special. Assume 
 	 that $\rmX$ is also provided with an action by ${\mathbb G}_{\it m}$ commuting 
 	with the action by $\rmG$. We will further assume that the base field $k$ is infinite and that it contains a $\sqrt{-1}$. Let $\rmM$ denote a motivic spectrum that has the rigidity property. Then,
adopting the
 terminology as in \cite[8.3]{CJ23-T1},  that for a linear algebraic group $\rmG$, $\BG = \colimm \BG^{\it gm,m}$ and $\EG = \colimm \EG^{\it gm,m}$,
 	one obtains the homotopy commutative diagram
 	\[\xymatrix{ {\rmh^{*, \bullet}(\EG{\underset {\rmG} \times}\rmX, \rmM)} \ar@<1ex>[r]^{i^*} \ar@<1ex>[d]^{tr_{\rmX}^*} &  {\rmh^{*, \bullet}(\EG{\underset {\rmG} \times}\rmX^{{\mathbb G}_{\it m}}, \rmM)}  \ar@<1ex>[dl]^{tr_{\rmX^{{\mathbb G}_{\it m}}}^*} \\
 		{\rmh^{*, \bullet}(\BG, \rmM)}} \]
 where $\rmh^{*, \bullet}(\quad, \rmM)$ denotes the cohomology with respect to the motivic spectrum $\rmM$ and $i: \EG{\underset {\rmG} \times}\rmX^{{\mathbb G}_{\it m}} \ra \EG{\underset {\rmG} \times}\rmX$ is the map induced by the closed immersion $\rmX^{{\mathbb G}_{\it m}} \ra \rmX$.		
 \end{proposition}
 \begin{proof} \index{transfer: fixed points for torus actions}
In view of the assumption that $\rmX$ is projective, we invoke the Bialynicki-Birula decomposition of $\rmX$ into
finitely many locally closed subschemes $\rmX_{\alpha}^+$, 
so that each $\rmX_{\alpha}^+$ is an affine space bundle on $\rmX_{\alpha}$, which is a connected component of the fixed point scheme 
$\rmX^{{\mathbb G}_{\it m}}$. (See  \cite[Theorem 2.1]{dB}, \cite{B-B}.) Since the actions of $\rmG$ and ${\mathbb G}_{\it m}$ commute, and $\rmG$ is connected, 
each connected component of the
fixed point scheme $\rmX^{{\mathbb G}_{\it m}}$ is stable by $\rmG$. Therefore, it follows that each of the $\rmX_{\alpha}^+$ is also stable by $\rmG$.
\vskip .1cm
In view of the assumed rigidity property for the spectrum $\rmM$, Theorem ~\ref{add.transf}(iv)
shows that 
\[tr_{\rmX}^* = {\rm RHom}(\EG^{\it gm,m}{\underset {\rmG} \times}{\tr'}_{\rmX}^{\rmG}, \rmM) = {\Sigma}_{\alpha} tr_{X_{\alpha}^+}^* = {\Sigma}_{\alpha}{\rm RHom}(\EG^{\it gm,m}{\underset {\rmG} \times}({\it i_{\alpha}^+} \circ {\tr'}_{\rmX_{\alpha}^+}), \rmM).\]
In view of the observation that each $\rmX_{\alpha}^+ \ra \rmX_{\alpha}$ is an affine-space  bundle, the last sum
identifies with 
\[{\Sigma}_{\alpha}{\rm RHom}(\EG^{\it gm,m}{\underset {\rmG} \times}({\it i_{\alpha}} \circ {\tr'}_{\rmX_{\alpha}}), \rmM) = {\rm RHom}(\EG^{\it gm,m}{\underset {\rmG} \times}({\it i} \circ {\tr'}_{\rmX^{{\mathbb G}_{\it m}}}), \rmM) = tr_{\rmX^{{\mathbb G}_{\it m}}}^* \circ i^*,\]
where $i_{\alpha}^+: \rmX_{\alpha}^+ \ra \rmX$, $i_{\alpha}: \rmX_{\alpha} \ra \rmX$ and $i: \rmX^{{\mathbb G}_{\it m}} \ra \rmX$ are the locally closed
immersions. Now one takes the colimit as $m \ra \infty$, to  complete the proof. \qed
\end{proof}
\begin{remark} An example of the situation considered in the above corollary is the following. Let $\rmX={\rm GL}_{n+1}/\rmB_{n+1}$, which is the variety of all Borel subgroups in ${\rm GL}_{n+1}$.
 Let ${\mathbb G}_{\it m}$ denote the $1$-parameter subgroup imbedded in ${\rm GL}_{n+1}$ as the diagonal matrices of the form $\rmI_n \times {\mathbb G}_{\it m}$, with ${\mathbb G}_{\it m}$ appearing in the $(n+1, n+1)$-position. 
 Then consider the action of this ${\mathbb G}_{\it m}$ by conjugation on $\rmX$. Then let $\rmG = {\rm GL}_n$ acting by conjugation on $\rmX$: then the
  actions by $\rmG$ and ${\mathbb G}_{\it m}$ commute. 
\end{remark}
 \subsection{\bf Double Coset formulae}
 \vskip .1cm 
 In this section, we establish various double coset formulae, the analogues of which have been known in the setting of group cohomology for finite groups and also for compact Lie groups. We will explicitly consider only the motivic framework, since the corresponding results in the \'etale framework may be established by entirely similar arguments.
 \vskip .1cm
 The main context in which we consider double coset formulae
 will be as follows. Let $\rmG$ denote a linear algebraic group, and let $\rmH$, $\rmK$ denote two closed linear algebraic subgroups. 
 We will further assume that the group $\rmG$ is {\it special}, when dealing with motivic contexts. This is mainly to keep our discussion
 simpler: on considering the \'etale contexts, i.e., generalized \'etale cohomology theories, clearly there is no need to make this assumption.
 \vskip .1cm
 Let $\rmX$ denote a smooth $\rmG$-scheme. Then, adopting the
 terminology as in \cite[8.3]{CJ23-T1},  that for a linear algebraic group $\rmH$, $\BH = \colimm \BH^{\it gm,m}$ and $\EH = \colimm \EH^{\it gm,m}$, we obtain the cartesian squares, where the map $\rmp_K$ is induced
 by the inclusion $\rmK \ra \rmG$ and the maps $\pi_{\rmH}$, $\tilde \pi_{\rmH}$ are induced by the projection $\rmG/\rmH \ra {\rm Spec} \, \k$:
 \be \begin{equation}
  \label{double.coset.0}
  \xymatrix{{\EK{\underset {\rmK} \times} \rmG/\rmH}   \ar@<1ex>[r]^{\tilde \rmp_K} \ar@<1ex>[d]^{\tilde \pi_H}  & {\EG{\underset {\rmG} \times} \rmG/\rmH} \ar@<1ex>[d]^{\pi_H} &\quad \quad {\EK{\underset {\rmK} \times} ({\rmG\times_{\rmH} \rmX})}   \ar@<1ex>[r]^{\tilde \rmp_K} \ar@<1ex>[d]^{\tilde \pi_H}  & {\EG{\underset {\rmG} \times} ({\rmG\times_{\rmH} \rmX})} \ar@<1ex>[d]^{\pi_H}\\
            {\BK} \ar@<1ex>[r]^{\rmp_K} & {\BG}                                                                                                                &\quad \quad                                                                              {\EK{\underset {\rmK} \times}\rmX } \ar@<1ex>[r]^{\rmp_K} & {\EG{\underset {\rmG} \times} \rmX  }.}
\end{equation} \ee
Observe that the first square is a special case of the second square by taking $\rmX = Spec \, \k$.
Now we make the following assumptions:
\begin{enumerate}[\rm(i)]
\item 
$\rmG/\rmH$ admits a {\it finite} decomposition $\rmG/\rmH = \sqcup_i \rmF_i$,
where each $\rmF_i$ is a locally closed and $\rmK$-stable smooth subscheme of $\rmG/\rmH$.
\item
Here we assume
that $\rmK$ acts on the left on $\rmG$ and on $\rmG/\rmH$, and $\rmH$ acts on the right on $\rmG$.
\end{enumerate}
In particular, it follows that each $\rmF_i$ is a
disjoint union of the double-cosets for the left-action of $\rmK$ on $\rmG$ and the right action of
$\rmH$ on $\rmG$. 
\begin{theorem}
	\label{double.coset.1} 
Assume the situation as in ~\eqref{double.coset.0}. \index{transfer: double coset formulae}
\begin{enumerate}[\rm(i)]
\item Let $\rmM$ denote a motivic spectrum and let $\rmh^{*, \bullet}(\quad, \rmM)$ denote the
  generalized cohomology defined with respect to the spectrum $\rmM$. Denoting the maps induced by the transfers
  \[\tr^{\rmG*}: \rmh^{*, \bullet}(\EG{\underset {\rmG} \times} ({\rmG\times_{\rmH} \rmX}), \rmM) \ra \rmh^{*, \bullet}({\EG{\underset {\rmG} \times} \rmX  }, \rmM), \quad \tr^{\rmK*}: \rmh^{*, \bullet}(\EK{\underset {\rmK} \times} ({\rmG\times_{\rmH} \rmX}), \rmM) \ra \rmh^{*, \bullet}(\EK{\underset {\rmK} \times}\rmX, \rmM),\]
  { and }
  \[\rmp_{\rmK}^*: \rmh^{*, \bullet}({\EG{\underset {\rmG} \times} \rmX  }, \rmM) \ra \rmh^{*, \bullet}(\EK{\underset {\rmK} \times}\rmX, \rmM), \quad {\tilde \rmp_K}^*: \rmh^{*, \bullet}(\EG{\underset {\rmG} \times}({\rmG\times_{\rmH} \rmX}), \rmM) \ra \rmh^{*, \bullet}(\EK{\underset {\rmK} \times} ({\rmG\times_{\rmH} \rmX}), \rmM)\]
 the corresponding pull-backs, we obtain:
 \be \begin{equation}
 \label{double.coset.1.eq}
 \rmp_{\rmK}^* \circ \tr^{\rmG*} = \tr^{\rmK*} \circ \tilde \rmp_{\rmK}^*.
 \end{equation} \ee
 \item
 Assume the base field $k$ is infinite and contains a $\sqrt{-1}$.  Let $\rmM$ denote a motivic spectrum that has the rigidity property as in Definition ~\ref{rigid.prop} and let $\rmh^{*, \bullet}(\quad, \rmM)$ denote the
  generalized cohomology defined with respect to the spectrum $\rmM$. Then, the map induced by the transfer 
 $\tr^{\rmK*} = \rmh^{*, \bullet}(\tr^{\rm K}, \rmM)$ admits a decomposition as $\Sigma _j \rmh^{*, \bullet}(i_j \circ \tr^{\rm K}_{\rmF_j}, \rmM)$, where $\tr^{\rm K*}_{\rmF_j}: \rmh^{*, \bullet}(\EK{\underset {\rmK} \times} (\rmF_j\times \rmX), \rmM) \ra \rmh^{*, \bullet}(\EK{\underset {\rmK} \times}\rmX, \rmM)$
  is the corresponding transfer and $i_j: \EK{\underset {\rmK} \times}(\rmF_j \times \rmX)  \ra
 \EK {\underset {\rmK} \times}(\rmG/\rmH \times \rmX) \cong \EK {\underset {\rmK} \times}(\rmG \times_{\rmH} \rmX) $ is the map induced by the inclusion $\rmF_j \ra \rmG/\rmH$.
\end{enumerate}
 \end{theorem}
 \begin{proof} \index{double coset formulae}
 (i) follows readily from the naturality of the transfer map established in \cite[Proposition 2.4]{CJ23-T2}. Next we consider (ii). This follows readily from Theorem ~\ref{add.transf}(iv), once we observe
	the isomorphism $\rmG\times_{\rmH}\rmX \cong \sqcup _j(\rmF_j \times \rmX)$ of $\rmK$-schemes. Here it is essential that
	$\rmX$ be a $\rmG$-scheme so that the action map induces a map $\rmG\times_{\rmH}\rmX \ra \rmX$. Then this map together
	with the projection $\rmG\times_{\rmH} \rmX \ra \rmG/\rmH$ induces a map $\rmG \times_{\rmH} \rmX \ra \rmG/\rmH \times \rmX$,
	which one can show is an isomorphism as both $\rmG \times _{\rmH} \rmX$ and $\rmG/\rmH \times \rmX$ fiber over $\rmG/\rmH$ with $\rmX$ as the fiber.
\end{proof}
\vskip .2cm \noindent
\begin{corollary} (Double Coset Formulae) 
	\label{double.coset.2}
	Let $  \rmh^{*, \bullet}$ denote a generalized cohomology theory defined with respect to
	a motivic spectrum $\rmM$. Assume  that the base field $k$ is infinite, contains a $\sqrt{-1}$, and that the spectrum $\rmM$ has the rigidity property as 
	in Definition ~\ref{rigid.prop}.
	\begin{enumerate}[\rm(i)]
	\item  Assume that $\rmG$ is a connected split reductive group, and that  $\rmK= \rmH= \rmT$ is a split maximal torus in $\rmG$. Let $\NT$ denote the normalizer of $\rmT$ and let $\rmW= \NT/\rmT$. 
	Then,  $tr^{\rmK*} \circ \tilde \rmp_{\rmK}^*$ (that is, the term appearing on the right-hand-side in ~\eqref{double.coset.1.eq}), identifies with
	 \[{\Sigma}_{ \rmw \eps \rmW}	C_{\rmw} \circ i_{\rmw}^*,\]
	 where 
	 $C_{\rmw}:   \rmh^{*, \bullet}(\EH\times_{\rmH}\rmX, \rmM) {\overset {\cong} \ra }  \rmh^{*, \bullet}(\EH^{\rmw}\times_{{\rmH}^{\rmw}}\rmX, \rmM)$ is the isomorphism induced by conjugation by $\rmw$, $\rmH ^{\rmw} = \rmw\rmH \rmw^{-1}$ and
	 $i_{\rmw}^*:    \rmh^{*, \bullet}(\EK^{\rmw}\times_{\rmK^{\rmw}}({\rmB}w\rmR_u(\rmB^{-}) \times \rmX), \rmM) \ra  
	\rmh^{*, \bullet}(\EK^{\rmw}\times_{\rmK^{\rmw}}(\rmG/\rmH \times \rmX), \rmM) \cong \rmh^{*, \bullet}(\EK^{\rmw}\times_{\rmK^{\rmw}}(\rmG\times_{\rmH} \rmX), \rmM)$ is the map induced by the 
	inclusion $i_{\rmw}: \rmB{\rmw}R_u(\rmB^{-}) \ra \rmG/\rmH$.
	\item
	 Suppose $\rmG$ is a connected split reductive group,  and that  $\rmH$ is a closed linear algebraic subgroup of $\rmG$ of maximal rank and $\rmK=\rmT$ is
	a split maximal torus in $\rmG$ and $\rmH$. We will further assume that $\rmH$ is either $\rmT$, a parabolic subgroup, or a Levi subgroup containing $\rmT$. Let $\rmW_{\rmG}$ ($\rmW_{\rmH}$) denote the Weyl group of $\rmG$ ($\rmH$, \res). In this case the right-hand-side of ~\eqref{double.coset.1.eq} may be written as
	\[{\Sigma}_{\rmw \eps \rmW_{\rmG}/\rmW_{\rmH}}  \circ C_{\rmw} \circ i_{\rmw}^*,\]
	where 
	$C_{\rmw}:   \rmh^{*, \bullet}(\EK\times_{\rmK}\rmX, \rmM) \ra   \rmh^{*, \bullet}(\EK^{\rmw}\times_{{\rmH}^{\rmw}}\rmX , \rmM) \cong \rmh^{*, \bullet}(\EK^{\rmw}\times_{{\rmH}^{\rmw}}(\rmS_{\rmw} \times \rmX) , \rmM)$ with
	$\rmS_{\rmw}$ being the stratum of $\rmG/\rmH$ that is indexed by $\rmw$ is as in (i), and 
	\[i_{\rmw}^*:    \rmh^{*, \bullet}(\EK^{\rmw}\times_{\rmK^{\rmw}}(S_{\rmw} \times \rmX), \rmM) \ra  \rmh^{*, \bullet}(\EK^{\rmw}\times_{\rmK^{\rmw}}(\rmG/\rmH \times \rmX), \rmM) \cong \rmh^{*, \bullet}(\EK^{\rmw}\times_{\rmK^{\rmw}}(\rmG\times_{\rmH} \rmX), \rmM)\]
	is the map induced by the 
	inclusion $i_{\rmw}: S_{\rmw} \ra \rmG/\rmH$, with $S_{\rmw}$ denoting the corresponding stratum indexed by $w \in \rmW$.
	\item
	Let $\cE$ denote a commutative ring spectrum in ${\Spt}(\k_{\rm mot})$, whose presheaves of homotopy groups are  all $\ell$-primary torsion for a
 fixed prime $\ell \ne char (k)$, and
let $\epsilon^*(\cE)$ denote the corresponding spectrum in $\Spt(\k_{et})$. 
 Assume that $\rmM$ is a module spectrum over $\cE$ that 
 has the rigidity property as in Definition ~\ref{rigid.prop}.  Then the 
 	results corresponding to (i) and (ii) also hold for generalized \'etale cohomology with respect to the spectrum $\epsilon ^*(\rmM)$.
        \end{enumerate}
 	\end{corollary}
\begin{proof}
	 First we will consider (i). In this case we first observe that the homogeneous space $\rmG/\rmT$
	admits a decomposition into the double cosets $\rmT \backslash \rmG/\rmT$ which will identify with
	affine spaces over each of the Bruhat-cells.  One begins
	with the Bruhat decomposition $\rmG = \sqcup_{w \eps \rmW} \rmB w \rmB^{-}$, where $\rmB$ is a Borel subgroup containing the given maximal torus $\rmT$ and $\rmB ^{-}$ is its opposite Borel subgroup. Then $ \rmG/ \rmT = \sqcup _{\rmw \eps \rmW}   \rmB \rmw \rmR _u( \rmB ^{-})$ where
	$\rmR _u(\rmB^{-})$ denotes the unipotent radical of $\rmB^{-}$. Now we invoke Theorem ~\ref{double.coset.1}(ii). 
	\vskip .2cm
         Observe that each of the
	  strata $\rmB\rmw \rmR _u(\rmB^{-})$ is an affine space and has a fixed ($k$-rational) point for the action of $\rmT$, which corresponds to the origin of the affine space $\rmB\rmw \rmR _u(\rmB^{-})$. We will 
	  denote this $k$-rational point by $0_{\rmw}$.
	  Therefore, the corresponding transfer (i.e., in the setting of Theorem ~\ref{double.coset.1}, the transfer denoted $tr^{\rmK*}$)
	  sends 
	  \[{\Sigma^{\infty}_{\T}} \BT_+ \mbox{ to } {\Sigma^{\infty}_{\T}} \BT_+ \simeq {\Sigma^{\infty}_{\T}}\ET{\underset {\rmT} \times}0_{\rmw}\]
	  by the map induced by 
	  sending ${\rm Spec} \, \k$ to the coset ${\tilde \rmw}\rmT$ in $\rmG/\rmT$, where $\tilde \rmw \eps \NT$ is the element corresponding to $\rmw$. This in fact corresponds to the 
	  self-map of ${\Sigma^{\infty}_{\T}} \BT_+$ induced by the automorphism of $\rmT$ defined by conjugation
	  by $\tilde \rmw$. This is because conjugation by $\tilde \rmw$ sends the chosen Borel subgroup $\rmB$ containing $\rmT$ to
	  $\tilde \rmw \rmB \tilde \rmw^{-1}$ and this sends the chosen $\rmT$ in $\rmB$ to its conjugate $\tilde \rmw\rmT \tilde \rmw^{-1}$. (See, for example, the discussion in \cite[(3.5) Theorem]{BM}.)
	  Moreover, one may observe that if $\rmT$ acts on a scheme $\rmY$, with the action denoted by $\mu: \rmT \times \rmY \ra \rmY$,
	  then for each $\tilde \rmw \eps \NT$, one may define
	  a new action on $\rmY$ by $\rmT$, by $(t, {\it y}) \mapsto \mu(C_{\tilde \rmw}(t), {\it y})$, where $C_{\tilde \rmw}$ denotes
	   conjugation by the element ${\tilde \rmw}  \eps \rmN(\rmT)$. Therefore, when $\rmG$ is provided
	   with an action on a smooth scheme $\rmX$ as in the second square in ~\eqref{double.coset.0}, the transfer sends
	   ${\Sigma^{\infty}_{\T}}(\ET\times _{T}\rmX)_+ $ to ${\Sigma^{\infty}_{\T}}(\ET^{\tilde \rmw}\times _{T^{\tilde \rmw}}({\rmB}w\rmR_u(\rmB^{-}) \times \rmX)_+ $, where
	   the superscript ${\tilde \rmw}$ denotes the new action of $\rmT$ on the relevant schemes involving the conjugation by ${\tilde \rmw}$.
	  This proves (i).  
\vskip .2cm 
	The proof of (ii) is similar. First we consider the case where $\rmH$ is a 
	parabolic subgroup, with Levi-factor $\rmL$. Then one knows that there is a set of simple roots ${\rm I}$,
	among the basis of simple roots $\Delta$, so that $ \rmL = L_{I}= Z_G((\cap _{\alpha \eps I}Ker (\alpha))^o)$ and $\rmH = \rmP_{\rmI}$, the corresponding parabolic subgroup. The Weyl group $\rmW_{\rmH}$ is then generated by the simple reflections
	$s_{\alpha}, \alpha \eps \rmI$. In this case, one obtains a decomposition of $\rmG$
	as $ \sqcup_{\rmw \eps \rmW_{\rmG}/\rmW_{\rmH}}\rmB \rmw \rmH$ and therefore, the  coset decomposition
	$\rmG/\rmH = \sqcup_{\rmw \eps \rmW_{\rmG}/\rmW_{\rmH}}  \rmB\rmw $. In case $\rmH$ is actually a Borel subgroup, $\rmW_{\rmH}$ is trivial.
	\vskip .2cm
	Observe that each stratum $\rmB w$ of $\rmG/\rmH$ is also an affine space and has a fixed ($k$-rational) point for the action of $\rmT$, which corresponds to the origin of the affine space $\rmB\rmw$. We will 
	  denote this $k$-rational point by $0_{\rmw}$. Therefore an argument as in the last case shows
	that the transfer (i.e., in the setting of Theorem ~\ref{double.coset.1}, the transfer denoted $tr^{\rmK*}$
	sends  
	\[{\Sigma^{\infty}_{\T}} \BT_+ \mbox{ to }{\Sigma^{\infty}_{\T}} \BT_+ \simeq {\Sigma^{\infty}_{\T}}\ET{\underset {\rmT} \times}0_{\rmw} \simeq {\Sigma^{\infty}_{\T}}\ET{\underset {\rmT} \times}(\rmB w)\]
	by the map induced by 
	  sending ${\rm Spec} \, \k$ to the coset ${\tilde \rmw}\rmH$ in $\rmG/\rmH$, where $\tilde \rmw \eps \rmN(\rmT)$ is the element corresponding to $\rmw$. This in fact corresponds to the 
	  self-map of ${\Sigma^{\infty}_{\T}} \BT_+$ induced by the automorphism of $\rmT$ defined by conjugation
	  by $\tilde \rmw$. Thus the above transfer sends ${\Sigma^{\infty}_{\T}} \BT_+$ to ${\Sigma^{\infty}_{\T}}(\ET^{\tilde w}{\underset {\rmT^{\tilde w}} \times}(\rmB w))_+$ and hence
	  sends ${\Sigma^{\infty}_{\T}}(\ET^{\tilde w}{\underset {\rmT^{\tilde w}} \times}\rmX)_+ $ to ${\Sigma^{\infty}_{\T}}(\ET^{\tilde w}{\underset {\rmT^{\tilde w}} \times}(\rmB w \times \rmX))_+$. We denote this map by $\tilde \rmC_{\rm w}$.
	  As a result the composition $\tr^{\rmK*} \circ \tilde \rmp_{\rmK}^*$ is equal to the sum $\Sigma _{{\rm w} \eps \rmW_{\rmG}/\rmW_{\rmH}} \tilde \rmC_{\rm w}^* \circ i_{\rm w}^*$ 
	  where $i_{\rmw}: \ET^{\tilde w}{\underset {\rmT^{\tilde w}} \times}(\rmB w \times \rmX) \ra \EG\times_{\rmG} (\rmB w \times \rmX)  \ra \EG\times_{\rmG} (\rmG/\rmH \times \rmX) \cong \EG\times_{\rmG} (\rmG \times_{\rmH} \rmX) $
	  is the obvious inclusion map. Denoting $\tilde \rmC_{\rm w}^*$ by $\rmC_{\rm w}$, this 
	proves (ii) in the case $\rmH$ is a parabolic subgroup. 
	\vskip .2cm 
	The case $\rmH = \rmL= \rmL_{\rmI}$ a Levi subgroup reduces to the above case, since
	we may start with the decomposition of $\rmG$ as $\sqcup_{w \eps \rmW_{\rmG}/\rmW_{\rmH}}\rmB w \rmP_ {\rmI}$ and hence
	a  coset decomposition $\rmG/ \rmL = \sqcup_{w \eps \rmW_{\rmG}/\rmW_{\rmH}}\rmB w\rm R_u( \rmP_{\rmI})$, with the strata again acyclic.
	We skip the proof of (iii) which is similar.
	\end{proof}
\begin{remark}
 Observe in (i)that each stratum ${\rmB}{\rm w}\rmR_u(\rmB^{-})$ is an affine space with the origin centered at the fixed point of $\rmT$
 indexed by ${\rm w} \in \rmW$. Therefore, the pull-back $i_{\rm w}^*$ in (i) is an isomorphism: hence we will identify 
 ${\Sigma}_{ \rmw \eps \rmW}C_{\rmw} \circ i_{\rmw}^*$ in (i) with ${\Sigma}_{ \rmw \eps \rmW} C_{\rmw}$.
 \end{remark}

	\vskip .2cm \noindent
\begin{corollary}
\label{double.coset.3}
Let $  \rmh^{*, \bullet}$ denote a generalized cohomology theory defined with respect to
a motivic spectrum $\rmM$. We will further assume that generalized cohomology theory is orientable in the sense that it has a theory of 
Chern classes, the base field $k$ is infinite,  contains a $\sqrt{-1}$,  and that the spectrum $\rmM$ has the rigidity 
property in Definition ~\ref{rigid.prop}.
\vskip .1cm
Assume $\rmX$ is a $\rmG$-scheme or an unpointed simplicial presheaf with $\rmG$-action, for a connected split reductive group $\rmG$, with split maximal torus $\rmT$. Then
\[  \rmh^{*, \bullet}(\EG{\underset {\rmG} \times} \rmX, \rmM) \cong   \rmh^{*, \bullet}(\ET{\underset {\T} \times}\rmX, \rmM)^{\rmW}, \quad \rmW = \NT/\rmT \]
if $|\rmW|$ is a unit in the cohomology theory $  \rmh^{*, \bullet}$.  Corresponding results also hold for generalized \'etale cohomology theories defined with respect to
the spectrum $\epsilon^*(\rmM) \eps \Spt(\k_{et})$ (that is, on the \'etale site). 
\end{corollary}
\begin{proof} Throughout the proof, we will denote the generalized motivic cohomology theory simply as $\rmh^{*, \bullet}(\quad)$.
        Let $\pi: \ET{\underset {\rmT} \times}\rmX \simeq \EG{\underset {\rmG} \times} (\rmG{\underset {\rmT} \times} \rmX) \ra \EG{\underset {\rmG} \times}\rmX$ denote the map induced by the map $\rmG{\underset {\rmT} \times} \rmX\ra \rmX$, sending $(g, x) \mapsto gx$ and let
	the corresponding transfer be denoted $tr$.
	 Then the  first step is to observe that the map
	\[\pi^*:  {\rmh}^{*, \bullet}(\EG{\underset {\rmG} \times  }\rmX  ) \ra  { \rmh}^{*, \bullet}(\ET{\underset {\T} \times}\rmX) \]
	is a split injection since $|\rmW|$ is a unit in the given generalized cohomology theory.  
	This can be seen using the transfer map:
	\[\tr^*:{ \rmh}^{*, \bullet}(\ET{\underset {\T} \times}\rmX) \ra {\rmh}^{*, \bullet}(\EG{\underset {\rmG} \times  }\rmX  )\]
	which is in fact the composition of the two transfer maps:
	\[\tr_f^*:{ \rmh}^{*, \bullet}(\ET{\underset {\T} \times}\rmX) \ra {\rmh}^{*, \bullet}(\E\NT{\underset {\NT} \times  }\rmX  ) \mbox{ and } \tr^*:{ \rmh}^{*, \bullet}(\E\NT{\underset {\NT} \times}\rmX) \ra {\rmh}^{*, \bullet}(\EG{\underset {\rmG} \times  }\rmX  ).\]
	Observe that the first transfer map is associated to the degree $|\rmW|$ finite \'etale map 
	$\ET{\underset {\T} \times}\rmX \ra \E\NT{\underset {\NT} \times  }\rmX$. The transfer for such finite \'etale maps (and more generally for any 
	projective smooth map) has been constructed in \cite[Corollary 3.24]{JP22}, where we show that pull-back by such a transfer corresponds to push-forward
	in any orientable generalized motivic cohomology theory.
	\vskip .2cm
	Then Corollary ~\ref{double.coset.2}(ii)
	shows that the image of the last map identifies with the $\rmW$-invariant part of  
	$  \rmh^{*, \bullet}(\ET{\underset {\T} \times}\rmX)$. This is a standard argument,
	but for the sake of completeness, we will provide further details.
	\vskip .2cm
	 Since $\chi_{\rm mot}(\rmG/\rmN(\rmT)) =\tau_{\rmG/\rmN(\rmT)}^*(1)=1$ and
	$\chi_{\rm mot}(\rmN(\rmT)/\rmT)= \tau_{\rmN(\rmT)/\rmT}^*(1)=|\rmW|$, one sees that $\chi_{\rm mot}(\rmG/\rmT) = \tau_{\rmG/\rmT}^*(1)=|\rmW|$. Therefore,
	we obtain: 
	\be \begin{equation}
	\label{image.p*.1}
	 tr^*\circ \pi^*(\alpha) = |\rmW|\alpha,  \alpha \eps  \rmh^{*, \bullet}(\EG{\underset {\rmG} \times}\rmX).
	 \end{equation} \ee
	Therefore, since $|\rmW|$ is assumed to be a unit, the map $\pi^*$ is injective.
	Next let $\alpha \eps  \rmh^{*, \bullet}(\ET{\underset {\rmT} \times}\rmX)^{\rmW}$. 
	Then, by Corollary ~\ref{double.coset.2}(ii), we obtain:
	\be \begin{equation}
	\label{image.p*.2}
	 \pi^* \circ  tr^*(\alpha) = {\Sigma}_{\rmw \eps \rmW}C_{\rmw}(\alpha)= |\rmW|\alpha.
	\end{equation} \ee
	Then ~\eqref{image.p*.1} and ~\eqref{image.p*.2} along with the assumption that $|\rmW|$ is a unit show that 
	$ \rmh^{*, \bullet}(\ET{\underset {\rmT} \times}\rmX)^{\rmW} \subseteq Image(\pi^*)$.
	Finally, one may observe that $C_{\rmw} \circ \pi^* = \pi^*$, for all $\rmw \eps \rmW$. 
	To see this, it is enough to observe that $\pi$ corresponds to the map
	\[\EG\times_{\rmG}(\rmG\times_{\rmT}\rmX) \ra \EG\times_{\rmG}\rmX\]
	discussed in the first paragraph of the proof, and that the action of the Weyl group on the source is induced by the action of $\rmW$ on $\rmG/\rmT$, which itself
	is induced by the corresponding action of $\rmW$ on $\rmG/\rmB$. Since $\rmG/\rmT$ is sent to $Spec \, \k$ under
	$\pi$, it follows that $C_{\rmw} \circ \pi^* = \pi^*$, for all $\rmw \eps \rmW$.
	This shows that the image
	of $\pi^*$ is contained in $ \rmh^{*, \bullet}(\ET{\underset {\rmT} \times}\rmX)^{\rmW}$.
	\end{proof}
	\vskip .2cm \noindent
\begin{corollary}
 \label{double.coset.4}\index{Brauer groups}
Assume the base field is infinite and contains a $\sqrt -1$.  Let $\rmG$ denote a connected reductive group that is {\it special} and $\rmT$ denote a split maximal torus of $\rmG$. 
 Let $\rmH^{*, \bullet}_{\rm mot}(\quad, \bZ/\ell^n)$ ($\rmH^*_{et}(\quad, \mu_{\ell^n}(\bullet))$) denote motivic cohomology with $\bZ/\ell^n$-coefficients (\'etale cohomology with
  respect to the sheaf $\mu_{\ell^n}(\bullet)$), where
 $\ell$ is a prime different from $char (k)$ and $n$ is a fixed positive integer. Then the following hold:
 \begin{enumerate}[\rm(i)]
 	\item Assume further that  $|\rmW|$ is prime to $\ell$. If $\rmX$ is any smooth $\rmG$-scheme or an unpointed simplicial presheaf with $\rmG$-action, then $\rmH^{*, \bullet}_{\rm mot}(\EG{\underset {\rmG} \times}\rmX, \bZ/\ell^n) \cong 
 	    \rmH^{*, \bullet}_{\rm mot}(\ET{\underset {\rmT} \times}\rmX, \bZ/\ell^n)^{\rmW}$, where $\rmT$ denotes a maximal torus in $\rmG$.
 	\item Let $\rmH$ denote a closed linear algebraic subgroup of maximal rank in $\rmG$ so that it is also {\it special}, and $\rmT$ is a maximal torus in $\rmH$ as well. 
 	Let $\rmW_{\rmH}$ denote the Weyl group ${\rm {N_{H}(T)}}/\rmT$, where ${\rm {N_H(T)}}$ denotes the normalizer of $\rmT$ in $\rmH$. Assume that $|\rmW_{\rmH}|$ is prime to $\ell$.
 	Then 
 	$ \rmH^{*, \bullet}_{\rm mot}(\rmG /\rmH, \bZ/ \ell^n) \cong  \rmH^{*, \bullet}_{\rm mot}(\rmG/\rmT, \bZ/\ell^n)^{\rmW_{\rmH}} $. 
       \item The statements corresponding to those in (i) and (ii) also hold 
 	 for \'etale cohomology with $\bZ/\ell^n$-coefficients.
       \item Therefore, under the assumptions of (i) ((ii)), 
         \[{\rm Br}(\EG{\underset {\rmG} \times}\rmX)_{\ell^n} \cong {\rm Br}(\ET{\underset {\rmT} \times}\rmX)^{\rm W}_{\ell^n}\]
 	\[ (\,{\rm Br}(\rmG/\rmH)_{\ell ^n} \cong {\rm Br}({\rm Spec} \, \k)_{\ell^n}, \res.)\]
 	Here the subscript $\ell^n$ denotes the $\ell^n$-torsion subgroup.
 	\item Assume in addition to the hypotheses in (i) that the base field $k$ is separably closed and that the 
 	 cycle map  induces an isomorphism
	 $\H^{i, j}_{\rm mot}(\rmX, \bZ/\ell^n) \ra \H^i_{et}(\rmX, \mu_{\ell^n}(j))$ for all $0\le i \le 2$ and all $0\le j \le 1$. 
	 Then the induced map $\H^{i,j}_{\rmG, mot}(\rmX, \bZ/\ell^n) = \H^{i,j}_{\rm mot} (\EG{\underset {\rmG} \times}\rmX, \bZ/\ell^n) \ra \H^i_{\rmG, et}(\rmX, \mu_{\ell^n}(j)) = \H^i_{et}(\EG{\underset {\rmG} \times}\rmX, \mu_{\ell^n}(j))$
	 is also an isomorphism for all $0 \le i \le 2$ and $0 \le j \ge 1$  provided $|\rmW|$ is relatively prime to $\ell$.

      \end{enumerate}
\end{corollary}

\begin{proof} \index{Brauer groups}
	
	The reason for assuming that the group $\rmG$ is special, is so that one can do
	the Borel construction $\EG{\underset {\rmG} \times} \rmX$ on the Zariski site itself. A similar reason holds for
	assuming that $\rmH$ is also special.
	\vskip .1cm
	The statement in (i) is clear from Corollary ~\ref{double.coset.3}. Clearly one obtains
	a corresponding statement in \'etale cohomology as well.  
	
	\vskip .2cm
   Next we will prove (ii) by observing the sequence of isomorphisms:
	\[  \rmH^{*, \bullet}_{\rm mot}(\rmG/\rmH, \bZ/\ell^n) \cong \rmH^{*, \bullet}_{\rm mot}(\EH{\underset {\rmH} \times }\rmG, \bZ/\ell^n) \cong H^{*, \bullet}_{\rm mot}(\ET{\underset {\rmT} \times }\rmG, \bZ/\ell^n)^{\rmW_{\rmH}} \cong H^{*, \bullet}_{\rm mot}(G/T, \bZ/\ell^n)^{\rmW_{\rmH}}. \]
	The first and last isomorphisms follow from the fact that $\EH$ and $\ET$ are both ${\mathbb A}^1$-acyclic. The statement in (i) 
	provides the second isomorphism, by replacing $\rmG$ ($\rmX$, $\rmW$) in (i) by $\rmH$ ($\rmG$, $\rmW_{\rmH}$, \res). This proves (ii) and clearly the same proof carries over to \'etale cohomology, thereby proving (iii).
	\vskip .2cm 
	One may observe the identification $\ET\times_{\rmT}\rmX \cong \EG\times_{\rmG}(\rmG\times_{\rmT}\rmX)$ with the 
	action of $\rmW$ on $\EG\times_{\rmG}(\rmG\times_{\rmT}\rmX)$ induced by its action on $\rmG/\rmT \simeq \rmG/\rmB$, where $\rmB$ is a Borel subgroup containing $\rmT$.
	On viewing
	 the cycle map as induced by the map ${\mathbb Z}/\ell^n(i) \ra \rmR\epsilon_* \epsilon^*{\mathbb Z}/\ell^n({\it i})$, it becomes
	  clear that it is compatible with the action of $\rmW$ on $\rmG/\rmT$ and on $\EG\times_{\rmG}(\rmG\times_{\rmT}\rmX)$.
	This completes the proof of the first statement in (iv).
	\vskip .2cm 
	We argue in a similar manner to prove the second statement in (iv). Observe that the action of $\rmW_{\rmH}$
	on $\ET\times_{\rmT} \rmG \cong \EH\times_{\rmH}(\rmH\times_{\rmT}\rmG)$ is induced by the action of $\rmW_{\rmH}$
	on $\rmH/\rmT$. 
	\vskip .1cm 
	Finally  we make use of the short-exact sequence
	\be \begin{equation}
	 \label{Kummer.seq.2}
	0 \ra \rmH^{2, 1}_{\rm mot}({\rmG}/\rmT, {\mathbb Z}/\ell^n) \ra \rmH^2_{et}({\rmG}/\rmT, \mu_{\ell^n}(1)) \ra \Br({\rmG}/\rmT)_{\ell^n} \ra 0 \,  
	\end{equation} \ee
	\vskip .1cm \noindent
	 Since the map $\rmH^{2, 1}_{\rm mot}({\rmG}/\rmT, {\mathbb Z}/\ell^n )\ra \rmH^2_{et}({\rmG}/\rmT, \mu_{\ell^n}(1))$ is the cycle map, which has been observed to be injective with its
	cokernel isomorphic to ${\rm Br}({\rm Spec} \, \k)_{\ell^n}$ as shown in Lemma ~\ref{cycle.lemma}, it follows that
	$\Br(\rmG/\rmT)_{\ell^n} \cong {\rm Br}(Spec \, \k)_{\ell^n}$. One may see that $\Br({\rm Spec} \, k)_{\ell^n}$ injects
	diagonally into the sum of copies of $\Br({\rm Spec} \, k)_{\ell^n}$ at the origins of the Schubert cells in 
	$\rmG/\rmB \simeq \rmG/\rmT$. Therefore, $\rmW_{\rmH}$ acts trivially on it.
	This completes the proof of the second statement in (iv).
	\vskip .2cm
	Next we will consider the last statement.
	The statement in (i) along with its counterpart in \'etale cohomology, provides the isomorphisms:
 \[\H^{*, \bullet}_{\rmG, mot}(\rmX, \bZ/\ell^n) \cong \H^{*, \bullet}_{\rmT, mot}(\rmX, \bZ/\ell^n)^{\rmW} \mbox{and}\]
\[\H^{*, \bullet}_{\rmG, et}(\rmX, \bZ/\ell^n) \cong \H^{*, \bullet}_{\rmT, et}(\rmX, \bZ/\ell^n)^{\rmW}.\]
       Therefore, we reduce to the case where $\rmG$ is replaced by a split torus $\rmT$. At this point, we observe that a choice of 
       $\BT^{\it gm,m} = \Pi_{i=1}^n{\mathbb P}^m $, if $\rmT = {\mathbb G}_{\it m} ^n$. 
\vskip .2cm
       Observe that $\ET^{\it gm,m} \ra \BT^{\it gm,m}$ is a Zariski locally trivial torsor for the action $\rmT$, as $\rmT = {\mathbb G}_{\it m}^n$ is a
       split torus, and hence is special as a linear algebraic group in the sense of Grothendieck: see \cite{Ch}. 
       Taking $n=1$, we see that $\pi^m: \rmE{\mathbb G}_{\it m}^{\it gm,m} \ra \rmB{\mathbb G}_{\it m}^{\it gm,m} ={\mathbb P}^m $ is such a torsor, so that there is a Zariski open cover 
       $\{\rmU_j|j=1, \cdots, \rmN\}$
       where $\pi^m_{|\rmU_j}$ is of the form $\rmU_j \times {\mathbb G}_{\it m} \ra \rmU_j$, $j=1, \cdots, \rmN$.
 \vskip .2cm
      Let $\{\rmV_0, \cdots \rmV_m\}$ denote the open cover of ${\mathbb P}^m$ obtained by letting $\rmV_i$ be the open subscheme where 
      the homogeneous coordinate $x_i$, $i=0,\cdots, m$ on ${\mathbb P}^m$ is non-zero.
      Without loss of 
      generality, we may assume the $\rmU_j$ refine the open cover $\{\rmV_i|i=0, \cdots m\}$. Finally the observation that the Picard groups of
      affine spaces are trivial, shows that one may in fact take $\rmN=m$ and $\rmU_j= \rmV_j$, $j=0, \cdots, m$.
      Now one may take an open cover of $\Pi_{i=1}^n{\mathbb P}^m$ by taking the product of the affine spaces that form the open cover of
      each factor ${\mathbb P}^m$. We will denote this open cover of $\Pi_{i=1}^n{\mathbb P}^m$ by $\{\rmW_{\alpha}|\alpha\}$.
\vskip .2cm
      Let $p: \ET^{\it gm,m} {\underset {\rmT} \times}\rmX \ra \BT^{\it gm,m}$ denote the obvious map, and let $\epsilon$ denote the
      map from the \'etale site to the Nisnevich site. Let ${\mathbb Z}/\ell^n(i)$ denote the motivic complex of weight $i$
      on the Nisnevich site of $\ET^{\it gm,m}{\underset {\rmT} \times}X$. Then one obtains the identification (see \cite{Voev11}, \cite{HW}):

\begin{equation}
   {\mathbb Z}/\ell^n(i) = \tau_{\le i} \rmR\epsilon_* \epsilon ^*({\mathbb Z}/\ell^n({\it i})).  
\end{equation} 
\vskip .1cm
     Therefore, on applying $\rmR{\it p}_*$, we obtain the natural maps:
\begin{align}
\label{cycl.map}
     \rmR{\it p}_*({\mathbb Z}/\ell^n(i)) {\overset {\simeq} \ra} \rmR{\it p}_*(\tau_{\le i} \rmR\epsilon_* \epsilon ^*({\mathbb Z}/\ell^n({\it i}))) &\ra \rmR{\it p}_* \rmR\epsilon_* \epsilon^*({\mathbb Z}/\ell^n({\it i}))\\
      &\cong \rmR{\it p}_*\rmR\epsilon_*\mu_{\ell^n}({\it i}) \cong \rmR\epsilon_* \rmR{\it p}_* \mu_{\ell^n}({\it i}).\notag
\end{align} 
      On taking the sections over each Zariski open set $\rmW_{\alpha}$ in the above cover of $\BT^{\it gm,m}$, we obtain a quasi-isomorphism, since affine-spaces
   are contractible for motivic cohomology, and also for \'etale cohomology with respect to $\mu_{\ell^n}(j)$ as the base field is separably closed with $\ell \ne char (k)$. Now
   a Mayer-Vietoris argument using the above open cover of $\BT^{\it gm,m}$, discussed in the Lemmas ~\ref{isom.lemma}, 
   ~\ref{1.intersect} and Proposition ~\ref{inductive.pf} completes the proof of the last statement.
  This completes the proof of Corollary ~\ref{double.coset.4}. \qed
	\end{proof}

\begin{remarks} (i) Observe that taking $\rmX = Spec \, \k$ in (i) and (ii), shows that the higher cycle map
\[\rmH^{*, \bullet}(\BG, {\mathbb Z}/\ell^n) \ra \rmH^{*, \bullet}_{et}(\BG, \mu_{\ell^n})\]
is an isomorphism for any linear algebraic group $\rmG$ that is special, if $\ell$ is prime to $|\rmW|$. In fact the same conclusion also
holds for all linear algebraic groups $\rmG$, since $\BG = \colimm \epsilon_*(\BG^{\it gm,m})$ by \cite[p. 135, Proposition 2.6]{MV}.
(One may consult \cite{DIJ23} for more details.)
\vskip .1cm
(ii) The forthcoming paper \cite{DIJ23}, extends the results of the above corollary in several directions.
First the statement (v) above is sharpened to prove that if the cycle map $\H^{2, 1}_{\rm mot}(\rmX, \bZ/\ell^n) \ra \H^2_{et}(\rmX, \mu_{\ell^n}(1))$
is an isomorphism (equivalently $\Br(\rmX)_{\ell^n}=0$, where $\Br(\rmX)_{\ell^n}$ denotes the $\ell^n$-torsion part of the Brauer group $\Br(\rmX)$), then the cycle map 
\[ \H^{2,1}_{\rm mot} (\EG{\underset {\rmG} \times}\rmX, \bZ/\ell^n) \ra \H^2_{et}(\EG{\underset {\rmG} \times}\rmX, \mu_{\ell^n}(1)) \]
is also an isomorphism. Equivalently, it is shown there under the above assumptions,  that a certain equivariant Brauer group $\Br_{\rmG}(\rmX)_{\ell^n}=0$.
 This is then shown to imply, under certain mild assumptions,  the triviality of the $\ell^n$-torsion
part of the equivariant Brauer group of the semi-stable locus for the $\rmG$-action on $\rmX$, and hence the triviality of the $\ell^n$-torsion part of the Brauer group of the corresponding GIT-quotient of $\rmX$.
\end{remarks}
We will conclude the above discussion with the following technical results that will be helpful in computing the Brauer groups.
We begin with the spectral sequences:
 \begin{align}
 \label{key.observ.1}
\rmE_2^{s,t} = \rmH^s_{Nis}(U, \rmR^tp_*(\tau_{\le j} \rmR\epsilon_* \epsilon ^*({\mathbb Z}/\ell^n(j)))) & \Rightarrow \rmH^{s+t}_{Nis}(p^{-1}(U), {\mathbb Z}/\ell^n(j))\\
\rmE_2^{s,t} = \rmH^s_{Nis}(U, \rmR^tp_*( \rmR\epsilon_* \epsilon ^*({\mathbb Z}/\ell^n(j)))) & \Rightarrow \rmH^{s+t}_{Nis}(p^{-1}(U),  \rmR\epsilon_* \epsilon ^*({\mathbb Z}/\ell^n(j)))\cong \rmH^{s+t}_{et}(p^{-1}(U), \mu_{\ell^n}(j)) \notag
\end{align} 
The obvious map from the first spectral sequence to the second induces an isomorphism on the $\rmE_2$-terms for $0 \le s+t \le j$, as $s \ge 0$. 
\begin{lemma}
\label{cycle.lemma}
Let $\rmG$ denote a linear algebraic group with $\rmB$  Borel subgroup, both defined over a field $k$. Then the following hold:
\begin{enumerate}[\rm(i)]
\item The cycle map $cycl: {\rmH_{\rmM}^{*, \bullet}(\rmG/\rmB, {\mathbb Z}/\ell^n)} \ra {\rmH_{et}^{*}(\rmG/\rmB, \mu_{\ell^n}(\bullet))}$ is an isomorphism
 when the base field $k$ is separably closed. 
\item
For any base field $k$, ${\rmH_{et}^{\it u}(\rmG/\rmB, \mu_{\ell^n}({\it v}))} \cong \oplus_{i+2m=u, j+m=v} {\rmH_{et}^{\it i}(\Speck, \mu_{\ell^n}({\it j}))} \otimes_{{\mathbb Z}/\ell^n} {\rmH_{\rmM}^{2{\it m}}(\rmG/\rmB, {\mathbb Z}/{\ell^n}({\it m}))}$.\\
 In particular,
 \[{\rmH_{et}^{2}(\rmG/\rmB, \mu_{\ell^n}(1))} \cong {\rmH_{et}^{2}(\Speck, \mu_{\ell^n}(1))} \otimes_{{\mathbb Z}/\ell^n} {\H_{M}^{0}(\rmG/\rmB, {\mathbb Z}/{\ell^n}(0))} \oplus {\rmH_{et}^{0}(\Speck, \mu_{\ell^n}(0))} \otimes_{{\mathbb Z}/\ell^n} {\rmH_{\rmM}^{2}(\rmG/\rmB, {\mathbb Z}/{\ell^n}(1))}\]
\[\cong {\rmH_{et}^{2}(\Speck, \mu_{\ell^n}(1))} \oplus {\rmH_{M}^{2}(\rmG/\rmB, {\mathbb Z}/{\ell^n}(1))}.\]
\item
 For any base field $k$, the cycle map
\[cycl: \rmH^{2,1}_{\rmM}(\rmG/\rmB, {\mathbb Z}/\ell^n) \ra \rmH^2_{et}(\rmG/\rmB, \mu_{\ell^n}(1))\]
is injective with cokernel isomorphic to $\rmH^2_{et}(\Speck, \mu_{\ell^n}(1))$.
\end{enumerate}
\end{lemma}
\begin{proof} Observe that $\rmG/\rmB$ is a flag variety, which is a projective smooth scheme stratified by affine cells.
Alternatively one may make use of a Bialynicki-Birula decomposition with the fixed points of the action of the maximal torus 
corresponding to the elements of the Weyl group $\rmW$. Therefore, (i) and (ii) follow readily and (iii) is an immediate consequence.
 \end{proof}
\begin{lemma}
 \label{isom.lemma} (See \cite[Lemma 9.5]{DIJ23}.)
 Let $p: \X \ra \Y$ denote a map of smooth  schemes over $k$, so that it is Zariski locally trivial, with fibers given by the
  scheme $\rmX$ satisfying the condition that the cycle map:
  \[cycl: \rmH^{2,1}_M(X, {\mathbb Z}/\ell^n) \ra \rmH^{2}_{et}(X, \mu_{\ell^n}(1))\]
  is an isomorphism. Let $\rmU$, $\rmV$ denote two Zariski open subschemes of $\cY$ so that $\X\times_{\Y}\rmU \cong \rmU \times \rmX$
  and $\X\times_{\Y}\rmV \cong \rmV \times \rmX$. Assume that the corresponding cycle maps
  \[\rmH^{2,1}_M(\X\times_{\Y}\rmU, {\mathbb Z}/\ell^n) \ra \rmH^{2}_{et}(\X\times_{\Y}\rmU, \mu_{\ell^n}(1)) \mbox{ and }\rmH^{2,1}_M(\X\times_{\Y}\rmV, {\mathbb Z}/\ell^n) \ra \rmH^{2}_{et}(\X\times_{\Y}\rmV, \mu_{\ell^n}(1))\]
  are both isomorphisms and the cycle map
  \[\rmH^{2,1}_M(\X\times_{\Y}(\rmU \cap \rmV), {\mathbb Z}/\ell^n) \ra \rmH^{2}_{et}(\X\times_{\Y}(\rmU\cap \rmV), \mu_{\ell^n}(1))\]
  is a monomorphism. Then the cycle map
  \[\rmH^{2,1}_M(\X\times_{\Y}(\rmU \cup \rmV), {\mathbb Z}/\ell^n) \ra \rmH^{2}_{et}(\X\times_{\Y}(\rmU\cup \rmV), \mu_{\ell^n}(1))\]
is an isomorphism.
\end{lemma}
\begin{proof}
For a subscheme $\rmW$ in $\rmY$, we will continue to let $\X_{W} = \X\times_{\Y}\rmW$.
Now we consider the commutative diagram with exact rows:
\[\xymatrix{{\rmH^{1,1}_{M}(\X_{\rmU}, {\mathbb Z}/\ell^n) \oplus \rmH^{1,1}_{M}(\X_{\rmV}, {\mathbb Z}/\ell^n)} \ar@<1ex>[r] \ar@<1ex>[d] & {\rmH^{1, 1}_M(\cX_{\rmU \cap \rmV},{\mathbb Z}/\ell^n)} \ar@<1ex>[r]  \ar@<1ex>[d]& {\rmH^{2, 1}_M(\cX_{\rmU \cup \rmV}, {\mathbb Z}/\ell^n)} \ar@<1ex>[d]\\
  {\rmH^{1}_{et}(\X_{\rmU}, \mu_{\ell^n}(1)) \oplus \rmH^{1}_{et}(\X_{\rmV}, \mu_{\ell^n}(1))} \ar@<1ex>[r] & {\rmH^{1}_{et}(\X_{\rmU \cap \rmV}, \mu_{\ell^n}(1))} \ar@<1ex>[r] & {\rmH^{2}_{et}(\cX_{\rmU \cup \rmV}, \mu_{\ell^n}(1)) }}
 \]
\[\xymatrix{{} \ar@<1ex>[r]  &  {\rmH^{2,1}_{M}(\X_{\rmU}, {\mathbb Z}/\ell^n) \oplus \rmH^{2,1}_{M}(\X_{\rmV}, {\mathbb Z}/\ell^n)} \ar@<1ex>[d] \ar@<1ex>[r] &  {\rmH^{2, 1}_M(\X_{\rmU \cap \rmV}, {\mathbb Z}/\ell^n)} \ar@<1ex>[d]\\
  {}  \ar@<1ex>[r] & {\rmH^{2}_{et}(\X_{\rmU}, \mu_{\ell^n}(1)) \oplus \rmH^{2,1}_{et}(\X_{\rmV}, \mu_{\ell^n}(1))}  \ar@<1ex>[r] & {\rmH^{2}_{et}(\X_{\rmU \cap \rmV}, \mu_{\ell^n}(1))} }
 \]
 In view of the spectral sequence in ~\eqref{key.observ.1} with $j=1$, one may observe that the second vertical map is an isomorphism.
 Therefore, a diagram chase applies to prove the required map is an isomorphism
\end{proof}

\begin{proposition} (See \cite[Lemma 9.5]{DIJ23}.)
 \label{inductive.pf}
 Let $p:\cX \ra \cY$ denote a map of smooth schemes over $k$, satisfying the hypotheses of Lemma ~\ref{isom.lemma}. We will further assume  the following:
 let $\rmU_i, i=1, \cdots, n$ denote open subsets of $\cY$, so that the hypotheses of Lemma ~\ref{isom.lemma} holds
 with $\rmU$, $\rmV$ denoting any two of these open sets. Assume further that there exists an affine space ${\mathbb A}^N$
 so that each $\rmU_i \cong {\mathbb A}^N$ and that each intersection $\rmU_{i_1} \cap \rmU_{i_2} \cong {\mathbb G}_m \times {\mathbb A}^{N-1}$.
 Then the following holds, where for a subscheme $\rmW$ in $\rmY$, we will let $\cX_{W} = \cX\times_{\cY}\rmW$, and  cycl will denote the higher cycle map:
 \begin{enumerate}[\rm(i)]
  \item  $cycl: \rmH^{2,1}_M(\cX_{(\rmU_1 \cup \cdots \cup \rmU_{n-1}) \cap \rmU_n}, {\mathbb Z}/\ell^n) \ra \rmH^2_{et}(\cX_{(\rmU_1 \cup \cdots \cup\rmU_{n-1}) \cap \rmU_n}, \mu_{\ell^n}(1))$
  is a monomorphism \mbox { and }
  \item $cycl: \rmH^{2,1}_M(\cX_{(\rmU_1 \cup \cdots \cup \rmU_{n-1}) \cup \rmU_n}, {\mathbb Z}/\ell^n) \ra \rmH^2_{et}(\cX_{(\rmU_1 \cup \cdots \cup \rmU_{n-1}) \cup \rmU_n}, \mu_{\ell^n}(1))$ is an isomorphism.
 \end{enumerate}
\end{proposition}
\begin{proof}
 We will prove these using ascending induction on $n$. We will first consider (i). Observe that the case $n=2$ is handled by 
 Lemma ~\ref{1.intersect}. 
\vskip .1cm
 Assume next that (i) holds when $\rmU_i$, $i=1, \cdots, n$ are any open subsets of $\cY$ satisfying the hypotheses.
 Let $\rmU_i$, $i=1, \cdots, n, n+1$ be open subsets satisfying the hypotheses. Let $\rmW_1 = (\rmU_1 \cup \cdots \cup  \rmU_{n-1}) \cap \rmU_{n+1}$
 and let $\rmW_2 = \rmU_n \cap \rmU_{n+1}$. Then we obtain the commutative diagram:
 \[\xymatrix{{\rmH^{1,1}_{M}(\cX_{\rmW_1}, {\mathbb Z}/\ell^n) \oplus \rmH^{1,1}_{M}(\cX_{\rmW_2}, {\mathbb Z}/\ell^n)} \ar@<1ex>[r] \ar@<1ex>[d] & {\rmH^{1, 1}_M(\cX_{\rmW_1 \cap \rmW_2},{\mathbb Z}/\ell^n)} \ar@<1ex>[r]  \ar@<1ex>[d]& {\rmH^{2, 1}_M(\cX_{\rmW_1 \cup \rmW_2}, {\mathbb Z}/\ell^n)} \ar@<1ex>[d]\\
  {\rmH^{1}_{et}(\cX_{\rmW_1}, \mu_{\ell^n}(1)) \oplus \rmH^{1}_{et}(\cX_{\rmW_2}, \mu_{\ell^n}(1))} \ar@<1ex>[r] & {\rmH^{1}_{et}(\cX_{\rmW_1 \cap \rmW_2}, \mu_{\ell^n}(1))} \ar@<1ex>[r] & {\rmH^{2}_{et}(\cX_{\rmW_1 \cup \rmW_2}, \mu_{\ell^n}(1)) }}
 \]
\[\xymatrix{{} \ar@<1ex>[r]  &  {\rmH^{2,1}_{M}(\cX_{\rmW_1}, {\mathbb Z}/\ell^n) \oplus \rmH^{2,1}_{M}(\cX_{\rmW_2}, {\mathbb Z}/\ell^n)} \ar@<1ex>[d] \ar@<1ex>[r] &  {\rmH^{2, 1}_M(\cX_{\rmW_1 \cap \rmW_2}, {\mathbb Z}/\ell^n)} \ar@<1ex>[d]\\
  {}  \ar@<1ex>[r] & {\rmH^{2}_{et}(\cX_{\rmW_1}, \mu_{\ell^n}(1)) \oplus \rmH^{2,1}_{et}(\cX_{\rmW_2}, \mu_{\ell^n}(1))}  \ar@<1ex>[r] & {\rmH^{2}_{et}(\cX_{\rmW_1 \cap \rmW_2}, \mu_{\ell^n}(1))} }
 \]
 Then the inductive assumption, together with Lemma ~\ref{1.intersect} show the map $\rm\rmH^{2,1}_{M}(\cX_{\rmW_1}, {\mathbb Z}/\ell^n) \ra \rmH^{2}_{et}(\cX_{\rmW_1}, \mu_{\ell^n}(1)) $
 is a monomorphism while Lemma ~\ref{1.intersect} shows the map $\rm\rmH^{2,1}_{M}(\cX_{\rmW_2}, {\mathbb Z}/\ell^n) \ra \rmH^{2}_{et}(\cX_{\rmW_2}, \mu_{\ell^n}(1)) $ is a monomorphism.
 Observe that $\rmW_1 \cup \rmW_2 = (\rmU_1 \cup \cdots \cup \rmU_n) \cap \rmU_{n+1}$.
 In view of the spectral sequence in ~\eqref{key.observ.1} with $j=1$, one may observe that the first two vertical maps are isomorphisms.
 Therefore, now a straight forward diagram chase then shows the cycle map $\rm\rmH^{2, 1}_M(\cX_{\rmW_1 \cup \rmW_2}, {\mathbb Z}/\ell^n) \ra \rmH^{2}_{et}(\cX_{\rmW_1 \cup \rmW_2}, \mu_{\ell^n}(1))$ is 
 a monomorphism, thereby completing the proof of (i).
 \vskip .1cm
 At this point (ii) follows readily from Lemma ~\ref{isom.lemma} by taking $\rmU = \rmU_1 \cup \cdots \cup \rmU_n$
 and $\rmV= \rmU_{n+1}$ there. Now observe that $\rmU \cap \rmV = (\rmU_1 \cup \cdots \cup \rmU_n) \cap \rmU_{n+1}$. (i) proved
 above shows that the cycle map
 \[\rmH^{2,1}_M(\cX\times_{\cY}(\rmU \cap \rmV), {\mathbb Z}/\ell^n) \ra \rmH^{2}_{et}(\cX\times_{\cY}(\rmU\cap \rmV), \mu_{\ell^n}(1))\]
  is a monomorphism. The inductive assumption now shows that 
  the cycle map
 \[\rmH^{2,1}_M(\cX\times_{\cY}\rmU, {\mathbb Z}/\ell^n) \ra \rmH^{2}_{et}(\cX\times_{\cY}\rmU, \mu_{\ell^n}(1))\]
  is an isomorphism. Therefore, the hypotheses of Lemma ~\ref{isom.lemma} are satisfied, so that Lemma ~\ref{isom.lemma} applies to complete the proof of (ii).
\end{proof}
\vskip .2cm

\begin{lemma} 
\label{1.intersect}
Assume that $\rmX$ is a smooth scheme so that the cycle map
\[cycl: \rmH^{i, 1}_M(X, {\mathbb Z}/\ell^n) \ra \rmH^i_{et}(X, \mu_{\ell^n}(1))\]
is an isomorphism for all $0 \le i \le 2$. Then 
the induced cycle map 
 $\rm\rmH^{i, 1}_M(\rmX \times {\mathbb G}_m, {\mathbb Z}/\ell^n) \ra \rmH^i_{et}(\rmX \times {\mathbb G}_m, \mu_{\ell^n}(1))$
is injective for  for all $0 \le i \le 2$.
\end{lemma}
\begin{proof} 
In view of the observation ~\eqref{key.observ.1} above, the above cycle map is an isomorphism for $i=0$ or $i=1$. Therefore, it suffices
to consider the case $i=2$.
  This follows from the commutative diagram 
 of localization sequences:
 \[\xymatrix{{\rmH^{2,1}_{\rmX \times \{0\} ,M}(\rmX \times {\mathbb A}^1, {\mathbb Z}/\ell^n)} \ar@<1ex>[r] \ar@<1ex>[d] & {\rmH^{2, 1}_M(\rmX \times {\mathbb A}^1,{\mathbb Z}/\ell^n)} \ar@<1ex>[r]  \ar@<1ex>[d]& {\rmH^{2, 1}_M(\rmX \times {\mathbb G}_m, {\mathbb Z}/\ell^n)} \ar@<1ex>[d]\\
  {\rmH^{2}_{\rmX \times \{0\}, et}(\rmX \times {\mathbb A}^1, \mu_{\ell^n}(1))} \ar@<1ex>[r] & {\rmH^{2}_{et}(\rmX \times {\mathbb A}^1, \mu_{\ell^n}(1))} \ar@<1ex>[r] & {\rmH^{2}_{et}(\rmX \times {\mathbb G}_m, \mu_{\ell^n}(1)) }}
 \]
\[\xymatrix{{} \ar@<1ex>[r]  & {\rmH^{3,1}_{\rmX \times \{0\}, M}(\rmX  \times {\mathbb A}^1, {\mathbb Z}/\ell^n)} \ar@<1ex>[d] \ar@<1ex>[r] & {\rmH^{3, 1}_M(\rmX \times {\mathbb A}^1, {\mathbb Z}/\ell^n)} \ar@<1ex>[d]\\
  {}  \ar@<1ex>[r] & {\rmH^{3}_{\rmX \times \{0\}, et}(\rmX  \times {\mathbb A}^1, \mu_{\ell^n}(1))}  \ar@<1ex>[r] & {\rmH^{3}_{et}(\rmX \times {\mathbb A}^1, \mu_{\ell^n}(1))} }
 \]
 The map $\rm\rmH^{3,1}_{\rmX \times \{0\}, M}(\rmX  \times {\mathbb A}^1, {\mathbb Z}/\ell^n) \ra \rmH^{3}_{\rmX \times \{0\}, et}(\rmX  \times {\mathbb A}^1, \mu_{\ell^n}(1))$
 identifies with the map 
 \[\rmH^{1,0}_{M}(\rmX \times  \{0\}, {\mathbb Z}/\ell^n) \ra \rmH^{1}_{et}(\rmX \times \{0\}, \mu_{\ell^n}(0))\]
 and $\rm\rmH^{1,0}_{M}(\rmX \times  \{0\}, {\mathbb Z}/\ell^n) \cong C\rmH^0(\rmX \times \{0\}, {Z}/\ell^n; -1) \cong 0$.
 Therefore this map is clearly injective. The map $\rm\rmH^{2,1}_{\rmX \times \{0\}, M}(\rmX  \times {\mathbb A}^1, {\mathbb Z}/\ell^n) \ra \rmH^{2}_{\rmX \times \{0\}, et}(\rmX  \times {\mathbb A}^1, \mu_{\ell^n}(1))$
 identifies with the map 
 \[\rmH^{0,0}_{M}(\rmX \times  \{0\}, {\mathbb Z}/\ell^n) \ra \rmH^{0}_{et}(\rmX \times \{0\}, \mu_{\ell^n}(0))\]
which is also an isomorphism. Now the required assertion follows from the following lemma.
\end{proof}
\vskip .2cm \noindent
As the next and final example, we consider the stable splittings of $\BGL_n$ as
$\bigvee_{i \le n} \BGL _i/\BGL _{i-1}$ in the motivic (and also \'etale) stable homotopy framework. Such splittings were originally obtained in 
\cite{Sn79} and then rederived in \cite{MP}. (See also \cite{K} for a derivation of this along the lines of \cite{Sn79}.) 
As a result, we will refer to these splittings as the Snaith-Mitchell-Priddy splittings and our arguments use the double coset decomposition as in
\cite{MP}.
\begin{corollary}
	\label{double.coset.5} \index{Snaith-Mitchell-Priddy splittings}
	For each integer $n \ge 1$, there exists a splitting 
	\[ {\Sigma_{\T}^{\infty}} \BGL _{n,+} \simeq \bigvee_{i \le n} {\Sigma_{\T}^{\infty}} \BGL _{i,+}/\BGL_{i-1, +} \] 
	in $\Spt(\k_{\rm mot})$, in case $char (k) =0$. In case $char (k)=p>0$, a  corresponding splitting holds on replacing the suspension spectra above
	with the corresponding suspension spectra with $p$-inverted. Let $\cE$ denote a ring spectrum in $\Spt(\k_{et})$ with all its homotopy groups 
	$\ell$-primary torsion, for some prime $\ell \ne char(k)$. Then, a corresponding splitting also holds in $\Spt(\k_{et}, \cE)$ after all the
	above objects have been smashed with the ring spectrum $\cE$. 
\end{corollary}
\begin{proof} 
	\index{Snaith-Mitchell-Priddy splittings}
	We will explicitly consider only the case where $char(k) =0$.
	We will fix a positive integer $m$ and consider the finite degree approximations of all classifying spaces of degree $m$. However, as $m$ will be
	fixed throughout our discussion, we will omit the superscript $m$ and $gm$ so that $\BG$ ($\EG$) will mean $\BG^{\it gm,m}$ ($\EG^{\it gm,m}$, \res), for suitably large $m$ and 
	for any linear algebraic group $\rmG$.
	The proof begins with the cartesian square:
	\be \begin{equation}
	    \label{SMP.1}
	    \xymatrix{{\rm E} \ar@<1ex>[r] \ar@<1ex>[d] & {\BGL _i \times \BGL _j} \ar@<1ex>[d]^{m_{i,j}}\\
	    	       {\BGL _r \times \BGL _s} \ar@<1ex>[r]^{m _{r,s}} & {\BGL _{i+j=r+s}} .}
	\end{equation}
	In this situation, we let the transfer $tr_{i,j}: {\Sigma^{\infty}_{\T}}(\BGL_{i+j, +}) \ra {\Sigma^{\infty}_{\T}}(\BGL_{i, +} \wedge  BGL_{j, +})$. 
	This fits in the framework of Theorem ~\ref{double.coset.1}(i), by taking $\rmG= \rmGL _{i+j=r+s}$, $\rmH= {\rm P}_{i,j}$, and 
	$\rmK = {\rm P}_{r,s}$. 
	Here ${\rm P}_{i,j}$ is the parabolic subgroup of $\rmGL_{i+j}$ that is the stabilizer of an $i$-dimensional subspace in ${\mathbb A}^{i+j}$
	 and ${\rm P}_{r,s}$ is the parabolic subgroup of $\rmGL_{i+j=r+s}$ that is the stabilizer of an $r$-dimensional subspace in ${\mathbb A}^{i+j=r+s}$.
	 Clearly the Levi-subgroup of ${\rm P}_{i,j} = \rmGL_i \times \rmGL_j$ and the Levi-subgroup of ${\rm P}_{r,s}= \rmGL_r \times \rmGL_s$.
	 Since the fibers of the obvious projections ${\rm P}_{i,j} \ra \GL_i \times \GL_j$ and ${\rm P}_{r,s} \ra \GL_r \times \GL_s$ are
	  unipotent groups, they are acyclic in motivic homotopy theory.
	Then $\rmE$ admits a decomposition into double cosets for $\rmH$ and $\rmK$ indexed by the 
	$({\Sigma}_r \times {\Sigma}_s) \times ({\Sigma}_i \times {\Sigma}_j)$-double cosets in ${\Sigma}_n$ as \cite[p. 1]{MP} shows. Replacing the parabolic subgroups by their 
	Levi-factors, each of the resulting components 
	is of the form:
	\[
	\EGL_r{\underset {\rmGL _r} \times} \rmGL _r/(\rmGL_a \times \rmGL_{r-a}) \times  \EGL_s{\underset{\rmGL _s} \times} \rmGL _s/(\rmGL_b \times \rmGL_{s-b}),  \]
	 as $a, b$ vary so that $a\le r, b \le s$ and $a+b=i$.
	 Therefore, the formula in Theorem ~\ref{double.coset.1}(i), which holds for all motivic spectra, applies to provide the identification in the stable motivic homotopy category:
	\be \begin{equation}
	    \label{double.coset.4.1}
	    tr_{i,j} \circ m_{r,s} = \bigvee_{a+b=i,\\a\le r, b \le s} n_{w_{a,b}}c_{w_{a,b}} \circ (tr_{a, r-a} \wedge tr_{b, s-b}),
	\end{equation} \ee
\vskip .1cm \noindent
where the following hold. First observe that
\[\EGL_r{\underset {\rmGL _r} \times} \rmGL _r/(\rmGL_a \times \rmGL_{r-a}) \simeq \BGL_a \times \BGL _{r-a} \mbox{ and }\]
\[\EGL_s{\underset{\rmGL _s} \times} \rmGL _s/(\rmGL_b \times \rmGL_{s-b}) \simeq \BGL _b \times \BGL _{s-b}.\]
Now $c_{w_{a,b}}:\BGL_a \times \BGL _{r-a} \times \BGL _b \times \BGL _{s-b} \ra 
\BGL _a \times \BGL _b \times \BGL _{r-a} \times \BGL _{s-b} \ra \BGL _i \times \BGL _j$ is the obvious map switching the two inner factors. $n_{w_{a,b}}$ is
a non-negative integer depending on the multiplicity of the above components.
\vskip .2cm
Next one defines ${\overline {\BGL}}_n = \BGL _n/\BGL _{n-1}$ and $f_{i,j}: {\Sigma^{\infty}_{\T}} \BGL _{i+j, +} {\overset {tr_{i,j}} \ra } \Sigma^{\infty}_{\T }\BGL _{i,+} \wedge \Sigma^{\infty}_{\T} \BGL _ {j,+} 
{\overset {\pi_{i,j}} \ra } {\Sigma^{\infty}_{\T}}{\overline {\BGL  }  }_{j,+}$, where $\pi_{i,j}$ is the obvious projection. Note that $f_{n,0}: {\Sigma^{\infty}_{\T}}\BGL_{n, +} \ra {\Sigma^{\infty}_{\T}}S^0$ is the augmentation and $f_{0, n}: {\Sigma^{\infty}_{\T}} \BGL_{n, +} \ra {\Sigma^{\infty}_{\T}}{\overline \BGL}_{n, +}$ is the projection.
By composing the maps on the two sides of ~\eqref{double.coset.4.1} with $\pi_{i,j}$, we obtain:
\be \begin{equation}
\label{double.coset.4.2}
f_{i,j} \circ m_{r,s} = \bigvee_{a+b=i,\\ a\le r, b \le s} n_{w_{a,b}}{\overline m}_{r-a, s-b}  \circ (f_{a, r-a} \wedge f_{b, s-b})
\end{equation} \ee
\vskip .1cm \noindent
where ${\overline  m}_{r-a,s-b}: {\Sigma^{\infty}_{\T}}{\overline \BGL}_{r-a, +} \wedge {\Sigma^{\infty}_{\T}}{\overline \BGL}_{s-b,+} \ra {\Sigma^{\infty}_{\T}}{\overline \BGL}_{r-a+s-b, +} =
{\Sigma^{\infty}_{\T}}{\overline \BGL}_{j, +}$ denotes the induced map.
\vskip .2cm
Now observe that the maps $f_{n-j, j}: {\Sigma^{\infty}_{\T}}\BGL _{n,+} \ra {\Sigma^{\infty}_{\T}} {\overline \BGL }_{j, +}$ define the map
\be \begin{equation}
   \label{double.coset.4.3}
   \Pi_{0 \le j \le n} f_{n-j, j} : {\Sigma^{\infty}_{\T}}\BGL _{n,+} \ra \Pi_{0 \le j \le n} {\Sigma^{\infty}_{\T}}  {\overline \BGL}_{j, +} \simeq \bigvee_{0 \le j \le n} {\Sigma^{\infty}_{\T}}  {\overline \BGL}_{j, +}
\end{equation} \ee 
\vskip .1cm \noindent
It suffices to show that this map is a weak-equivalence.  For this, we will adopt the argument given in \cite[Proof of Theorem 4.2]{MP}. Let $g_j: \BGL _j \ra \BGL _n$, for $n = i+j$,
denote the map induced by the inclusion of $\rmGL _j $ into the last $j \times j$ block in 
$\rmGL _n$. Now it suffices to show that the composition $\bar g_{j,+}= f_{n-j,j} \circ {\Sigma^{\infty}_{\T}} g_{j,+}$ is the projection ${\Sigma^{\infty}_{\T}} \BGL _{j, +} \ra {\Sigma^{\infty}_{\T}} {\overline \BGL }_{j, +}$, since then
 the map in ~\eqref{double.coset.4.3} would be a filtration preserving map that induces a weak-equivalence on the associated graded objects. 
\vskip .2cm
 Therefore, we proceed to show that, the composition $\bar g_{j,+}= f_{n-j,j} \circ {\Sigma^{\infty}_{\T}} g_{j,+}$ is the projection ${\Sigma^{\infty}_{\T}} \BGL _{j, +} \ra {\Sigma^{\infty}_{\T}} {\overline \BGL }_{j, +}$. We will take $r=i$, $s=j$ in ~\eqref{double.coset.4.2} and then pre-compose 
the map there with the map $\rmS^0 \wedge {\Sigma^{\infty}_{\T}} \BGL _{j, +} \ra {\Sigma^{\infty}_{\T}}\BGL_{i,+} \wedge {\Sigma^{\infty}_{\T}} \BGL_{j,+}$. Then the left-hand-side
yields $\bar g_{j, +}$, while the right-hand-side yields a finite sum of terms of the form:
\be \begin{align}
\label{double.coset.4.4}
 {\Sigma^{\infty}_{\T}}\rmS^0 \wedge {\Sigma^{\infty}_{\T}} \BGL _{j, +}  &\ra  {\Sigma^{\infty}_{\T}} \BGL _{i,+} \wedge {\Sigma^{\infty}_{\T}} \BGL _{j, +}  {\overset {tr_{a, i-a} \wedge tr_{b, j-b}} \ra}\\
 {\Sigma^{\infty}} _{\T}(\BGL_{a, +} \wedge \BGL _{i-a, +} \wedge \BGL _{b,+} \wedge \BGL _{j-b, +} ) &{\overset {\pi_{i-a} \wedge \pi_{j-b}} \ra}  {\Sigma^{\infty}_{\T}} {\overline \BGL }_{i-a, +}  \wedge {\Sigma^{\infty}_{\T}} {\overline \BGL }_{j-b, +} {\overset {{\overline m}_{i-a, j-b}}  \ra } {\Sigma^{\infty}_{\T}} {\overline \BGL} _{j, +}. \notag
 \end{align} \ee
\vskip .2cm \noindent
If $i>a$, then the above map ${\Sigma^{\infty}_{\T}}\rmS^0 \ra {\Sigma^{\infty}_{\T}}{\overline \BGL}_{i-a, +}$ will factor through ${\Sigma^{\infty}_{\T}}{ \BGL}_{i-a-1,+}$, so that $\rmS^0$ maps to the base point in ${\overline \BGL}_{i-a}$, and therefore the above map will be trivial. 
Clearly it is also trivial for $i<a$, so that the only non-trivial summand in ~\eqref{double.coset.4.4} is when $a=i$, and $b=0$. 
Therefore, the only non-trivial summand in ~\eqref{double.coset.4.4} will be a map of the form:
\[ {\Sigma^{\infty}_{\T}}\rmS^0 \wedge {\Sigma^{\infty}_{\T}}\BGL _{j, + } \ra {\Sigma^{\infty}_{\T}}\BGL_{i, +} \wedge {\Sigma^{\infty}_{\T}}\BGL_{j, +} {\overset {\epsilon \wedge \pi} \ra }
{\Sigma^{\infty}_{\T}}\rmS^0 \wedge {\overline \BGL}_{j, +} \cong {\Sigma^{\infty}_{\T}}{\overline \BGL}_{j, +}.\]

This identifies with the projection ${\Sigma^{\infty}_{\T}}\BGL _{j,+} \ra {\Sigma^{\infty}_{\T}} {\overline \BGL }_{j,+}$ thereby completing the proof of the corollary in the motivic setting, when all
 the classifying spaces have been replaced by a fixed finite degree approximation, to order $m$. One may simply take the (homotopy) colimit as  $m \ra \infty$ to obtain the
  corresponding statement for the infinite classifying spaces. The proof of
the corresponding statement in the \'etale setting is similar, and is therefore skipped.
\end{proof}
\vskip .2cm
\begin{remark} In \cite{K}, a motivic variant of the Snaith splitting is worked out. However, our approach discussed above 
is quite different from the approach taken in \cite{K}, as we make strong use of the double coset formulae as in the 
Mitchell-Priddy paper: \cite{MP}.
\end{remark}

\end{document}